\documentclass[10pt,reqno,a4paper]{amsart}
\usepackage[british]{babel}
\usepackage{amssymb,amstext,amsthm,eucal, float, upgreek}
\usepackage{latexsym,amsmath,amsfonts,amscd,mathrsfs, booktabs}
\usepackage{enumerate}
\usepackage{graphicx}
\usepackage{caption}
\usepackage{subcaption}
\usepackage{array,color}
\usepackage{enumitem}
\usepackage{tgpagella}
\usepackage{tikz, bbm}
\usepackage[utf8]{inputenc}
\usepackage[margin=2.8cm]{geometry}
\usepackage[small]{eulervm}
\usepackage[unicode]{hyperref}
\usepackage{framed}
\usepackage{xcolor}
\usepackage[normalem]{ulem}
\usepackage[arrow, matrix, curve]{xy}
\usepackage{scalerel,stackengine}
\stackMath
\newcommand\widecheck[1]{%
\savestack{\tmpbox}{\stretchto{%
  \scaleto{%
    \scalerel*[\widthof{\ensuremath{#1}}]{\kern-.6pt\bigwedge\kern-.6pt}%
    {\rule[-\textheight/2]{1ex}{\textheight}}
  }{\textheight}%
}{0.5ex}}%
\stackon[1pt]{#1}{\scalebox{-1}{\tmpbox}}%
}
\usetikzlibrary{arrows,decorations.markings,positioning}
\tikzset{inner sep=0pt, node distance=5mm,
  root/.style={circle,draw,minimum size=5pt,thick},
  broot/.style={circle,draw,minimum size=5pt,thick,fill},
  xroot/.style={circle,draw,minimum size=5pt,thick,fill=gray!70!white},
  crossroot/.style={circle,draw,minimum size=5pt,thick,label=below:$\times$},
  doublearrow/.style={postaction={decorate},   decoration={markings,mark=at position .6 with {\arrow[line width=1.2pt]{>}}},double distance=1.6pt,thick},
  doublenoarrow/.style={double distance=1.6pt,thick},
  rdoublearrow/.style={postaction={decorate},   decoration={markings,mark=at position .4 with {\arrowreversed[line width=1.2pt]{>}}},double distance=1.6pt,thick},
  rtriplearrow/.style={postaction={decorate},   decoration={markings,mark=at position .4 with {\arrowreversed[line width=1.2pt]{>}}},double distance=2.5pt,thick},
  curvedline/.style={bend=right}
}

\usetikzlibrary{positioning}
\tikzset{main node/.style={rectangle,rounded corners,fill=blue!20,draw,minimum size=1cm,inner sep=0pt},
            }
\tikzset{sec node/.style={rectangle,rounded corners,fill=green!20,draw,minimum size=1cm,inner sep=0pt},
            }
\usetikzlibrary{shapes.misc}
\tikzset{cross/.style={cross out, draw=black, minimum size=2*(#1-\pgflinewidth), inner sep=0pt, outer sep=0pt},
cross/.default={3pt}}

\definecolor{mygray}{gray}{0.7}

\usepackage[draft]{todonotes}   

\usetikzlibrary{arrows,decorations.markings,positioning}
\tikzset{inner sep=0pt, node distance=5mm,
  root/.style={circle,draw,minimum size=5pt,thick},
  broot/.style={circle,draw,minimum size=5pt,thick,fill},
  xroot/.style={circle,draw,minimum size=5pt,thick,label=below:$\times$},
  doublearrow/.style={postaction={decorate},   decoration={markings,mark=at position .6 with {\arrow[line width=1.2pt]{>}}},double distance=1.6pt,thick},
  rdoublearrow/.style={postaction={decorate},   decoration={markings,mark=at position .4 with {\arrowreversed[line width=1.2pt]{>}}},double distance=1.6pt,thick},
	rtriplearrow/.style={postaction={decorate},   decoration={markings,mark=at position .4 with {\arrowreversed[line width=1.2pt]{>}}},double distance=2.5pt,thick},
	ltriplearrow/.style={postaction={decorate},   decoration={markings,mark=at position .6 with {\arrow[line width=1.2pt]{>}}},double distance=2.5pt,thick},
  curvedline/.style={bend=right}
}
\hypersetup{
  pdftitle   = {Exceptionally simple super-PDE for $F(4)${}},
  pdfkeywords = {},
  pdfauthor  = {Andrea Santi and Dennis The},
  pdfcreator = {\LaTeX\ with package \flqq hyperref\frqq}
}

\numberwithin{equation}{section}
\theoremstyle{plain}
\newtheorem{theorem}{Theorem}[section]
\newtheorem{proposition}[theorem]{Proposition}
\newtheorem{prop}[theorem]{Proposition}
\newtheorem{corollary}[theorem]{Corollary}
\newtheorem{cor}[theorem]{Corollary}
\newtheorem{lemma}[theorem]{Lemma}

\theoremstyle{definition}
\newtheorem{rem}[theorem]{Remark}

\newtheorem{definition}[theorem]{Definition}
\newcommand{\Hom}{\operatorname{Hom}}

\newcommand{\End}{\operatorname{End}}

\newcommand{\GL}{\operatorname{GL}}
\newcommand{\SO}{\operatorname{SO}}

\newcommand{\Spin}{\operatorname{Spin}}

\newcommand{\Cl}{C\ell}
\newcommand{\ad}{\operatorname{ad}}

\newcommand{\stab}{\mathfrak{stab}}
\newcommand{\der}{\mathfrak{der}}
\newcommand{\fosp}{\mathfrak{osp}}
\newcommand{\fgl}{\mathfrak{gl}}
\newcommand{\fsl}{\mathfrak{sl}}
\newcommand{\fso}{\mathfrak{so}}
\newcommand{\fsp}{\mathfrak{sp}}
\newcommand{\fspin}{\mathfrak{spin}}

\newcommand{\fm}{\mathfrak{m}}
\newcommand{\fp}{\mathfrak{p}}

\newcommand{\fg}{\mathfrak{g}}
\newcommand{\fh}{\mathfrak{h}}

\newcommand{\fX}{\mathfrak{X}}
\renewcommand{\1}{\mathbbm{1}}

\newcommand{\CC}{\mathbb{C}}

\newcommand{\be}{\boldsymbol{e}}

\newcommand{\vol}{\operatorname{vol}}
\newcommand{\gr}{\operatorname{gr}}
\newcommand{\bep}{\boldsymbol{\epsilon}}

\newcommand{\res}{\operatorname{res}}

\newcommand\diag{\operatorname{diag}}
\newenvironment{psm}
  {\left(\begin{smallmatrix}}
  {\end{smallmatrix}\right)}

 \newcommand\CO{\operatorname{CO}}

\newcommand\Stab{\operatorname{Stab}}
 \newcommand\CSpin{\operatorname{CSpin}}
 
 \newcommand\bbC{\mathbb{C}}
 \newcommand\bbP{\mathbb{P}}
 \newcommand\bbS{\mathbb{S}}
 \newcommand\bbZ{\mathbb{Z}}
 
 \newcommand\sfQ{\mathsf{Q}}
 \newcommand\sfZ{\mathsf{Z}}
 \newcommand\LG{\operatorname{LG}}
 
 \newcommand{\fz}{\mathfrak{z}}

 \newcommand\op{\oplus}

 \newcommand\fco{\mathfrak{co}}
 
 \newcommand\fcspo{\mathfrak{cspo}}
 
 \newcommand\id{\operatorname{id}}
 
 \newcommand{\fq}{\mathfrak{q}}
 \newcommand{\ff}{\mathfrak{f}}
 \newcommand\cU{\mathcal{U}}
 \newcommand\cV{\mathcal{V}}
 \newcommand{\fC}{\mathfrak{C}}
 
 \newcommand\bbA{\mathbb{A}}
 
 \newcommand\cC{\mathcal{C}}
 \newcommand\cD{\mathcal{D}}

 \newcommand\cL{\mathcal{L}}

 \newcommand\bX{\mathbf{X}}

 \newcommand\bS{\mathbf{S}}
 
 \newcommand\bU{\mathbf{U}}

 \newcommand{\Ch}{Ch}
 \renewcommand\ss{{\rm ss}}
 \newcommand\fms[1]{\fm_{\bar{#1}}}
 \newcommand\prn{\operatorname{pr}}
 
 \newcommand\finf{\mathfrak{inf}}

 \newcommand\bbV{\mathbb{V}}
 
 \newcommand\pr{\operatorname{pr}}

 \newcommand{\comm}[1]{}

\allowdisplaybreaks
\begin{document}
\title[Exceptionally simple super-PDE for $F(4)$]{Exceptionally simple super-PDE for $F(4)$}

\author{Andrea Santi}
\address{Andrea Santi, Dipartimento di Matematica, Universit\'a degli Studi di Roma ``Tor Vergata'',
Via della ricerca scientifica 1, 00133 Roma, ITALY}
\email{santi@mat.uniroma2.it, asanti.math@gmail.com}

\author{Dennis The}
\address{Dennis The, Department of Mathematics and Statistics, UiT The Arctic University of Norway, Troms\o{} 90-37, Norway}
\email{dennis.the@uit.no}

\thanks{}

 \begin{abstract}
 For the largest exceptional simple Lie superalgebra $F(4)$, having dimension $(24|16)$, we provide two explicit geometric realizations as supersymmetries, namely as the symmetry superalgebra of super-PDE systems of second and third order respectively.
 \end{abstract}

\maketitle

\section{Introduction}\label{S1}
\vskip0.3cm\par

In Kac's celebrated classification \cite{Kac}, the largest exceptional complex\footnote{Throughout the article, we exclusively work over $\bbC$, and use $\fsl(2)$, $\fso(7)$ to refer to $\fsl(2,\bbC)$, $\fso(7,\bbC)$, etc.} simple Lie superalgebra is $F(4)$ of dimension $(24|16)$ (this is not to be confused with the simple Lie algebra $F_4$ of dimension $52$). Similar to the other exceptional simple Lie superalgebras, this Lie superalgebra is traditionally described by introducing the brackets on its even and
odd components and not as the symmetry superalgebra of some simple algebraic or geometric
structure. One of the reasons is that its smallest non-trivial representation is the
adjoint representation \cite{SS}.  Our main goal is to establish the first explicit geometric realizations of $F(4)$ as the symmetry superalgebra of certain systems of super-PDE.  This paper can be regarded as accompanying \cite{KST2021}, where the analogous task was recently carried out for the exceptional Lie superalgebra $G(3)$.

 Our {exceptionally simple super-PDE} depend on an {\em even} scalar function $u$ of several independent variables $\{ x^0, x^1, ... \}$ and are explicitly given by:
 \begin{enumerate}
 \item a 2nd order system, with $x^0,x^1,x^2$ even, and $x^3,x^4$ odd:
  \begin{align} \label{E:F4-2PDE}
 \begin{split}
 u_{00} &= u_{22} (u_{12})^2 + 2 u_{12} u_{23} u_{24}, \\
 u_{01} &= \tfrac{1}{2} (u_{12})^2, \quad
 u_{02} = u_{22} u_{12} + u_{23} u_{24}, \quad
 u_{03} = u_{12} u_{23}, \quad
 u_{04} = u_{12} u_{24}, \\
 u_{11} &= 0, \quad
 u_{12} = -u_{34}, \quad
 u_{13} = 0, \quad 
 u_{14} = 0.
 \end{split}
 \end{align}
 \item a 3rd order system, with all $x^0,x^1,x^2,x^3$ odd:
 \begin{align} \label{E:F4-3PDE}
 u_{0ab} = u_{ab} u_{123}, \quad 1 \leq a < b \leq 3.
 \end{align}
 \end{enumerate}
Our convention for partial derivatives is that $u_{ij} = \partial_{x^i} u_j = \partial_{x^i} \partial_{x^j} u$, and we also recall that $u_{ij} = (-1)^{|i||j|} u_{ji}$, where $|i|,|j| \in \bbZ_2$ are the parities of $x^i$ and $x^j$. Our main result is:
 
 \begin{theorem} \label{T:main}
 The contact symmetry superalgebra of \eqref{E:F4-2PDE} or \eqref{E:F4-3PDE} is isomorphic to $F(4)$.
 \end{theorem}
  
 The system \eqref{E:F4-2PDE} may not appear to live up to its moniker, but let us express it parametrically.  Let $(\lambda_1,\lambda_2| \theta_1,\theta_2)$ be coordinates on the supermanifold $\bbC^{2|2}$, i.e., $\lambda_i$ are even while $\theta_i$ are odd, and let us group them together into
the symbol $T = (t^1,t^2 | t^3,t^4)= (\lambda_1,\lambda_2| \theta_1,\theta_2)$. We define the supersymmetric cubic form
 \begin{align*} 
 \fC(T^3) := \lambda_1 (\lambda_2)^2 + 2 \lambda_2\theta_1\theta_2,
 \end{align*}
and its derivatives $\fC_b(T^2) := \tfrac{1}{3} \partial_{t^b}(\fC(T^3))$ and 
 $\fC_{ab}(T) := \tfrac{1}{2} \partial_{t^a}(\fC_b(T^2))$.
 Then the system \eqref{E:F4-2PDE} takes the compact form
 \begin{align} \label{E:cubicPDE}
 \begin{pmatrix}
 u_{00} & u_{0b}\\
 u_{a0} & u_{ab}
 \end{pmatrix}
  = 
 \begin{pmatrix}
 \fC(T^3) & \frac{3}{2}\fC_b(T^2)\\
 \frac{3}{2}\fC_a(T^2) & 3\fC_{ab}(T)
 \end{pmatrix}\;,
 \end{align}
where $1\leq a,b\leq 4$.
 For varying choices of cubic form $\fC$, the formula \eqref{E:cubicPDE} uniformly yields geometric realizations of the simple Lie algebras $B_\ell, D_\ell, G_2, F_4, E_6, E_7, E_8$, as well as the exceptional simple Lie superalgebra $G(3)$ \cite{The2018, KST2021}.  We claim that this uniformity remarkably extends also to $F(4)$.
 
  Explicit expressions of the contact symmetries of the systems are also provided -- see Table \ref{F:sym} in the first case and Table \ref{F:FlatQctSym} in the second case.  More precisely, these are generating superfunctions (see \eqref{E:Sf}) for contact supervector fields that are symmetries of the super-PDE, and the Lie superalgebra structure is given by the Lagrange bracket (see \eqref{E:LB}), which is induced on the space of superfunctions from the Lie bracket of supervector fields.

 Underlying the super-PDE \eqref{E:F4-2PDE} and \eqref{E:F4-3PDE} are rich geometric stories that led to their discovery.  For \eqref{E:F4-2PDE}, the story begins with the first geometric realizations of $G_2$ in 1893 due to Cartan \cite{Cartan1893} and Engel \cite{Engel1893}.  From a modern viewpoint, Engel considered a contact 5-manifold $(M^5,\cC)$ endowed with a field of twisted cubics $\cV \subset \bbP(\cC)$, which defines a {\sl $G_2$-contact structure} $(M,\cC, \cV)$.  In general, this geometry has non-trivial local invariants, but the so-called {\it flat} model is maximally symmetric, has symmetry algebra $G_2$, is locally unique, and is equivalent to Engel's model.  Twisted cubics are projective {\em Legendrian varieties}, so their osculation yields affine tangent spaces $\widehat{T}_\ell \cV \subset \cC$ that are Lagrangian (i.e., maximally isotropic) for the natural conformal symplectic form on $\cC$.  Hence, the family $\widehat\cV = \{ \widehat{T}_\ell \cV : \ell \in \cV \}$ defines a submanifold of the total space of the Lagrange--Grassmann bundle $\pi:\widetilde{M}=LG(\cC)\rightarrow M$ consisting of all Lagrangian subspaces of $(M,\cC)$. Now, $(M,\cC)$ and $(\widetilde{M},\widetilde\cC)$ are locally equivalent to the first and second order jet-spaces $J^1(\bbC^2,\bbC)$ and $J^2(\bbC^2,\bbC)$ with their Cartan distributions, respectively, so $\widehat\cV \subset \widetilde{M}$ is a coordinate-independent realization of a 2nd order PDE, which inherits all the symmetries of $\cV$. Below, we provide a discussion in Lie-theoretic terms  on why the symmetry algebra has in fact dimension no larger than $\dim (G_2)$.

Carrying this out for the flat $G_2$-contact structure yields Cartan's realization of $G_2$ as the algebra of contact symmetries of
 \begin{align}
 u_{xx} = \tfrac{1}{3} u_{yy}^3, \quad 
 u_{xy} = \tfrac{1}{2} u_{yy}^2.
 \end{align}
This is \eqref{E:cubicPDE} specialized to $T = t \in \bbC$ and $\fC(T^3) = \frac{1}{3} t^3$.  
 
 This geometric construction generalizes to all complex simple Lie algebras that are not of type $A$ or $C$ \cite{The2018}.  The adjoint variety $M = G/P$ is a parabolic homogeneous quotient, which is a complex contact manifold whose contact distribution $\cC$ admits a $G$-invariant field of Legendrian varieties $\cV \subset \bbP(\cC)$, and the triple $(M,\cC,\cV)$ is the {\sl  flat $G$-contact structure}.  We may osculate in each fibre, obtaining the 2nd order PDE $\widehat\cV \subset \widetilde{M}$, which is $G$-invariant by naturality of the construction, and \eqref{E:cubicPDE} is its local description.  (The cubic form $\fC$ is uncovered from repeated osculations of $\cV$ at the chosen basepoint, and $\cV$ is locally described in terms of it.) In the recent paper \cite{KST2021}, a careful analysis at the Lie superalgebra cohomological level and the osculation construction in terms of the functor of points were applied to find the {\em $G(3)$-contact super-PDE system} with symmetry superalgebra $G(3)$:
 \begin{equation}\label{SC-PDE}
 \begin{split}
u_{xx} &= \tfrac13 u_{yy}^3 + 2 u_{yy} u_{y\nu} u_{y\tau}, \quad
u_{xy} = \tfrac12 u_{yy}^2 + u_{y\nu} u_{y\tau}, \\
u_{x\nu}&= u_{yy} u_{y\nu}, \qquad\ u_{x\tau} = u_{yy} u_{y\tau}, \qquad\ u_{\nu\tau} = -u_{yy},
 \end{split}
 \end{equation}
where $u=u(x,y|\nu,\tau):\mathbb C^{2|2}\to \mathbb C^{1|0}$.

The algebraic starting point in all these cases is a grading on the Lie superalgebra $\fg$ of contact type: we have $\fg = \fg_{-2} \op \ldots \op \fg_2$ as a $\bbZ$-graded Lie superalgebra whose symbol $\fm = \fg_-$ is the Heisenberg superalgebra (i.e., $\dim(\fg_{-2}) = (1|0)$ and the super-skewsymmetric bracket $\eta: \Lambda^2 \fg_{-1} \to \fg_{-2}$ is non-degenerate).  The superalgebra $\der_{{\rm gr}}(\fm)$ of zero-degree derivations of $\fm$ is isomorphic to the conformal symplectic-orthogonal superalgebra $\mathfrak{cspo}(\fg_{-1})$ and $\fg_0$ is a subalgebra.  The associated parabolic subalgebra is $\fp = \fg_{\geq 0}$.  More generally, for an arbitrary $\bbZ$-graded Lie superalgebra $\fg = \fg_- \op \fg_0 \op \fg_+$, the vanishing of the Spencer cohomology groups $H^{d,1}(\fm,\fg)$ in degrees $d > 0$ is equivalent to $\fg \cong \prn(\fm,\fg_0)$, where $\prn(\fm,\fg_0)$ is the so-called {\sl Tanaka--Weisfeiler prolongation}, which is the maximal effective graded Lie superalgebra extending $\fm$ and $\fg_0 \subset \der_{{\rm gr}}(\fm)$. In the classical setting, Kostant's version of the Bott--Borel--Weil theorem \cite{Kos} is available  to efficiently assess this vanishing, and the fact that $\fg \cong \prn(\fm,\fg_0)$ implies that the symmetry algebra of the geometric structure under consideration has dimension bounded by $\dim(\fg)$.  In this setting, both $\cV$ and the osculation $\widehat\cV$ of $\cV$ described above fibrewise reduce the structure algebra to $\fg_0$, and the vanishing of the above Spencer groups leads to $(M,\cC,\cV)$ and $(M,\cC,\widehat\cV)$ having the same contact symmetry algebra with upper bound $\dim(\fg)$.
 
 In the super-setting, the analogous vanishing of the Spencer cohomology groups equivalently characterizes $\fg \cong \prn(\fm,\fg_0)$ but Kostant's result does not hold in general. Cohomology groups have been known for some irreducible (i.e, depth $1$) supergeometries \cite{LPS} and some distinguished Borel subalgebras \cite{Coul}, but, for instance, not for the parabolic subalgebras of depth $2$ that we consider in this paper.  In \cite[Thm.1.1]{KST2022}, we showed that the symmetry superalgebra of an associated supergeometry modelled on the data $(\fm,\fg_0)$ is bounded by $\dim \prn(\fm,\fg_0)$ in the strong sense: the even and odd parts are respectively bounded by those of the prolongation.
 
 Turning now to $\fg = F(4)$, there are two inequivalent contact gradings:
 \begin{enumerate}
 \item[$(i)$] the {\sl mixed contact grading}, having $\fg_0 \cong \mathfrak{cosp}(4|2;\alpha) \cong \bbC \op \mathfrak{osp}(4|2;\alpha)$ with $\alpha = 2$, and $\fg_{-1}$ given by the unique non-trivial $\fg_0$-irreducible representation of dimension $(6|4)$.
 \item[$(ii)$] the {\sl odd contact grading}, having $\fg_0 \cong \fco(7)$ and $\fg_{-1} \cong \bbS$ given by the 8-dimensional spin representation of $\fso(7)$.  In this case, $\fg_0$ is purely even and $\fg_{-1}$ is purely odd.
 \end{enumerate}
 
 In both cases, we show (see Theorems \ref{thm:235H^1} and \ref{sec:Spencer}) that $H^{d,1}(\fm,\fg) = 0$ for $d > 0$, and hence $\fg \cong \prn(\fm,\fg_0)$.  This has two implications.  First, the geometric structures under consideration are encoded by a contact supermanifold equipped with a structure superalgebra reduced to $\fg_0 \subset \der_{{\rm gr}}(\fm) \cong \mathfrak{cspo}(\fg_{-1})$.  For the mixed contact case, a field of certain supervarieties $\cV \subset \bbP(\cC)$ realizes this reduction; for the odd contact case, it is accomplished by a conformal supersymmetric quartic tensor $[\sfQ] \in \Gamma(\odot^4 \cC^*)$. Secondly, the maximal symmetry dimension associated with either of these structures is $\dim(\fg) =\dim (F(4))=(24|16)$ and, in each case, there is a flat structure that realizes the symmetry upper bound.  In this article, we do not present our results on the second Spencer cohomology groups $H^2(\fm,\fg)$,  since they are quite involved and ultimately will deserve a separate work in the framework of deformations of mixed-contact and odd-contact $F(4)$-supergeometries.
 
 In \S\ref{S:mixedct}, for the mixed contact case, we carry out the osculation described above on the flat structure $(M^{7|4},\cC,\cV \subset \bbP(\cC))$ to obtain the 2nd order super-PDE system \eqref{E:F4-2PDE} and establish $\fg = F(4)$ as its contact symmetry superalgebra. 
 
  In \S\ref{S:oddct}, we consider the odd contact case, which is completely different from the aforementioned osculation story.  Since $\fg_{-1} \cong \bbS$ is purely odd, the super-skew form $\eta$ is symmetric in the classical sense.  The natural variety present here is the null quadric $\cV = \{ \eta = 0 \} \subset \bbP(\bbS)$, but this is also an orbit of $\CO(\bbS) = \CO(8)$, so $\cV$ does {\it not} reduce the structure algebra from $\der_{\rm gr}(\fm)\cong\mathfrak{co}(8)$ to $\fg_0 \cong \fco(7)$.  Instead, the reduction is provided by the conformal class $[\sfQ]$ of a supersymmetric quartic tensor $\sfQ$, which, fixing an adapted Witt frame of odd supervector fields on $\cC$, is the classical Cayley skew $4$-form (see \eqref{E:Q} for its explicit expression). It is remarkable that $\sfQ$ arises also as the supersymmetric counterpart of Freudenthal's quartic invariant \cite[(4.8)]{Fre1954}, \cite[\S 4.11]{Fre1964}, or rather its realization in terms of a contact grading \cite[p.155-156]{LM2002}, \cite{Hel}; see Remark \ref{R:Cayley}.  Recently, the conformal class $[\sfQ]$ of this quartic was used to realize exceptional simple Lie algebras as symmetries of geometric structures on contact manifolds and 2nd order PDE \cite{LNS, The2018, AGMM}.  In our odd contact $F(4)$ setting, $[\sfQ]$ is related to a 3rd order super-PDE, a phenomenon that has no parallel in our study \cite{KST2021} of $G(3)$.	
	
 Let us sketch here how to connect $[\sfQ]$ to a 3rd order super-PDE. Under the reduction to $\fg_0 \cong \fco(7)$, there is for any point $\ell \in \cV$ a distinguished $\eta$-Lagrangian subspace $L_\ell \cV \subset \widehat{T}_\ell \cV$.  This is a {\em self-dual} plane w.r.t. $\eta$ (contrasted to $\sfQ$ which, in our conventions, is anti self-dual), and this property clearly persists to the conformal class $[\eta]$.  The family $L(\cV) := \{ L_\ell \cV : \ell \in \cV \}$ of all such {\sl Lagrangian tangent planes} is an {\em open} $\SO(\bbS)$-orbit in $LG(\bbS)$, and again insufficient to enforce a reduction to $\fg_0 \cong \fco(7)$.  Instead, it is the flag manifold $F(\cV):=\{(\ell,L): \ell\in\cV, L = L_\ell\cV\}$ with its incidence condition $\ell \subset L_\ell \cV$ that encodes the $\fg_0$-reduction, with $[\sfQ]$ canonically determined from $F(\cV)$ and conversely (see Proposition \ref{P:VLV} for the precise statement).
  
 Starting from the geometric data $(M^{1|8},\cC,[\sfQ])$, we may consider the
{\sl incidence Lagrange-Grassmann bundle} $\widetilde{M}^o$  determined fibrewise by $F(\cV)$.
Since the map $\ell \mapsto L_\ell \cV$ is injective, then $\widetilde{M}^o$ is diffeomorphic to 
the self-dual Lagrange--Grassmann bundle $\widetilde{M}_+=LG_+(\cC)$. The Cartan superdistribution $\widetilde\cC$ of rank $(6|4)$ is induced on $\widetilde{M}^o$ as usual (tautologically  via the pullback of the self-dual Lagrangian subspaces $L_\ell \cV$ of $(M,\cC)$) but, in addition, the pullback of the corresponding lines $\ell \subset L_\ell \cV$ distinguishes an extra subsuperdistribution $\widetilde\cC^o \subset \widetilde\cC$ of rank $(6|1)$.  Slightly finer is a $(3|1)$-subsuperdistribution $\cD \subset \widetilde\cC^o$ canonically determined by an appropriate tensorial condition, and we also establish that $\cD = \cD_{\bar{0}} \op \cD_{\bar{1}}$ has geometrically distinguished even and odd parts.  The passage from $(M,\cC,[\sfQ])$ to $(\widetilde{M}^o,\widetilde\cC,\cD)$ is akin to ``lifting a geometric structure to a correspondence space'' familiar in twistor theory.
 
\sloppy We can now motivate the connection to super-PDE.  Locally, $(M,\cC) \cong (J^1(\bbC^{0|4},\bbC^{1|0}),\cC)$, and $(\widetilde{M}^o,\widetilde\cC) \cong (J^2(\bbC^{0|4},\bbC^{1|0}),\widetilde\cC)$, with the latter additionally equipped with the superdistribution $\cD=\cD_{\bar 0}\oplus\cD_{\bar 1}$.  The geometric construction of the third order jet space $(J^3(\bbC^{0|4},\bbC^{1|0}),\widecheck\cC)$ consists of the isotropic subspaces of $\widetilde\cC$ of rank $(0|4)$ that are complementary to the vertical superdistribution of $\pi:\widetilde{M}^o\cong\widetilde M_+=LG_+(\cC) \to M$, and we distinguish in there the subsupermanifold $\Sigma$ of those that contain $\cD_{\bar{1}}$.  This $\Sigma$ is a distinguished 3rd order super-PDE, and carrying this out for the flat structure $(M,\cC,[\sfQ])$ yields the remarkable system \eqref{E:F4-3PDE}.

The structure of the paper is as follows. In \S \ref{sec:alg} we recall the basics of $F(4)$,
its parabolic subalgebras and $\mathbb Z$-gradings, give a spinorial description of the odd contact grading of $F(4)$ akin to those used in supergravity theories and finish with some combinatorial results that are crucial for our cohomological computations. Then \S \ref{sec:SCG} is devoted to Spencer cohomology, an important ingredient in the proof that $F(4)$
is the symmetry superalgebra of the two differential equations \eqref{E:F4-2PDE} and \eqref{E:F4-3PDE} of the paper. The main results are  Theorems \ref{thm:235H^1} and \ref{sec:Spencer}, while some straightforward computations are postponed to Appendix \ref{S:Spencer-odd-contact-grading}. A short resum\'e of jet-superspaces, including contact vector fields and generating superfunctions, is provided in \S \ref{S:J1J2} for the reader's convenience.  The last two sections are devoted to deriving \eqref{E:F4-2PDE} and \eqref{E:F4-3PDE}: first we introduce the mixed-contact $F(4)$-supergeometries and consider the corresponding osculation construction using the functor of points in \S\ref{S:mixedct} and then focus on odd-contact $F(4)$-supergeometries and the incidence Lagrange--Grassmann bundle in  \S\ref{S:oddct}. The main results here are Theorem  \ref{T:symbound-mixed} and \ref{thm:areyoulookingforalabel} in \S\ref{S:mixedct} and Theorems \ref{T:symbound}, \ref{T:flatF4ct}, \ref{T:F4-5gr}, \ref{T:3PDE} in \S\ref{S:oddct}, together with the explicit description of the solution superspace of \eqref{E:F4-3PDE}.

\section{Algebraic aspects and parabolic subalgebras of $F(4)$}
\label{sec:alg}
\subsection{Root systems for $F(4)$}

The complex simple Lie superalgebra (shortly, LSA) $\fg = F(4)$ has dimension $(24|16)$, with even and odd parts given by
 \[
 \fg_{\bar 0} = B_3 \oplus A_1, \quad \fg_{\bar 1} = \mathbb{S}\boxtimes \bbC^2 .
 \]
 Here, $\fg_{\bar 0}$ is a direct sum of the complex Lie algebras $B_3 := \fso(7)$ and $A_1 := \fsp(2)$, while 
 $\fg_{\bar 1}$ (as a $\fg_{\bar 0}$-module) is the exterior tensor product of the 8-dimensional spin representation $\mathbb{S}$ of $B_3$ and the standard representation of $A_1$. 
The somewhat unusual notation $\fsp(2)$ will be reserved specifically to the
ideal $A_1$ of $\fg_{\bar 0}$
throughout the whole paper, to avoid confusion with other $\fsl(2)$-subalgebras.
 
 We fix a Cartan subalgebra $\fh$ of $\fg$, which by definition is a Cartan subalgebra for $\fg_{\bar 0}$, so all are conjugate.  Adopting the root conventions of \cite[\S 2.19]{FSS}, we fix functionals $\epsilon_i, \delta \in \fh^*$ with scalar products
 $\langle \epsilon_i, \epsilon_j \rangle = \delta_{ij}$,\,\,
 $\langle \epsilon_i, \delta \rangle = 0$,\,\,
 $\langle \delta, \delta \rangle = -3$.  (This scalar product is induced from the Killing form on $F(4)$, which is non-degenerate.)  The $F(4)$ root system $\Delta = \Delta_{\bar 0} \cup \Delta_{\bar 1} \subset \fh^* \backslash \{ 0 \}$ splits into even and odd roots given by
 \begin{align*}
 \Delta_{\bar{0}}: &\quad \pm \delta,\quad \pm \epsilon_1,\quad \pm \epsilon_2,\quad \pm \epsilon_3,\quad \pm (\epsilon_1 \pm \epsilon_2),\quad \pm (\epsilon_1 \pm \epsilon_3),\quad  \pm (\epsilon_2 \pm \epsilon_3)\\
 \Delta_{\bar{1}}: &\quad \frac{1}{2} (\pm \delta \pm \epsilon_1 \pm \epsilon_2 \pm \epsilon_3)
 \end{align*}
 The Weyl group is generated by reflections $S_\alpha$ w.r.t. even roots $\alpha \in \Delta_{\bar 0}$.  Here, $S_\alpha(\beta) = \beta - \frac{2\langle \beta, \alpha \rangle}{\langle \alpha, \alpha \rangle} \alpha$ is the usual reflection formula for any $\alpha \in \Delta_{\bar 0}$, and $S_\alpha$ preserves both $\Delta_{\bar 0}$ and $\Delta_{\bar 1}$. 
 
 Given a simple root system $\Pi$ and a simple root $\alpha \in \Pi$ that is odd and isotropic, i.e., $\alpha \in \Delta_{\bar 1}$ and $\langle \alpha, \alpha \rangle = 0$, we define (see \cite{LSS, Ser2011}) an {\sl odd reflection} by
 \begin{align*}
 S_\alpha(\beta) = \begin{cases}
 \beta + \alpha, & \langle \beta,\alpha \rangle \neq 0,\\
 \beta, & \langle \beta,\alpha \rangle = 0; \,\, \beta \neq \alpha,\\
 -\alpha, & \beta = \alpha,
 \end{cases}
 \end{align*}
where $\beta\in\Pi$.
 Any odd reflection maps $\Pi$ to another  simple root system that is inequivalent under the action of the Weyl group; moreover all possible inequivalent simple root systems can be generated from successively applying odd reflections.  These are shown in Table \ref{F:F4-SRS}, along with corresponding Dynkin diagrams, Dynkin labels indicating the coefficients of the highest root (w.r.t. the given simple root system), and odd reflections indicated with red dashed lines.  All positive roots are listed in Table \ref{F:F4-PosSys}.
 \begin{footnotesize}
 \begin{table}[h]
 \begin{center}
 \[
 \begin{array}{|c|c|l|} \hline
 \mbox{Label} & \mbox{Dynkin diagram} & \multicolumn{1}{c|}{\mbox{Simple roots}}\\ \hline\hline
 I & \begin{array}{c}
 \begin{tikzpicture}[remember picture]
 \draw (0,0) -- (1,0);
 \draw (1,0.05) -- (2,0.05);
 \draw (1,-0.05) -- (2,-0.05);
 \draw (2,0) -- (3,0);
 \draw (1.6,0.15) -- (1.4,0) -- (1.6,-0.15);
 \node[draw,circle,inner sep=2pt,fill=mygray] (A1) at (0,0) {};
 \node[draw,circle,inner sep=2pt,fill=white] at (1,0) {};
 \node[draw,circle,inner sep=2pt,fill=white] at (2,0) {};
 \node[draw,circle,inner sep=2pt,fill=white] at (3,0) {};
 \node[above] at (0,0.25) {2};
 \node[above] at (1,0.25) {3};
 \node[above] at (2,0.25) {2};
 \node[above] at (3,0.25) {1};
 \node[below] at (0,-0.25) {$\alpha_1$};
 \node[below] at (1,-0.25) {$\alpha_2$};
 \node[below] at (2,-0.25) {$\alpha_3$};
 \node[below] at (3,-0.25) {$\alpha_4$};
 \end{tikzpicture}
 \end{array} & \begin{array}{l}
 \alpha_1 = \frac{1}{2}(\delta - \epsilon_1 - \epsilon_2 - \epsilon_3)\\
 \alpha_2 = \epsilon_3\\
 \alpha_3 = \epsilon_2 - \epsilon_3\\
 \alpha_4 = \epsilon_1 - \epsilon_2
 \end{array} \\ \hline
 II & \begin{array}{c}
 \begin{tikzpicture}[remember picture]
 \draw (0,0) -- (1,0);
 \draw (1,0.05) -- (2,0.05);
 \draw (1,-0.05) -- (2,-0.05);
 \draw (2,0) -- (3,0);
 \draw (1.6,0.15) -- (1.4,0) -- (1.6,-0.15);
 \node[draw,circle,inner sep=2pt,fill=mygray] (B1) at (0,0) {};
 \node[draw,circle,inner sep=2pt,fill=mygray] (B2) at (1,0) {};
 \node[draw,circle,inner sep=2pt,fill=white] at (2,0) {};
 \node[draw,circle,inner sep=2pt,fill=white] at (3,0) {};
 \node[above] at (0,0.25) {2};
 \node[above] at (1,0.25) {3};
 \node[above] at (2,0.25) {2};
 \node[above] at (3,0.25) {1};
 \node[below] at (0,-0.25) {$\alpha_1$};
 \node[below] at (1,-0.25) {$\alpha_2$};
 \node[below] at (2,-0.25) {$\alpha_3$};
 \node[below] at (3,-0.25) {$\alpha_4$};
 \end{tikzpicture}
 \end{array} & 
 \begin{array}{l}
 \alpha_1 = \frac{1}{2}(-\delta + \epsilon_1 + \epsilon_2 + \epsilon_3)\\
 \alpha_2 = \frac{1}{2}(\delta - \epsilon_1 - \epsilon_2 + \epsilon_3)\\
 \alpha_3 = \epsilon_2 - \epsilon_3\\
 \alpha_4 = \epsilon_1 - \epsilon_2
 \end{array}\\ \hline
  III & \begin{array}{c}
 \begin{tikzpicture}[remember picture]
 \draw (0,0.05) -- (1,0.05);
 \draw (0,-0.05) -- (1,-0.05);
 \draw (1,0.05) -- (2,0.05);
 \draw (1,-0.05) -- (2,-0.05);
 \draw (0.4,0.15) -- (0.6,0) -- (0.4,-0.15);
 \draw (1,0) -- (1.5,-1);
 \draw (2,0) -- (1.5,-1);
 \node[draw,circle,inner sep=2pt,fill=mygray] (C2) at (1,0) {};
 \node[draw,circle,inner sep=2pt,fill=mygray] (C3) at (2,0) {};
 \node[draw,circle,inner sep=2pt,fill=white] at (0,0) {};
 \node[draw,circle,inner sep=2pt,fill=white] at (1.5,-1) {};
 \node[above] at (0,0.25) {1};
 \node[above] at (1,0.25) {2};
 \node[above] at (2,0.25) {2};
 \node[above] at (1.5,-0.75) {2};
 \node[below] at (0,-0.25) {$\alpha_1$};
 \node[below] at (0.95,-0.25) {$\alpha_2$};
 \node[below] at (2.05,-0.25) {$\alpha_3$};
 \node[right] at (1.7,-1) {$\alpha_4$};
 \end{tikzpicture}
 \end{array} & 
  \begin{array}{l}
 \alpha_1 = \epsilon_1 - \epsilon_2\\
 \alpha_2 = \frac{1}{2}(\delta - \epsilon_1 + \epsilon_2 - \epsilon_3)\\
 \alpha_3 = \frac{1}{2}(-\delta + \epsilon_1 + \epsilon_2 - \epsilon_3)\\
 \alpha_4 = \epsilon_3
 \end{array}\\ \hline
  IV & \begin{array}{c}
 \begin{tikzpicture}[remember picture]
 \draw (0,-0.5) -- (1,0);
 \draw (0,0.55) -- (1,0.05);
 \draw (0,0.45) -- (1,-0.05);
 \draw (0,0.5) -- (0,-0.5);
 \draw (0.05,0.5) -- (0.05,-0.5);
 \draw (-0.05,0.5) -- (-0.05,-0.5);
 \draw (1,0.05) -- (2,0.05);
 \draw (1,-0.05) -- (2,-0.05);
 \draw (1.6,0.15) -- (1.4,0) -- (1.6,-0.15);
 \node[draw,circle,inner sep=2pt,fill=mygray] (D1) at (0,0.5) {};
 \node[draw,circle,inner sep=2pt,fill=mygray] (D2) at (0,-0.5) {};
 \node[draw,circle,inner sep=2pt,fill=mygray] (D3) at (1,0) {};
 \node[draw,circle,inner sep=2pt,fill=white] at (2,0) {};
 \node[above] at (0,0.7) {1};
 \node[above] at (0.2,-0.3) {2};
 \node[above] at (1,0.25) {3};
 \node[above] at (2,0.25) {2};
 \node[left] at (-0.2,0.5) {$\alpha_1$};
 \node[left] at (-0.2,-0.5) {$\alpha_2$};
 \node[below] at (1,-0.25) {$\alpha_3$};
 \node[below] at (2,-0.25) {$\alpha_4$};
 \end{tikzpicture}
 \end{array} & 
 \begin{array}{l}
 \alpha_1 = \frac{1}{2}(\delta + \epsilon_1 - \epsilon_2 - \epsilon_3)\\
 \alpha_2 = \frac{1}{2}(\delta - \epsilon_1 + \epsilon_2 + \epsilon_3)\\
 \alpha_3 = \frac{1}{2}(-\delta + \epsilon_1 - \epsilon_2 + \epsilon_3)\\
 \alpha_4 = \epsilon_2 - \epsilon_3
 \end{array}\\ \hline
 V & \begin{array}{c}
 \begin{tikzpicture}[remember picture]
 \draw (0,0) -- (1,0);
 \draw (0,0.05) -- (1,0.05);
 \draw (0,-0.05) -- (1,-0.05);
 \draw (0.4,0.15) -- (0.6,0) -- (0.4,-0.15);
 \draw (2,0.05) -- (3,0.05);
 \draw (2,-0.05) -- (3,-0.05);
 \draw (1,0) -- (2,0);
 \draw (2.6,0.15) -- (2.4,0) -- (2.6,-0.15);
 \node[draw,circle,inner sep=2pt,fill=white] at (0,0) {};
 \node[draw,circle,inner sep=2pt,fill=mygray] (E2) at (1,0) {};
 \node[draw,circle,inner sep=2pt,fill=white] at (2,0) {};
 \node[draw,circle,inner sep=2pt,fill=white] at (3,0) {};
 \node[above] at (0,0.25) {1};
 \node[above] at (1,0.25) {2};
 \node[above] at (2,0.25) {3};
 \node[above] at (3,0.25) {2};
 \node[below] at (0,-0.25) {$\alpha_1$};
 \node[below] at (1,-0.25) {$\alpha_2$};
 \node[below] at (2,-0.25) {$\alpha_3$};
 \node[below] at (3,-0.25) {$\alpha_4$};
 \end{tikzpicture}
 \end{array} & \begin{array}{l}
 \alpha_1 = \delta\\
 \alpha_2 = \frac{1}{2}(-\delta + \epsilon_1 - \epsilon_2 - \epsilon_3)\\
 \alpha_3 = \epsilon_3\\
 \alpha_4 = \epsilon_2 - \epsilon_3
 \end{array} \\ \hline
VI & \begin{array}{c}
 \begin{tikzpicture}[remember picture]
 \draw (0,0) -- (1,0);
 \draw (0,0.05) -- (1,0.05);
 \draw (0,-0.05) -- (1,-0.05);
 \draw (0.4,0.15) -- (0.6,0) -- (0.4,-0.15);
 \draw (1,0.05) -- (2,0.05);
 \draw (1,-0.05) -- (2,-0.05);
 \draw (1.6,0.15) -- (1.4,0) -- (1.6,-0.15);
 \draw (2,0) -- (3,0);
 \node[draw,circle,inner sep=2pt,fill=white] at (0,0) {};
 \node[draw,circle,inner sep=2pt,fill=mygray] (F2) at (1,0) {};
 \node[draw,circle,inner sep=2pt,fill=white] at (2,0) {};
 \node[draw,circle,inner sep=2pt,fill=white] at (3,0) {};
 \node[above] at (0,0.25) {2};
 \node[above] at (1,0.25) {4};
 \node[above] at (2,0.25) {3};
 \node[above] at (3,0.25) {2};
 \node[below] at (0,-0.25) {$\alpha_1$};
 \node[below] at (1,-0.25) {$\alpha_2$};
 \node[below] at (2,-0.25) {$\alpha_3$};
 \node[below] at (3,-0.25) {$\alpha_4$};
 \end{tikzpicture}
 \end{array} & \begin{array}{l}
 \alpha_1 = \delta\\
 \alpha_2 = \frac{1}{2}(-\delta - \epsilon_1 + \epsilon_2 + \epsilon_3)\\
 \alpha_3 = \epsilon_1 - \epsilon_2\\
 \alpha_4 = \epsilon_2 - \epsilon_3
 \end{array} \\ \hline
 \end{array}
 \]
 \begin{tikzpicture}[overlay, red, remember picture]
\draw[<->, ultra thick, densely dotted] (A1) to[in=140, out = -140] (B1);
\draw[<->, ultra thick, densely dotted] (B2) to[in=150, out = -20] (C3);
\draw[<->, ultra thick, densely dotted] (C2) to[in=130, out = -130] (D3);
\draw[<->, ultra thick, densely dotted] (D2) to[in=140, out = -110] (E2);
\draw[<->, ultra thick, densely dotted] (D1) to[in=170, out = -150] (F2);
 \end{tikzpicture}
 \end{center}
 \caption{Inequivalent simple root systems for $F(4)$}
 \label{F:F4-SRS}
 \end{table}

\begin{table}[h]
  \[
 \begin{array}{|c|l|l|} \hline
 & \multicolumn{1}{c|}{\Delta_{\bar{0}}^+} & \multicolumn{1}{c|}{\Delta_{\bar{1}}^+}\\ \hline
 I & 
 \begin{array}{l}
\alpha_2,\ \alpha_3,\ \alpha_4,\ \alpha_2+\alpha_3,\ \alpha_3+\alpha_4,\\ 
2\alpha_2+\alpha_3,\ \alpha_2+\alpha_3+\alpha_4,\\ 
2\alpha_2+\alpha_3+\alpha_4,\ 2\alpha_2+2\alpha_3+\alpha_4,\\ 
2\alpha_1+3\alpha_2+2\alpha_3+\alpha_4
 \end{array}  & 
 \begin{array}{l}
\alpha_1,\ \alpha_1+\alpha_2,\ \alpha_1+\alpha_2+\alpha_3,\\
\alpha_1+2\alpha_2+\alpha_3,\ \alpha_1+\alpha_2+\alpha_3+\alpha_4,\\
\alpha_1+2\alpha_2+\alpha_3+\alpha_4,\ \alpha_1+2\alpha_2+2\alpha_3+\alpha_4,\\
\alpha_1+3\alpha_2+2\alpha_3+\alpha_4
 \end{array}  \\ \hline
 
 II & 
 \begin{array}{l}
\alpha_3,\ \alpha_4,\ \alpha_1+\alpha_2,\ \alpha_3+\alpha_4,\ \alpha_1+\alpha_2+\alpha_3,\\
2\alpha_1+2\alpha_2+\alpha_3,\ \alpha_1+\alpha_2+\alpha_3+\alpha_4,\\
2\alpha_1+2\alpha_2+\alpha_3+\alpha_4,\ 2\alpha_1+2\alpha_2+2\alpha_3+\alpha_4,\\
\alpha_1+3\alpha_2+2\alpha_3+\alpha_4
 \end{array} & 
 \begin{array}{l}
\alpha_1,\ \alpha_2,\ \alpha_2+\alpha_3,\\
\alpha_2+\alpha_3+\alpha_4,\ \alpha_1+2\alpha_2+\alpha_3,\\
\alpha_1+2\alpha_2+\alpha_3+\alpha_4,\ \alpha_1+2\alpha_2+2\alpha_3+\alpha_4,\\
2\alpha_1+3\alpha_2+2\alpha_3+\alpha_4
 \end{array} \\ \hline
 
 III & 
 \begin{array}{l}
\alpha_1,\ \alpha_4,\ \alpha_2+\alpha_3,\ \alpha_1+\alpha_2+\alpha_3,\\
\alpha_2+\alpha_3+\alpha_4,\ \alpha_1+2\alpha_2+\alpha_4,\ \alpha_2+\alpha_3+2\alpha_4,\\
\alpha_1+\alpha_2+\alpha_3+\alpha_4,\ \alpha_1+\alpha_2+\alpha_3+2\alpha_4,\\
\alpha_1+2\alpha_2+2\alpha_3+2\alpha_4
 \end{array} & 
 \begin{array}{l}
\alpha_2,\ \alpha_3,\\ 
\alpha_1+\alpha_2,\ \alpha_2+\alpha_4,\ \alpha_3+\alpha_4,\\
\alpha_1+\alpha_2+\alpha_4,\ \alpha_1+2\alpha_2+\alpha_3+\alpha_4,\\
\alpha_1+2\alpha_2+\alpha_3+2\alpha_4 
 \end{array} \\ \hline

 IV & 
 \begin{array}{l}
\alpha_4,\ \alpha_1+\alpha_2,\ \alpha_1+\alpha_3,\ \alpha_2+\alpha_4,\\
\alpha_1+\alpha_3+\alpha_4,\ \alpha_2+\alpha_3+\alpha_4,\ 2\alpha_2+2\alpha_3+\alpha_4,\\
\alpha_1+\alpha_2+2\alpha_3+\alpha_4,\ \alpha_1+2\alpha_2+3\alpha_3+\alpha_4,\\
\alpha_1+2\alpha_2+3\alpha_3+2\alpha_4
 \end{array} & 
 \begin{array}{l}
\alpha_1,\ \alpha_2,\ \alpha_3,\\
\alpha_3+\alpha_4,\ \alpha_1+\alpha_2+\alpha_3,\\
\alpha_2+2\alpha_3+\alpha_4,\ \alpha_1+\alpha_2+\alpha_3+\alpha_4,\\
\alpha_1+2\alpha_2+2\alpha_3+\alpha_4 
 \end{array} \\ \hline
 
 V & \begin{array}{l}
\alpha_1,\ \alpha_3,\ \alpha_4,\ \alpha_3+\alpha_4,\ 2\alpha_3+\alpha_4,\\
\alpha_1+2\alpha_2+\alpha_3,\ \alpha_1+2\alpha_2+\alpha_3+\alpha_4,\\
\alpha_1+2\alpha_2+2\alpha_3+\alpha_4,\ \alpha_1+2\alpha_2+3\alpha_3+\alpha_4,\\
\alpha_1+2\alpha_2+3\alpha_3+2\alpha_4
 \end{array} & 
 \begin{array}{l}
\alpha_2,\ \alpha_1+\alpha_2,\ \alpha_2+\alpha_3,\\
\alpha_1+\alpha_2+\alpha_3,\ \alpha_2+\alpha_3+\alpha_4,\\
\alpha_2+2\alpha_3+\alpha_4,\ \alpha_1+\alpha_2+\alpha_3+\alpha_4,\\
\alpha_1+\alpha_2+2\alpha_3+\alpha_4
 \end{array} \\ \hline
 
 VI & 
 \begin{array}{l}
\alpha_1,\ \alpha_3,\ \alpha_4,\ \alpha_3+\alpha_4,\ \alpha_1+2\alpha_2+\alpha_3,\\
\alpha_1+2\alpha_2+\alpha_3+\alpha_4,\ \alpha_1+2\alpha_2+2\alpha_3+\alpha_4,\\
2\alpha_1+4\alpha_2+2\alpha_3+\alpha_4,\ 2\alpha_1+4\alpha_2+3\alpha_3+\alpha_4,\\
2\alpha_1+4\alpha_2+3\alpha_3+2\alpha_4
 \end{array} & 
 \begin{array}{l}
\alpha_2,\ \alpha_1+\alpha_2,\ \alpha_2+\alpha_3,\\
\alpha_1+\alpha_2+\alpha_3,\ \alpha_2+\alpha_3+\alpha_4,\\
\alpha_1+\alpha_2+\alpha_3+\alpha_4,\ \alpha_1+3\alpha_2+2\alpha_3+\alpha_4,\\
2\alpha_1+3\alpha_2+2\alpha_3+\alpha_4
 \end{array} \\ \hline
 
 \end{array}
 \]
 \caption{Positive roots for $F(4)$ w.r.t. simple root systems}
 \label{F:F4-PosSys}
 \end{table}
 \end{footnotesize}

 \subsection{Parabolic subalgebras}
 \label{sec:parabolicsubalgebras}
 Given a simple root system $\Pi=\{ \alpha_i \}_{i=1}^4 \subset \fh^*$, we let $\{ \sfZ_i \}_{i=1}^4 \subset \fh$ be the dual basis. We specify a subset $I_\fp \subset \{ 1, 2, 3, 4 \}$ (corresponding to a subset of the set of simple roots), and define a {\sl grading element} by $\sfZ = \sum_{i \in I_\fp} \sfZ_i$. This induces a $\mathbb Z$-grading 
 \begin{align}
\label{eq:Zgrading}
 \fg = \bigoplus_{k \in \bbZ} \fg_k, \quad
 \fg_k = \{ x \in \fg : [\sfZ,x] = k x \},
 \end{align}
of the LSA $\fg$, i.e., $[\fg_k, \fg_\ell] \subset \fg_{k+\ell}$, for all $k,\ell \in \bbZ$.  In particular, $\fg_0$ is a sub-LSA, with $\sfZ$ central in $\fg_0$, and each $\fg_k$ is a $\fg_0$-module.  The non-degenerate Killing form on $F(4)$ induces a $\fg_0$-module isomorphism $\fg_{-k} \cong \fg_k^*$.  Let $\mu = \max\{ k\in\mathbb Z_+ : \fg_{-k} \neq 0 \}$ be the {\sl depth} of the grading, which agrees with the {\sl height} of the grading, i.e., $\mu = \max\{ k\in\mathbb Z_+ : \fg_k \neq 0 \}$.
 The {\sl parabolic} sub-LSA corresponding to the $\mathbb Z$-grading \eqref{eq:Zgrading} is defined by $\fp = \fg_{\geq 0} = \bigoplus_{k \geq 0} \fg_k$, and the {\sl symbol algebra} by $\fm= \bigoplus_{k< 0} \fg_k$. The symbol algebra is bracket-generated by $\fg_{-1}$ (namely, $\fg_k=[\fg_{-1},\fg_{k+1}]$ for all $k\leq -2$) and 
the $\mathbb Z$-grading \eqref{eq:Zgrading} is {\em transitive} in the sense of N.\ Tanaka \cite{Tan}, i.e., if $[X,\fg_{-1}]=0$ for $X\in\fp$, then $X=0$. Letting $\der_{\gr}(\fm)$ be the LSA of the zero degree derivations of $\fm$, we then have $\fg_0 \subset \der_{\gr}(\fm)$ and, by the bracket-generating property, also $\der_{\gr}(\fm) \subset\fgl(\fg_{-1})$.
 
 In Table \ref{F:max-parabolic}, we give details for all (inequivalent) gradings associated to {\sl maximal} parabolic subalgebras $\fp \subset \fg$, i.e., those with $|I_\fp| = 1$. In this case $\fg_0=\mathbb{C}Z\oplus\fg_0^{\ss}$, with $\fg_0^{\ss}$ semisimple ideal. In the first column, the parabolic is ornamented by a superscript indicating the simple root system from Table \ref{F:F4-SRS}, and a subscript $k\in \{ 1, 2, 3, 4 \}$ for which $I_\fp = \{ k \}$. 

\begin{footnotesize}
\begin{table}[h]
 \[
 \begin{array}{|l|c|l|l|l|} \hline
\text{Parabolic} & \mbox{depth } \mu & \dim\fg_0, \,\dim \fg_{-1}, \,...,\, \dim\fg_{-\mu} & \fg_0^{\ss} & \fg_{-1} \mbox{ as $\fg_0$-module} \\ \hline
\fp^{\rm I}_1 & 2 & (22|0,0|8,1|0) & \mathfrak{spin}(7) & \CC^{0|8} 
  \vphantom{\frac{A^A}{A^A}}\\
\fp^{\rm I}_2=\fp^{\rm II}_2 & 3 & (10|2,3|3,3|3,1|1) & \fsl(3)\oplus\fsl(1|1) & 
\CC^{3|0}\boxtimes\CC^{1|1}
  \vphantom{\frac{A^A}{A^A}}\\
\fp^{\rm I}_3=\fp^{\rm II}_3=\fp^{\rm III}_2 & 2 & (8|4,6|4,2|2) & \fsl(2)\oplus\fsl(1|2) & 
\CC^{2|0}\boxtimes\bigwedge^2(\CC^{1|2})
  \vphantom{\frac{A^A}{A^A}}\\
\fp^{\rm I}_4=\fp^{\rm II}_4=\fp^{\rm III}_1=\fp^{\rm IV}_1=\fp^{\rm V}_1 & 1 & (12|8,6|4) &
\mathfrak{osp}(2|4) & \CC^{6|4}
  \vphantom{\frac{A^A}{A^A}}\\ 
\fp^{\rm II}_1=\fp^{\rm III}_4=\fp^{\rm IV}_2=\fp^{\rm VI}_1 & 2 & (10|6,4|4,3|1) & \fsl(1|3) & 
\bigodot^3(\mathbb C^{1|3})
  \vphantom{\frac{A^A}{A^A}}\\
\fp^{\rm III}_3=\fp^{\rm IV}_4=\fp^{\rm V}_4=\fp^{\rm VI}_4 & 2 & (10|8,6|4,1|0) & 
\mathfrak{osp}(4|2;\alpha)\;\text{with}\;\alpha=2 & \CC^{6|4}
  \vphantom{\frac{A^A}{A^A}}\\ 
\fp^{\rm IV}_3=\fp^{\rm V}_3=\fp^{\rm VI}_3 & 3 & (8|4,4|4,2|2,2|0) & 
\fsl(2)\oplus\fsl(2|1) & \CC^{2|0}\boxtimes \bigwedge^2(\CC^{2|1})
  \vphantom{\frac{A^A}{A^A}}\\
\fp^{\rm V}_2 & 2 & (14|0,0|8,5|0) & \fsl(2)\oplus\mathfrak{spin}(5) & \CC^{2|0}\boxtimes\CC^{0|4} 
  \vphantom{\frac{A^A}{A^A}}\\
\fp^{\rm VI}_2 & 4 & (12|0,0|6,3|0,0|2,3|0) & \fsl(2)\oplus\fsl(3) & 
\CC^{2|0}\boxtimes\CC^{0|3}
  \vphantom{\frac{A^A}{A^A}}\\
\hline
 \end{array}
 \] 
 \caption{Gradings of $F(4)$ corresponding to maximal parabolics}
 \label{F:max-parabolic}
 \end{table}
 \end{footnotesize}

We make a few remarks concerning the adjoint action of $\fg_0^{\ss}$ on $\fg_{-1}$, which is always irreducible:
 \begin{itemize}
\item The representations $\CC^{0|8}$ and $\CC^{0|4}$ in the description of $\fg_{-1}$ in the cases $\fp^I_{1}$ and $\fp^{V}_2$ are the spin representations in dimension $7$ and $5$ (thought with odd parity). The case $\fp^{V}_2$ is relevant for $5$-dimensional supergravity (see, e.g., \cite{AS} and references therein);
 \item For the only $\mu=1$ case, $\fg_{-1}$ is the unique (up to parity change) irreducible singly atypical representation of $\mathfrak{osp}(2|4)$ of dimension $(6|4)$ \cite[Table 3.64]{FSS}. 
 \item Finally, we remark that in the case $\fp^{VI}_4$ we have $\fg_0^{\ss}\cong
\mathfrak{osp}(4|2;\alpha)$ with $\alpha=2$, whose smallest non-trivial irreducible representation is  of dimension $(6|4)$. This can be seen using the Cartan matrix
\begin{equation}
\label{eq:CartanVI}
\begin{pmatrix}
2 & 3 & 0 &0 \\
-1 &0 & -1 &0\\
0 &-2 &2 &-1\\
0 &0 &-1 & 2
\end{pmatrix}
\end{equation}
corresponding to Dynkin diagram $VI$ and removing the last column and row. (We stress that in our conventions the Cartan matrix has entries equal to $\alpha_i(h_{\alpha_j})$, where $h_{\alpha_j}$ are coroots of the simple root system, $i$ is the row index and $j$ the column index.)
The resulting $3 \times 3$ matrix is, upon rescaling of the column corresponding to the odd isotropic root, the Cartan matrix 
$$
\begin{pmatrix}
2 & 1 & 0  \\
-1 &0 & -1 \\
0 & \alpha &2
\end{pmatrix}
$$
of $\mathfrak{osp}(4|2;\alpha)$ for $\alpha=-\tfrac{2}{3}$. Since all $\alpha$'s on the same orbit under the action of the permutation group $\mathfrak{S}_3$
generated by $\alpha\to \alpha^{-1}$ and $\alpha\to -(1+\alpha)$ are equivalent, the first claim follows. 
(We recall for the reader's convenience that the undeformed $\mathfrak{osp}(4|2)$ corresponds instead to $\alpha=1,-2,-\tfrac12$.)
According to \cite[pg. 917]{MR787332}, the unique smallest 
non-trivial irreducible representation of $\mathfrak{osp}(4|2;\alpha)$ for $\alpha=2$ is $10$-dimensional, so the adjoint action of $\fg_0^{\ss}$ on $\fg_{-1}$ has to be isomorphic to this representation. In our setting, we have $(\fg_{-1})_{\bar 0}\cong \CC\boxtimes\odot^2\CC^2\boxtimes\CC^2$
and $(\fg_{-1})_{\bar 1}\cong \CC^2\boxtimes \CC^2\boxtimes\CC$ 
w.r.t. $\mathfrak{osp}(4|2;\alpha)_{\bar 0}\cong\fsp(2)\oplus\fsl(2)\oplus\fsl(2)$, which is also in agreement with \cite[pg. 917]{MR787332}.
 \end{itemize}

The remainder of our article will focus on {\sl contact} gradings, i.e., those for which $\mu = 2$ and $\fm$ is a supersymmetric analogue of a classical Heisenberg algebra.  In particular, these have $\dim(\fg_{-2}) = (1|0)$, and there are precisely two such cases above:
 \begin{itemize}
 \item[$(i)$] the grading associated to $\fp^{\rm VI}_4$, which we refer to as ``mixed contact grading''; 
 \item[$(ii)$] the grading associated to $\fp^{\rm I}_1$, which we refer to as ``odd contact grading''.
 \end{itemize}
 The organization of the (positive) roots are given in Tables \ref{F:mixed-roots} and \ref{F:odd-roots} respectively.
 
 \begin{table}[h]
 \[
 \begin{array}{|c|l|l|} \hline
 k & \multicolumn{1}{c|}{\Delta_{\bar 0}^+(\fg_k)} &  \multicolumn{1}{c|}{\Delta_{\bar 1}^+(\fg_k)} \\ \hline
 0 & 
 \begin{array}{r@{}r@{}r@{}r}
 \alpha_1\\
 && \alpha_3\\
 \alpha_1 &+ 2\alpha_2 &+ \alpha_3\\
 \end{array} & 
 
 \begin{array}{r@{}r@{}r@{}r}
 & \alpha_2 \\
 & \alpha_2 &+ \alpha_3 \\
 \alpha_1 &+ \alpha_2 \\
 \phantom{\,\,\,}\alpha_1 &+ \alpha_2 &+ \alpha_3
 \end{array}
 \\ \hline
 1 & 
 \begin{array}{r@{}r@{}r@{}r}
 &&&\alpha_4\\
 &&\alpha_3 &+ \alpha_4\\
 \alpha_1 &+ 2\alpha_2 &+ \alpha_3 &+ \alpha_4\\
 \alpha_1 &+ 2\alpha_2 &+ 2\alpha_3 &+ \alpha_4\\
 2\alpha_1 &+ 4\alpha_2 &+ 2\alpha_3 &+ \alpha_4\\ 
 2\alpha_1 &+ 4\alpha_2 &+ 3\alpha_3 &+ \alpha_4
 \end{array} &
 
 \begin{array}{r@{}r@{}r@{}r}
 &\alpha_2 &+ \alpha_3 &+ \alpha_4 \\
 \alpha_1 &+ \alpha_2 &+ \alpha_3 &+ \alpha_4\\
 \alpha_1 &+ 3\alpha_2 &+ 2\alpha_3 &+ \alpha_4\\
 2\alpha_1 &+ 3\alpha_2 &+ 2\alpha_3 &+ \alpha_4
 \end{array}
 \\ \hline
 2 & 
 \begin{array}{r@{}r@{}r@{}r}
 2\alpha_1 &+ 4\alpha_2 &+ 3\alpha_3 &+ 2\alpha_4
 \end{array} & \\ \hline
 \end{array}
 \] 
 \caption{Mixed contact grading (associated to the parabolic $\fp^{\rm VI}_4$)}
 \label{F:mixed-roots}
 \end{table}
 
 \begin{table}[h]
  \[
 \begin{array}{|c|l|l|} \hline
 k & \multicolumn{1}{c|}{\Delta_{\bar 0}^+(\fg_k)} &  \multicolumn{1}{c|}{\Delta_{\bar 1}^+(\fg_k)} \\ \hline
 0 & 
 \begin{array}{l}
\alpha_2,\ \alpha_3,\ \alpha_4,\ \alpha_2+\alpha_3,\ \alpha_3+\alpha_4,\\ 
2\alpha_2+\alpha_3,\ \alpha_2+\alpha_3+\alpha_4,\\ 
2\alpha_2+\alpha_3+\alpha_4,\ 2\alpha_2+2\alpha_3+\alpha_4,\\ 
 \end{array} & 
  \\ \hline
 1 & &
 \begin{array}{l}
\alpha_1,\ \alpha_1+\alpha_2,\ \alpha_1+\alpha_2+\alpha_3,\\
\alpha_1+2\alpha_2+\alpha_3,\ \alpha_1+\alpha_2+\alpha_3+\alpha_4,\\
\alpha_1+2\alpha_2+\alpha_3+\alpha_4,\ \alpha_1+2\alpha_2+2\alpha_3+\alpha_4,\\
\alpha_1+3\alpha_2+2\alpha_3+\alpha_4
 \end{array}
 \\ \hline
 2 & 
 \begin{array}{l}
 2\alpha_1+3\alpha_2+2\alpha_3+\alpha_4
 \end{array}
 & \\ \hline
 \end{array}
 \]
 \caption{Odd contact grading (associated to the parabolic $\fp^{\rm I}_1$)}
 \label{F:odd-roots}
 \end{table}
We note that the odd contact grading of $\fg=F(4)$ is consistent with the parity, in the sense that $\fg_{\bar 0}=\fg_{-2}\oplus\fg_0\oplus\fg_{2}$ and $\fg_{\bar 1}=\fg_{-1}\oplus\fg_{1}$.

In \S\ref{sec:SCG} we will be interested in computing some cohomology groups naturally associated to the mixed and odd contact gradings of $\fg=F(4)$. Since the classical Bott--Borel--Weil theorem does not hold in general for Lie superalgebras, the description of $\fg$ via roots can sometimes be ineffective and one has to resort to more explicit presentations. This is the case for our odd contact grading, for which we provide a spinorial presentation (akin to those used in supergravity theories) in the next subsection.

 \subsection{A spinorial description of the odd contact grading}
 \label{eq:spinorial-description}

Let $(\bbV,g)$ be  a  $7$-dimensional complex vector space $\bbV$ with a non-degenerate symmetric bilinear form $g$ and $(\be_\mu)_{\mu=1,\dots, 7}$ an orthonormal basis of $\bbV$.
We identify $\bbV$ with $\bbV^*$ using $g$, in particular, $\Lambda^2 \bbV\cong\Lambda^2 \bbV^*\cong\fso(\bbV)$ via $v \wedge w \mapsto \big(u\mapsto g(v,u)w - g(w,u) v\big)$, for all $v,w,u\in \mathbb V$.
We let $\mathbb S\cong \mathbb C^8$ be an irreducible module
of the Clifford algebra $\Cl(\bbV)\cong 2\mathbb C(8)$ (there are two such modules up to isomorphism, 
we choose the one for which the volume $\vol\in\Cl(\bbV)$ acts as minus the identity on $\mathbb S$) and
let $\sigma:\fso(\bbV)\to\mathfrak{gl}(\mathbb S)$ be the spinor representation of $\fso(\bbV)$. 
There is an $\fso(\bbV)$-invariant non-degenerate symmetric bilinear form
$\langle -,-\rangle$ on $\mathbb S$ that satisfies
$\langle v\cdot s,t\rangle=-\langle s,v\cdot t\rangle$
for all $s,t\in\mathbb S$, $v\in \bbV$, where $\cdot$ is the Clifford action (see, e.g., \cite{AC}).

The Clifford algebra is generated by the image of the map $\bbV \to \Cl(\bbV)$ that sends $\be_\mu$ to
$\Gamma_\mu$, where $\Gamma_\mu \Gamma_\nu + \Gamma_\nu \Gamma_\mu = - 2 g_{\mu\nu}$. We denote by $\Gamma_{\mu_1\cdots\mu_p}$ the totally
antisymmetric product
\begin{equation*}
  \Gamma_{\mu_1\cdots\mu_p} 
	:= \tfrac1{p!} \sum_{\rho} (-1)^\rho
  \Gamma_{\mu_{\rho(1)}} \cdots \Gamma_{\mu_{\rho(p)}}~,
\end{equation*}
where $(-1)^\rho$ is the sign of a permutation $\rho$ of $\{1,\dots,p\}$. The explicit isomorphism of vector spaces $\Lambda^\bullet \bbV\cong\Cl(\bbV)$ is built out
of the maps $\Lambda^p \bbV \to \Cl(\bbV)$ that send $\be_{\mu_1} \wedge \dots \wedge \be_{\mu_p} \mapsto
  \Gamma_{\mu_1\dots\mu_p}$
and then extending linearly. We recall that $\sigma(A):\mathbb S\to \mathbb S$ is half the Clifford action of $A\in\fso(\bbV)$ as a bivector, in other words $\sigma(\be_\mu\wedge\be_\nu)=\tfrac12\Gamma_{\mu\nu}$.

We have natural isomorphisms of $\fso(\bbV)$-modules
\begin{equation*}
\label{eq:cliffordmultivector}
\odot^2\mathbb S\cong \Lambda^0 \bbV\oplus\Lambda^3 \bbV\qquad\text{and}\qquad\Lambda^2 \mathbb S\cong\Lambda^1 \bbV\oplus\Lambda^2 \bbV\;,
\end{equation*}
whence $\End(\mathbb S)\cong\mathbb S\otimes \mathbb S\cong \bigoplus_{p=0}^{3}\Lambda^p \bbV$. We may define multivectors 
$\omega^{(p)}=\omega^{(p)}(s,t)\in\Lambda^ p\bbV$ for all $s,t\in\mathbb S$ accordingly:
\begin{equation*}
\omega^{(p)}(s,t)=\tfrac{1}{p!}(\overline s \Gamma^{\mu_1\cdots\mu_p} t)\Gamma_{\mu_1\cdots\mu_p}\in\Lambda^p \bbV\,,
\end{equation*}
where the pairing $\langle s,t\rangle$ has been denoted simply  by $\overline s t$, indices are raised/lowered using $g_{\mu\nu}$, and Einstein's summation convention on repeated indices is in force. Clearly 
\begin{equation} \label{E:SymProp}
\begin{aligned}
\omega^{(p)}(s,t)&=(-1)^{p(p+1)/2}\omega^{(p)}(t,s)
\end{aligned}
\end{equation}
for all $s,t\in\mathbb S$ and $0\leq p\leq 7$. We will make extensive use of the well-known {\it Fierz Identity}:
\begin{proposition}
For all $s,t\in\mathbb S$, we have
\begin{equation*}
\label{eq:Fierz-Identity}
\begin{aligned}
s\overline t&=\tfrac18\sum_{p=0}^3\omega^{(p)}(s,t)
\;\\
&=\tfrac18\Big(\overline s t+ (\overline s\Gamma^\mu t)\Gamma_\mu+\tfrac12(\overline s\Gamma^{\mu\nu}t)\Gamma_{\mu\nu}+\tfrac16(\overline s\Gamma^{\mu\nu\rho}t)\Gamma_{\mu\nu\rho}\Big)\;,
\end{aligned}
\end{equation*}
as endomorphisms of $\mathbb S$.
\end{proposition}
Another useful combinatorial property is
$\Gamma_{\mu\nu}\Gamma_{[p]}\Gamma^{\mu\nu}=\big(7-(7-2p)^2\big)\Gamma_{[p]}$ 
for all $\Gamma_{[p]}\in\Lambda^p\bbV$.
A direct consequence is the following.
\begin{lemma}
\label{cor:Jacobi-odd-odd-odd}
The identity
$(\overline t t)s-(\overline t s)t=\tfrac13\omega^{(2)}(s,t)\cdot t$
holds for all $s,t\in\mathbb S$.
\end{lemma}
\begin{proof} From the Fierz Identity and \eqref{E:SymProp}, we have $t\overline{t} = \frac{1}{8} \left(\overline t t+\tfrac16(\overline t\Gamma^{\mu\nu\rho}t)\Gamma_{\mu\nu\rho}\right)$ for all $t \in \bbS$. The difference $(\overline t t)s-(\overline t s)t-\tfrac13\omega^{(2)}(s,t)\cdot t$ is equal to
\begin{equation*}
\begin{aligned}
(\overline t t)s-t(\overline t s)-\tfrac13\tfrac12(\overline s\Gamma^{\mu\nu}t)\Gamma_{\mu\nu}t
&=(\overline t t)s-\tfrac18\Big(\overline t t+\tfrac16(\overline t\Gamma^{\mu\nu\rho}t)\Gamma_{\mu\nu\rho}\Big)s-\tfrac13\tfrac12\Gamma_{\mu\nu}t(\overline{\Gamma^{\mu\nu}t}s) \\
&=\tfrac{7}{8}(\overline t t)s-\tfrac18\tfrac16(\overline t\Gamma^{\mu\nu\rho}t)\Gamma_{\mu\nu\rho}s\\
&\;\;\;-\tfrac16
\tfrac18\Big(\overline{\Gamma_{\mu\nu}t} \Gamma^{\mu\nu}t+\tfrac16(\overline{\Gamma_{\mu\nu}t}\Gamma^{\alpha\beta\delta}\Gamma^{\mu\nu}t)\Gamma_{\alpha\beta\delta}\Big)s \\
&=\tfrac{7}{8}(\overline t t)s-\tfrac18\tfrac16(\overline t\Gamma^{\mu\nu\rho}t)\Gamma_{\mu\nu\rho}s-\tfrac16
\tfrac18\Big(42\overline{t}t-(\overline{t}\Gamma^{\alpha\beta\delta}t)\Gamma_{\alpha\beta\delta}\Big)s\\
&=0\;,
\end{aligned}
\end{equation*}
where in the next-to-last step we used that
\begin{equation*}
\begin{aligned}
\langle \Gamma_{\mu\nu}t,\Gamma^{\mu\nu}t\rangle&=-\langle t,\Gamma^{\mu\nu}\Gamma_{\mu\nu}t\rangle=42\langle t,t\rangle\;,\\
\langle \Gamma_{\mu\nu}t,\Gamma^{\alpha\beta\delta}\Gamma^{\mu\nu}t\rangle&=-\langle \Gamma_{\mu\nu}\Gamma^{\alpha\beta\delta}\Gamma^{\mu\nu}t,t\rangle=-6\langle \Gamma^{\alpha\beta\delta} t,t\rangle\;,
\end{aligned}
\end{equation*}
for all $s,t\in\mathbb S$.
\end{proof}
The odd contact grading of $\fg=F(4)$ is of the form $\fg=\fg_{-2}\oplus\fg_{-1}\oplus\fg_{0}\oplus\fg_{1}\oplus\fg_{2}$, with $\fg_{\pm 2}\cong\mathbb C$, $\fg_{\pm 1}\cong \mathbb S$ and $\fg_0\cong \fso(\bbV)\oplus\mathbb C$ as $\fso(\bbV)$-modules, with the center of $\fg_0$ spanned by the grading element $Z$. 
We fix a basis $\1$ of $\fg_{-2}$ so that the only non-trivial Lie bracket of the symbol $\fm=\fg_{-2}\oplus\fg_{-1}$ is
$[s,t]=\langle s,t \rangle\1$, for all $s,t\in \fg_{-1}$. It is convenient to fix also a basis $\1^\dagger$ of $\fg_{2}$ and adorn every element of $\fg_{1}$ similarly, i.e., $s^\dagger \in\fg_{1}$ for all $s\in \mathbb S$.
\begin{proposition}
\label{prop:explicit-spinorial}
The non-trivial Lie brackets of the Lie superalgebra $\fg=F(4)$ are the natural action of $\fg_0$ on each graded component of the odd contact grading of $\fg$ and
\begin{equation*}
\label{eq:Lie-brackets-F(4)}
\begin{aligned}
{}[s,t]&=\langle s,t\rangle\1\;,\quad[s^\dagger,t^\dagger]=\langle s,t\rangle\1^\dagger\;,\\
[s^\dagger,\1]&=s\;,\quad[\1^\dagger,s]=s^\dagger\;,\\
[s^\dagger, t]&=\tfrac{1}{3}\omega^{(2)}(s,t)-\tfrac{1}{2}\langle s,t\rangle Z\;,\quad[\1^\dagger,\1]=-Z\;,
\end{aligned}
\end{equation*}
for all $s,t\in \mathbb S$. Here $\omega^{(2)}(s,t)\in\Lambda^2 \bbV$ is thought as an element of $\fso(\bbV)$.
\end{proposition}
\begin{proof}
The proof uses $\fg_0$-equivariance and the
direct computation of the Jacobi identities to determine the values of the constants involved. 
The only identity that is not immediate is the one with three odd elements, which follows from Lemma \ref{cor:Jacobi-odd-odd-odd} by partial polarization. 
\end{proof}
\begin{rem}A posteriori, these Lie brackets can be made explicit in terms of our geometric realization in terms of generating superfunctions -- see Table \ref{F:FlatQctSym} and Remark \ref{R:omega2} in \S\ref{S:oddct} .
\end{rem}

We conclude this section with an auxiliary but important finer consequence of the Fierz Identity that will be useful later on. For this purpose, we fix a basis $(\bep_\alpha)_{\alpha=1,\dots,8}$ of $\mathbb S$ with $\langle \bep_\alpha,\bep_\beta\rangle=\delta_{\alpha\beta}$ and let $(\bep^\alpha)_{\alpha=1,\dots, 8}$ be the dual basis  of $\mathbb S^*$. Clearly $\bep^\alpha=\overline\bep_\alpha$. 
\begin{proposition}
\label{prop:omega-vanishes}
Let $\omega:\mathbb S\to\fso(\bbV)$ be a linear map that satisfies $\sigma(\omega_s)s=0$ for all $s\in\mathbb S$. Then $\omega=0$.
\end{proposition}
\begin{proof}
We write $\omega=\tfrac12\omega_\alpha{}^{\mu\nu}\Gamma_{\mu\nu}\otimes \bep^\alpha $ and start with
\begin{equation*}
\begin{aligned}
0&=2\sigma(\omega_s)s=\omega_s\cdot s=\tfrac12\omega_\alpha{}^{\mu\nu}\Gamma_{\mu\nu}s(\bep^\alpha s)\\
&=\tfrac12\omega^\alpha{}^{\mu\nu}\Gamma_{\mu\nu}s(\overline s\bep_\alpha)=\tfrac12\tfrac18\omega^\alpha{}^{\mu\nu}\Gamma_{\mu\nu}
\Big(\overline s s+\tfrac16(\overline s\Gamma^{\mu_1\mu_2\mu_3}s)\Gamma_{\mu_1\mu_2\mu_3}\Big) \bep_\alpha\;,
\end{aligned}
\end{equation*}
which holds for all $s\in S$. Since $\odot^2\mathbb S\cong \Lambda^0 \bbV\oplus\Lambda^3 \bbV$, we may split this equation into
\begin{align}
\label{eq:to-be-split-1}
\tfrac12\omega^{\alpha\mu\nu}\Gamma_{\mu\nu}\bep_\alpha&=0\;,
\\
\label{eq:to-be-split-2}
\tfrac12\omega^{\alpha\mu\nu}\Gamma_{\mu\nu}\Gamma_{\mu_1\mu_2\mu_3}\bep_\alpha&=0\;,
\end{align}
for any $1\leq \mu_1<\mu_2<\mu_3\leq 7$. To proceed further, we first note that
\begin{equation*}
\begin{aligned}
\Gamma^{\mu_1\mu_2}\Gamma_{\mu\nu}\Gamma_{\mu_1\mu_2\mu_3}&=\Gamma^{\mu_1\mu_2}\Gamma_{\mu\nu}\Big(\Gamma_{\mu_1\mu_2}\Gamma_{\mu_3} + g_{\mu_2\mu_3}\Gamma_{\mu_1} - g_{\mu_1\mu_3}\Gamma_{\mu_2}\Big)\\
&=\Big(\Gamma^{\mu_1\mu_2}\Gamma_{\mu\nu}\Gamma_{\mu_1\mu_2}\Big)\Gamma_{\mu_3}
+\Gamma_{\mu_1\mu_3}\Gamma_{\mu\nu}\Gamma^{\mu_1}
-\Gamma_{\mu_3\mu_2}\Gamma_{\mu\nu}\Gamma^{\mu_2}\\
&=-2\Gamma_{\mu\nu}\Gamma_{\mu_3}
-2\Gamma_{\mu_3\mu_2}\Gamma_{\mu\nu}\Gamma^{\mu_2} = -2\Gamma_{\mu\nu}\Gamma_{\mu_3}
-2\Big(\Gamma_{\mu_3}\Gamma_{\mu_2} + g_{\mu_3\mu_2}\Big)\Gamma_{\mu\nu}\Gamma^{\mu_2}\\
&=-4\Gamma_{\mu\nu}\Gamma_{\mu_3}
-2\Gamma_{\mu_3}\Big(\Gamma_{\mu_2}\Gamma_{\mu\nu}\Gamma^{\mu_2}\Big)\\
&=-4\Gamma_{\mu\nu}\Gamma_{\mu_3}
+6\Gamma_{\mu_3}\Gamma_{\mu\nu}\;.
\end{aligned}
\end{equation*}
Hence
\begin{equation}
\begin{aligned}
0&=\tfrac12\omega^{\alpha\mu\nu}\Gamma^{\mu_1\mu_2}\Gamma_{\mu\nu}\Gamma_{\mu_1\mu_2\mu_3}\bep_\alpha=
-2\omega^{\alpha\mu\nu}\Gamma_{\mu\nu}\Gamma_{\mu_3}
\bep_\alpha+3\omega^{\alpha\mu\nu}
\Gamma_{\mu_3}\Gamma_{\mu\nu}
\bep_\alpha\\
&=-2\omega^{\alpha\mu\nu}\Gamma_{\mu\nu}\Gamma_{\mu_3}
\bep_\alpha\Longrightarrow \tfrac12\omega^{\alpha\mu\nu}\Gamma_{\mu\nu}\Gamma_{\mu_3}\bep_\alpha=0\;,
\end{aligned}
\end{equation}
where we first used \eqref{eq:to-be-split-2} and then \eqref{eq:to-be-split-1}. In a completely similar way, we get
\begin{equation}
\label{eq:lasteqhere}
\begin{aligned}
\tfrac12\omega^{\alpha\mu\nu}\Gamma_{\mu\nu}\Gamma_{\mu_2\mu_3}\bep_\alpha&=0\;,
\end{aligned}
\end{equation}
so that, summarizing \eqref{eq:to-be-split-1}-\eqref{eq:lasteqhere}, we have $\tfrac12\omega^{\alpha\mu\nu}\Gamma_{\mu\nu}\Theta\bep_\alpha=0$ for all $\Theta\in\End(\mathbb S)\cong \bigoplus_{p=0}^{3}\Lambda^p \bbV$. Taking $\Theta=\bep_\delta\otimes\bep^\beta$ yields
$\tfrac12\omega_\beta{}^{\mu\nu}\Gamma_{\mu\nu}\bep_\delta=0$ for all $1\leq \beta,\delta\leq 8$, that is, $\omega=0$.
\end{proof}
\section{The main algebraic result: computation of Spencer cohomology groups}
\label{sec:SCG}
 \subsection{Basic definitions}
We are interested in the Spencer cohomology of $\fg=F(4)$ for its contact gradings $\fg=\fg_{-2}\oplus\cdots\oplus\fg_{2}$, i.e., the cohomology associated to the Chevalley--Eilenberg complex $C^{\bullet}(\fm,\fg)$, where
the symbol $\fm=\fg_{-2}\oplus\fg_{-1}$ acts on $\fg$ via the adjoint representation. 
The space of $n$-cochains is $C^n(\fm,\fg)=\fg\otimes\Lambda^n\fm^*$ for all $n\geq 0$, where 
$\Lambda^\bullet$ is meant in the super-sense, and the differential $\partial:C^{n}(\fm,\fg) \to
C^{n+1}(\fm,\fg)$ is explicitly given for $n=0,1$ by the following expressions:
\begin{align}
  \begin{split}\label{eq:CE0}
    &\partial : C^{0}(\fm,\fg) \to C^{1}(\fm,\fg)\\
    &\partial\varphi(X) = (-1)^{x|\varphi|}[X,\varphi]\,,
  \end{split}
  \\
  \begin{split}\label{eq:CE1}
    &\partial : C^{1}(\fm,\fg)\to C^{2}(\fm,\fg)\\
    &\partial\varphi(X,Y) = (-1)^{x|\varphi|}[X,\varphi(Y)] - (-1)^{y(x+|\varphi|)}[Y,\varphi(X)] - \varphi([X,Y])\,,
  \end{split}
\end{align}
where $x,y$ denote the parity of elements $X,Y$ of $\fm$
and $|\varphi|$ the parity of $\varphi\in C^{n}(\fm,\fg)$.

Note that the ${\mathbb Z}$-degree in $\fg$ extends to the space of cochains by declaring that $\fg_d^*$ has degree $-d$ and that the differential $\partial$ has the degree zero. In particular, the complex breaks up
into the direct of sum of complexes for each degree and the group
\begin{equation}
\label{eq:dn}
H^n(\fm,\fg)=\bigoplus_{d\in\mathbb Z} H^{d,n}(\fm,\fg)
\end{equation}
into the direct sum of its homogeneous components.  The space $C^{d,n}(\fm,\fg)$ of $n$-cochains of degree $d$ has a natural $\fg_0$-module structure and the same is true for the spaces of cocycles and coboundaries, as $\partial$ is $\fg_0$-equivariant;
this implies that any $H^{d,n}(\fm,\fg)$ has a representation of $\fg_0$ and therefore of the reductive Lie algebra $(\fg_0)_{\bar 0}$. This equivariance is very useful in our arguments, as we will  demonstrate.
 
We are interested in the group $ H^{d,n}(\fm,\fg)$ for $n=0, 1$ and $d\geq 0$, due to the following classical result of Tanaka \cite{Tan}, whose proof extends verbatim to the supercase:
\begin{proposition} $H^{d,1}(\fm,\fg) = 0$ for $d > 0$ if and only if $\fg \cong \prn(\fm,\fg_0)$.
\end{proposition}
Above, $\prn(\fm,\fg_0)$ refers to the maximal graded Lie superalgebra that extends $\fm$ and $\fg_0$, such that if $X \in \prn(\fm,\fg_0)$ is of non-negative degree and satisfies $[X,\fm] = 0$, then $X = 0$.

\subsection{Spencer cohomology of the mixed contact grading}
 \label{S:Spencer-mixed-contact-grading}
This is the case with associated parabolic subalgebra $\fp=\fp^{VI}_4$. Table \ref{F4table} below recollects the components of the mixed contact grading of $\fg=F(4)$, emphasizing their structure as modules for $\fg_0^{\ss}\cong\mathfrak{osp}(4|2;\alpha)$ (with $\alpha=2$), together with branchings w.r.t. the purely even subalgebra $(\fg_0)_{\bar 0}\cong\mathbb C Z\oplus\mathfrak{sl}(2)\oplus\mathfrak{sl}(2)\oplus\mathfrak{sl}(2)$. In this section, it is convenient to order the three $\fsl(2)$'s in $(\fg_0)_{\bar 0}$ in the opposite order of \S\ref{sec:parabolicsubalgebras}: w.r.t. the simple root system $VI$, the first $\fsl(2)$ corresponds to the root $\alpha_3$, the second to the root $\alpha_1+2\alpha_2+\alpha_3$ and the last to the root $\alpha_1$. This is due to the fact that $\alpha_1=\delta$ is the root generating the ideal $\fsp(2)$ of $\fg_{\bar 0}$, which is more convenient to have as last.
\begin{small}
\begin{table}[h]
\begin{center}
\begin{tabular}{c||*{1}{c}|*{1}{c}|*{1}{c}||}
\toprule
  {\bfseries Graded Components} & ${\bf \fg_0}$ & ${\bf\fg_{\pm 1}}$ & ${\bf\fg_{\pm 2}}$  \\
\midrule[0.02em]\midrule[0.02em]
& $\CC^{1|0}$ &  &  \\
{As $\fg_0$-module} & $\mathfrak{osp}(4|2;\alpha)\;\text{with}\;\alpha=2$ & $\CC^{6|4}$ & $\CC^{1|0}$ \\
\midrule[0.02em]
& $\CC\boxtimes\CC\boxtimes\CC$  &  &  \\
{Even part as $(\fg_0)_{\bar 0}$-module} & $\fsl(2)\boxtimes\CC\boxtimes\CC$\;\;\;\;\, & $\mathbb C^2\boxtimes\odot^2\mathbb C^2\boxtimes\CC$\,\,  & 
$\CC\boxtimes\CC\boxtimes\CC$ \,\,\\
 & $\CC\boxtimes\fsl(2)\boxtimes\CC$\;\;\;\;\, &  & \,\,\\
{}  & \;\;\;\;\;$\CC\boxtimes\CC\boxtimes\fsp(2)$  &  & \\
\midrule[0.02em]
& &  & \\
{Odd part as $(\fg_0)_{\bar 0}$-module}& \,\,$\CC^{2}\boxtimes\CC^{2}\boxtimes\CC^2$ & $\CC\boxtimes\CC^{2}\boxtimes\CC^2$ & \\
&  &  &  \\
\midrule[0.02em]\midrule[0.02em]
{Dimension} & \multicolumn{1}{c|}{$10|8$} & \multicolumn{1}{c|}{$6|4$} & \multicolumn{1}{c||}{$1|0$} \\ 
\bottomrule
\end{tabular}
\end{center}
\caption{Graded components of the mixed contact grading.}
\label{F4table}
\end{table}
\end{small}
\par
\begin{lemma}
\label{lem:centrI}
\label{lem:centralizers}
\hfill\smallskip\par
\begin{itemize}
\item[(i)] The centralizer of $\fm$ in $\fg$ coincides with the center of $\fm$, 
hence $H^{d,0}(\fm,\fg)=0$ for all $d\geq 0$;
	\item[(ii)] The centralizer of $(\fm_{-1})_{\bar 0}$ in $\fg$ is given by $\fm_{-2}\oplus(\fm_{-1})_{\bar 1}\oplus\mathfrak{sp}(2)$.
\end{itemize}
\end{lemma}
\begin{proof}
(i) The ideal of $\fg$ that is generated by the centralizer of $\fm$ in $\fp$ is easily seen to be contained in $\fp$, hence it is trivial by simplicity of $\fg$. This proves the first claim, and the claim on cohomology follows readily from definitions. (ii) Some of the components obtained by restriction of the bracket of $F(4)$ to the irreducible $(\fg_0)_{\bar 0}$-modules of Table \ref{F4table} are automatically zero, by $(\fg_0)_{\bar 0}$-equivariance, parity and $\mathbb Z$-degree.
 It is a straightforward matter  using the root system of $F(4)$ in Table  \ref{F:mixed-roots} to verify that all other components have ``full rank'' -- i.e., image as large as permitted by Schur's lemma, parity and $\mathbb Z$-degree -- with the sole exception
of the Lie brackets between the irreducible $(\fg_0)_{\bar 0}$-components of $\fg_0$. This readily implies (ii).
\end{proof}
There is an interesting and very useful relationship between the Spencer groups \eqref{eq:dn} for the Lie superalgebra $\fg$ and the classical Spencer groups.
Let
\begin{equation*}
0\longrightarrow K^{n}\stackrel{\imath}{\longrightarrow}\Lambda^n \fm^*\stackrel{\res}{\longrightarrow} \Lambda^n \fm_{\bar 0}^*\longrightarrow 0
\end{equation*}
be the short exact sequence given by the natural restriction map $\res:\Lambda^n \fm^*\to\Lambda^n \fm_{\bar 0}^*$ with kernel
$$K^0=0\;,\qquad K^n=\sum_{1\leq i\leq n}\Lambda^{n-i}\fm_{\bar 0}^*\otimes\Lambda^i\fm_{\bar 1}^*	\qquad\text{for}\qquad n>0\;,$$ and let
\begin{equation}
\label{eq:shortexact}
0\longrightarrow \fg\otimes K^\bullet\stackrel{\imath}{\longrightarrow} C^{\bullet}(\fm,\fg)\stackrel{\res}{\longrightarrow} C^{\bullet}(\fm_{\bar0},\fg)\longrightarrow 0
\end{equation}
be the associated short exact sequence of differential complexes. With some abuse of notation, we give the following.
\begin{definition}
\label{def:complexM1}
The differential complex $C^{\bullet}(\fm_{\bar 1},\fg)=\fg\otimes K^\bullet$ is the subcomplex of $C^{\bullet}(\fm,\fg)$ given by 
$C^{0}(\fm_{\bar 1},\fg)=0$ and the $n$-cochains, $n\geq 1$, that vanish when all entries are in $\fm_{\bar 0}$. 
\end{definition}
It is not difficult to see that every morphism in the sequence \eqref{eq:shortexact} is $(\fg_0)_{\bar 0}$-equivariant. The associated long exact sequence in cohomology, together with (i) of Lemma \ref{lem:centrI} and the fact that $\partial$ is $\fg_0$-equivariant gives the following general result.
\begin{proposition}
\label{prop:longexact}
For all $d\geq 0$, there exists a long exact sequence of vector spaces
\begin{equation*}
\begin{split}
\label{sequenza}
0\longrightarrow \xi^d_{\fg}(\fm_{\bar 0})\longrightarrow H^{d,1}(\fm_{\bar 1},\fg)\longrightarrow H^{d,1}(\fm, \fg)\longrightarrow
H^{d,1}(\fm_{\bar 0},\fg)
\end{split}
\end{equation*}
where $\xi^d_{\fg}(\fm_{\bar 0})$ is the component of degree $d$ of the centralizer of $\fm_{\bar 0}$ in $\fg$. The morphisms in the sequence are all $(\fg_0)_{\bar 0}$-equivariant.
\end{proposition}
 Note that the mixed contact grading is compatible with the decomposition
 \begin{equation*}
\label{eq:decomp}
\fg = \fg_{\bar 0}\oplus\fg_{\bar 1}
      = \big(\fso(7)\oplus\fsp(2)\big)\oplus \big(\mathbb S\boxtimes\CC^2\big)\;;
 \end{equation*}
more precisely it induces the contact grading 
 \raisebox{-0.08in}{$
\begin{tikzpicture}
 \draw (1,0) -- (2,0);
 \draw (2,0.07) -- (3,0.07);
 \draw (2,-0.07) -- (3,-0.07); 
 \draw (2.4,0.15) -- (2.6,0) -- (2.4,-0.15);
 \node[draw,circle,inner sep=2pt,fill=white] at (1,0) {};
 \node[draw,circle,inner sep=2pt,fill=white] at (2,0) {};
 \node[draw,circle,inner sep=2pt,fill=white] at (3,0) {};
\node[below] at (2,-0.1) {$\times$};
 \end{tikzpicture}
 $}
on $\fso(7)$, the ideal $\fsp(2)$ of $\fg_{\bar 0}$ sits all in degree zero and the odd part $\fg_{\bar 1}$ has no graded components in degrees $\pm 2$. In particular, $Z$ is an element of $\fso(7)$, which coincides precisely with the grading element of  \raisebox{-0.08in}{$
\begin{tikzpicture}
 \draw (1,0) -- (2,0);
 \draw (2,0.07) -- (3,0.07);
 \draw (2,-0.07) -- (3,-0.07); 
 \draw (2.4,0.15) -- (2.6,0) -- (2.4,-0.15);
 \node[draw,circle,inner sep=2pt,fill=white] at (1,0) {};
 \node[draw,circle,inner sep=2pt,fill=white] at (2,0) {};
 \node[draw,circle,inner sep=2pt,fill=white] at (3,0) {};
\node[below] at (2,-0.1) {$\times$};
 \end{tikzpicture}
 $},
and $\fm_{\bar 0}=(\fm_{-1})_{\bar 0}\oplus (\fm_{-2})_{\bar 0}$ is the Heisenberg algebra
with $(\fm_{-1})_{\bar 0}\cong \mathbb C^2\boxtimes\odot^2\mathbb C^2$ and $(\fm_{-2})_{\bar 0}\cong\mathbb C$ as $\mathfrak{sl}(2)\oplus\mathfrak{sl}(2)$-modules. (Our conventions are consistent with the given ordering of the three copies of $\fsl(2)$ in $(\fg_0)_{\bar 0}\cong \mathbb C Z\oplus\mathfrak{sl}(2)\oplus\mathfrak{sl}(2)\oplus\mathfrak{sp}(2)$, namely the simple roots of $\fso(7)$ are $\alpha_3$, $\alpha_4$ and $\alpha_1+2\alpha_2+\alpha_3$.) Since $\fms{0} \subset \fg_{\bar{0}} = \fso(7)\oplus \fsp(2)$ is the symbol of the contact grading 
of $\fso(7)$ and $\fsp(2)$ sits in degree zero, we can use Kostant's version of the Bott--Borel--Weil theorem \cite{Kos} to compute the Lie algebra cohomology group $H^1(\fms{0},\fg)$ as a module for $(\fg_0)_{\bar{0}}$. 
 
Now, a direct application of Kostant's version of the Bott--Borel--Weil theorem tells us that, as $\fsl(2)\oplus\fsl(2)$-modules:
\begin{equation*}
\begin{aligned}
H^{d,1}(\fm_{\bar 0},\fso(7))&\cong
\begin{cases}\odot^2 \CC^2\boxtimes\odot^4\CC^2\;\text{if}\; d=0\;\\ 
0\;\;\,\,\text{for all}\; d> 0
\end{cases}\\
H^{d,1}(\fm_{\bar 0},\CC)&\cong
\begin{cases} \CC^2\boxtimes\odot^2\CC^2\;\text{if}\; d=1\;\\
0\;\;\,\,\text{for all}\; d\geq 0, d\neq 1
\end{cases}\\
H^{d,1}(\fm_{\bar 0},\mathbb S)&\cong
\begin{cases} 
\CC^2\boxtimes \odot^3\CC^2\;\;\,\,\text{if}\; d=0\;\\
0\;\;\,\,\text{for all}\; d>0
\end{cases}
\end{aligned}
\end{equation*}
so that, thanks to \eqref{eq:decomp},  the only non-trivial homogeneous components of $H^{1}(\fm_{\bar 0},\fg)$ are given by the following $(\fg_0)_{\bar 0}$-modules:
\begin{equation*}
\begin{aligned}
H^{0,1}(\fm_{\bar 0},\fg)_{\bar 0}&\cong \odot^2\CC^2\boxtimes\odot^4\CC^2\boxtimes\CC\\
H^{0,1}(\fm_{\bar 0},\fg)_{\bar 1}&\cong \CC^2\boxtimes\odot^3\CC^2\boxtimes\CC^2\\
H^{1,1}(\fm_{\bar 0},\fg)_{\bar 0}&\cong \CC^2\boxtimes\odot^2\CC^2\boxtimes\fsp(2).
\end{aligned}
\end{equation*}

The following is a consequence of the above discussion, Proposition \ref{prop:longexact}, Lemma \ref{lem:centralizers}(ii), and the immediate fact that the centralizers of $\fm_{\bar 0}$ and 
$(\fm_{-1})_{\bar 0}$ coincide for the mixed contact grading.
\begin{proposition}
\label{prop:3exse}
There exist long exact sequences of $(\fg_0)_{\bar 0}$-modules
\begin{equation*}
\begin{split}
0\longrightarrow \fsp(2)&\longrightarrow H^{0,1}(\fm_{\bar 1},\fg)_{\bar 0}\longrightarrow H^{0,1}(\fm, \fg)_{\bar 0}\longrightarrow
\odot^2\CC^2\boxtimes\odot^4\CC^2\boxtimes\CC\\
0&\longrightarrow H^{0,1}(\fm_{\bar 1},\fg)_{\bar 1}\longrightarrow H^{0,1}(\fm, \fg)_{\bar 1}\longrightarrow
\CC^2\boxtimes\odot^3\CC^2\boxtimes\CC^2
\end{split}
\end{equation*}
\begin{equation*}
\begin{split}
\label{sequenzaIII}
\qquad\qquad\,\; 0&\longrightarrow  H^{1,1}(\fm_{\bar 1},\fg)_{\bar 0}\longrightarrow H^{1,1}(\fm, \fg)_{\bar 0}\longrightarrow \CC^2\boxtimes\odot^2\CC^2\boxtimes\fsp(2)\\
\qquad\qquad\,\; 0&\longrightarrow  H^{1,1}(\fm_{\bar 1},\fg)_{\bar 1}\longrightarrow H^{1,1}(\fm, \fg)_{\bar 1}\longrightarrow 0
\end{split}
\end{equation*}
and
\begin{equation*}
\begin{split}
\label{sequenzaIV}
0\longrightarrow H^{d,1}(\fm_{\bar 1},\fg)\longrightarrow H^{d,1}(\fm, \fg)\longrightarrow 0
\end{split}
\end{equation*}
for all $d>1$.
\end{proposition}
To prove our main cohomological result, it is now mostly a matter of determining the groups $H^{d,1}(\fm_{\bar 1},\fg)$ of the complex of Definition \ref{def:complexM1}, which is a rather straightforward task.
\begin{theorem}
\label{thm:235H^1}
Let $\fg=\fg_{-2}\oplus\cdots\oplus\fg_{2}$ be the mixed contact grading of $\fg=F(4)$, with associated parabolic subalgebra $\fg_{\geq 0}=\fp_4^{\rm VI}$ and Levi factor $\fg_0\cong \mathbb CZ\oplus\mathfrak{osp}(4|2;\alpha)$, where $\alpha=2$. Then
$H^{d,1}(\fm, \fg)=0$ for all $d>0$, while $H^{0,1}(\fm,\fg)$ is an irreducible $\fg_0$-module of dimension $(18|16)$. In particular $\fg_0$ is a maximal subalgebra of the Lie superalgebra $\der_{\gr}(\fm)\cong\mathbb C Z\oplus \mathfrak{spo}(6|4)$ 
of the zero-degree derivations  of $\fm$.
\end{theorem}
\begin{proof}
If $d>4$ then $C^{d,1}(\fm,\fg)=0$ just by degree reasons. We will then deal with $d=0,\ldots, 4$.

We start by showing that $H^{d,1}(\fm_{\bar 1},\fg)=0$ for $d=1,\ldots, 4$. The groups in question consist just of cocycles $\varphi\in\Hom(\fm_{\bar 1},\fg)$ (we recall that $C^{0}(\fm_{\bar 1},\fg)=0$ by definition).
We first note that
$0=\partial\varphi(X,Y) 
=[X,\varphi(Y)]$ for all $X\in\fm_{\bar 0}$ and $Y\in\fm_{\bar 1}$, so that
$
\varphi(\fm_{\bar 1})\subset\fm_{-2}\oplus\fm_{\bar 1}\oplus\mathfrak{sp}(2)
$
thanks to $\xi_{\fg}(\fm_{\bar 0})=\xi_{\fg}((\fm_{-1})_{\bar 0})$ and Lemma \ref{lem:centralizers}. In particular $H^{d,1}(\fm_{\bar 1},\fg)=0$ for $d=2,3,4$, so that $H^{d,1}(\fm,\fg)=0$ for $d=2,3,4$ too by Proposition \ref{prop:3exse}.

Let now $d=1$ and decompose
$\varphi=\varphi_{\bar 0}+\varphi_{\bar 1}$ into the sum of its even $\varphi_{\bar 0}:\fm_{\bar 1}\to (\fg_0)_{\bar 1}$ and odd components 
$\varphi_{\bar 1}:\fm_{\bar 1}\to (\fg_0)_{\bar 0}$, which are both cocycles. By our discussion above $\varphi$ takes values in $\fsp(2)$, i.e., $\varphi_{\bar 0}=0$ and $\varphi_{\bar 1}:\fm_{\bar 1}\to \fsp(2)$; moreover
$
0=\partial\varphi_{\bar 1}(X,Y)=-[X,\varphi_{\bar 1}(Y)] -[Y,\varphi_{\bar 1}(X)]
$
for all $X,Y\in\fm_{\bar 1}$, hence $[X,\varphi_{\bar 1}(Y)]=-[Y,\varphi_{\bar 1}(X)]$.  Now $\fm_{\bar 1}\cong \mathbb C\boxtimes\mathbb C^2\boxtimes\CC^2$, so that the space of cocycles $Z^{1,1}(\fm_{\bar 1},\fg)\cong \bbC^*\boxtimes(\bbC^2)^*\boxtimes Z^{1,1}(\mathbb C^{0|2},\fsp(2))$ as a $(\fg_{0})_{\bar 0}$-module, where $Z^{1,1}(\mathbb C^{0|2},\fsp(2))$ is the first prolongation of $\fsp(2)$ acting on the purely odd $2$-dimensional space $\mathbb C^{0|2}$. It is well-known that this prolongation is trivial, see e.g. \cite{MR2552682}, so that $\varphi_{\bar 1}=0$ too and $H^{1,1}(\fm_{\bar1},\fg)=Z^{1,1}(\fm_{\bar 1},\fg)=0$.  (Recall that $B^{1,1}(\fm_{\bar 1},\fg) = 0$ since $C^0(\fm_{\bar 1},\fg) = 0$.)

Proposition \ref{prop:3exse} then tells us $H^{1,1}(\fm,\fg)_{\bar 1}=0$ and that there exists a short exact sequence of $(\fg_0)_{\bar 0}$-modules
\begin{equation*}
\begin{split}
0\longrightarrow H^{1,1}(\fm, \fg)_{\bar 0}\longrightarrow \CC^2\boxtimes\odot^2\CC^2\boxtimes\fsp(2)\;.
\end{split}
\end{equation*}
By irreducibility $H^{1,1}(\fm,\fg)_{\bar 0}$ either vanishes or it is an irreducible $(\fg_0)_{\bar 0}$-module
isomorphic to $\CC^2\boxtimes\odot^2\CC^2\boxtimes\fsp(2)$. 
Interestingly, we may set this dichotomy
without any computation, thanks to the fact that $H^{1,1}(\fm,\fg)$ carries a representation of the Lie superalgebra $\fg_0\cong \mathbb C Z\oplus \mathfrak{osp}(4|2;\alpha)$, and not just of its even part $(\fg_0)_{\bar 0}$. Since $H^{1,1}(\fm,\fg)$ is purely even, the odd part of 
$\mathfrak{osp}(4|2;\alpha)$ acts trivially, so does the whole of $\mathfrak{osp}(4|2;\alpha)$ (as its even part is generated by the odd one). Therefore
$H^{1,1}(\fm,\fg)_{\bar 0}=0$, concluding the proof of the first claim of the theorem.

It remains to study the case $d=0$. By definition, the group $H^{0,1}(\fm,\fg)$ is the quotient of the Lie superalgebra 
$\der_{\gr}(\fm)$ of the zero-degree derivations  of $\fm$ by $\fg_0\cong \mathbb CZ\oplus\mathfrak{osp}(4|2;\alpha)$. Clearly 
$\der_{\gr}(\fm)\cong\mathbb C Z\oplus \mathfrak{spo}(6|4)$ and a simple dimension count shows that $\dim H^{0,1}(\fm,\fg)=(18|16)$.

 A line of arguments close to the one used for $H^{d,1}(\fm_{\bar 1},\fg)=0$ for $d>0$ says 
$H^{0,1}(\fm_{\bar 1},\fg)_{\bar 1}=0$ and that
$
H^{0,1}(\fm_{\bar 1},\fg)_{\bar 0}\cong \big(\mathbb C\boxtimes \fsl(2)\boxtimes\CC\big)\oplus\big(\CC\boxtimes\CC\boxtimes\fsp(2)\big)
$,
 as endomorphisms of $\fm_{\bar 1}\cong\CC\boxtimes \CC^2\boxtimes\CC^2$ that, in addition, act trivially on $\fm_{\bar 0}$.  Proposition \ref{prop:3exse}
then gives two exact sequences
\begin{equation*}
\begin{split}
0&\longrightarrow \big(\CC\boxtimes\CC\boxtimes\fsp(2)\big)\longrightarrow
 \big(\mathbb C\boxtimes \fsl(2)\boxtimes\CC\big)\oplus\big(\CC\boxtimes\CC\boxtimes\fsp(2)\big)\longrightarrow
\\
&\phantom{\longrightarrow \big(\CC\boxtimes\CC\boxtimes\fsl(2)\big)\;}\longrightarrow
 H^{0,1}(\fm, \fg)_{\bar 0}\longrightarrow
\odot^2\CC^2\boxtimes\odot^4\CC^2\boxtimes\CC\\
0&\longrightarrow H^{0,1}(\fm, \fg)_{\bar 1}\longrightarrow
\CC^2\boxtimes\odot^3\CC^2\boxtimes\CC^2
\end{split}
\end{equation*}
and using $(\fg_0)_{\bar 0}$-equivariance and dimension counting we arrive at 
\begin{equation}
\label{eq:even-and-odd}
\begin{split}
H^{0,1}(\fm, \fg)_{\bar 0}&\cong
(\odot^2\CC^2\boxtimes\odot^4\CC^2\boxtimes\CC)\oplus\big(\mathbb C\boxtimes \fsl(2)\boxtimes\CC\big)\;,\\
H^{0,1}(\fm, \fg)_{\bar 1}&\cong
\CC^2\boxtimes\odot^3\CC^2\boxtimes\CC^2\;.
\end{split}
\end{equation}
 Let now $M=M_{\bar 0}\oplus M_{\bar 1}$ be a non-zero $\mathfrak{osp}(4|2;\alpha)$-irreducible submodule of $H^{0,1}(\fm, \fg)$. Note that $M$ cannot be purely even or odd, otherwise $\mathfrak{osp}(4|2;\alpha)_{\bar 1}$ would act trivially on $M$ and so would $\mathfrak{osp}(4|2;\alpha)$, contradicting \eqref{eq:even-and-odd}. Then $M_{\bar 1}=H^{0,1}(\fm, \fg)_{\bar 1}$, by $\mathfrak{osp}(4|2;\alpha)_{\bar 0}$-irreducibility.
It is now a matter of going through the list of all irreducible $\mathfrak{osp}(4|2;\alpha)$-modules that, upon branching to $\mathfrak{osp}(4|2;\alpha)_{\bar 0}$, include $H^{0,1}(\fm, \fg)_{\bar 1}\cong\CC^2\boxtimes\odot^3\CC^2\boxtimes\CC^2$ and at least one irreducible component of $H^{0,1}(\fm, \fg)_{\bar 0}$. According to \cite[\S 5]{MR787332} (with the caution of considering all $\alpha$'s on the same orbit of $\alpha=2$ under the action of the permutation group $\mathfrak{S}_3$
generated by $\alpha\to \alpha^{-1}$ and $\alpha\to -(1+\alpha)$), any such module has at least {\it three} irreducible components for $\mathfrak{osp}(4|2;\alpha)_{\bar 0}$. It follows that $M=H^{0,1}(\fm,\fg)$, i.e., $H^{0,1}(\fm,\fg)$ is irreducible. In the convention of \cite[pag. 195 and eq. (5.6) at pag. 197]{MR787332}, our module is labeled by the triple $(p,q,r)=(\tfrac{1}{2},\tfrac{3}{2},\tfrac{1}{2})$ for $\alpha=2$ (equivalently for $(\sigma_1,\sigma_2,\sigma_3)=(3,-1,-2)$).  

The claim on the maximality of the embedding $\mathbb C Z\oplus\mathfrak{osp}(4|2;\alpha)\subset \mathbb C Z\oplus\mathfrak{spo}(6|4)$ is clear.
\end{proof}
\begin{cor} \label{C:pr-235}
  Let $\fg$ be as in Theorem \ref{thm:235H^1}, then $\fg \cong \pr(\fm,\fg_0)$.
 \end{cor}
\subsection{Spencer cohomology of the odd contact grading}
This is the case with associated parabolic subalgebra $\fp=\fp^{I}_1$, and we use the conventions and results from \S\ref{eq:spinorial-description}. We immediately note that $H^{d,0}(\fm,\fg)=0$ for all $d\geq 0$ (the proof of (i) of Lemma \ref{lem:centralizers} carries over here unchanged) and that the reduction of the structure algebra from $\mathfrak{co}(\mathbb S)$ to $\fg_0\cong\mathfrak{co}(\bbV)$ (acting via the spinor representation) is encoded in 
\begin{equation*}
\begin{aligned}
H^{0,1}(\fm,\fg)&=\der_{\gr}(\fm)/\fg_0
\cong\mathfrak{co}(\mathbb S)/\fg_0
\\
&\cong\big(\Lambda^0\bbV\oplus\Lambda^2\mathbb S\big)/\big(\Lambda^0\bbV\oplus\Lambda^2 \bbV\big)\cong\Lambda^1 \bbV,
\end{aligned}
\end{equation*}
which is $\fg_0$-irreducible.  The following main result subsumes the first Spencer cohomology group for all non-negative degrees.

\begin{theorem}
\label{sec:Spencer}
 Let $\fg=\fg_{-2}\oplus\cdots\oplus\fg_{2}$ be the odd contact grading of $\fg = F(4)$, with parabolic subalgebra $\fg_{\geq 0}=\fp_1^{\rm I}$ and Levi factor $\fg_0\cong\mathfrak{co}(\bbV)$. Then $H^{d,1}(\fm,\fg)=0$ for all $d>0$, whereas
 $H^{0,1}(\fm,\fg)\cong 	\Lambda^1 \bbV$ as an $\fso(\bbV)$-module.  In particular $\fg_0$ is a maximal subalgebra of the Lie superalgebra $\der_{\gr}(\fm)\cong\mathfrak{co}(\mathbb S)$ 
of the zero-degree derivations  of $\fm$.
\end{theorem}
\begin{proof}
Given the preliminary results from \S\ref{eq:spinorial-description} as Propositions \ref{prop:explicit-spinorial} and \ref{prop:omega-vanishes}, the proof goes over rather straightforwardly. It is therefore postponed to Appendix  \ref{S:Spencer-odd-contact-grading}.
\end{proof}
\begin{cor}
  Let $\fg$ be as in Theorem \ref{sec:Spencer}, then $\fg \cong \pr(\fm,\fg_0)$.
 \end{cor}
 
 \section{Jet-superspaces and contact vector fields}
 \label{S:J1J2}
 
For references on the theory of supermanifolds (superdistributions, superbundle, etc.), we refer the reader to \cite{Kos1977, Leites1980, KST2022}.

 Consider a contact supermanifold $(M,\cC)$, i.e., $M$ is a supermanifold and it is equipped with a corank $(1|0)$ maximally non-integrable superdistribution.  Locally, this means that $\cC = \ker(\sigma)$, where $\sigma$ is an even $1$-form on $M$, and $\eta = d\sigma|_\cC$ is a non-degenerate even $2$-form on $\cC$.  Since $\sigma$ is well-defined only up to scale, then the conformal class $[\eta]$ is distinguished, and we refer to this as a {\sl $\operatorname{CSpO}$-structure} on $\cC$. 
 
 Assuming $\dim M = (2p+1|2q)$, there exist local coordinates $(x^i,u,u_i)$ on $M$, $1 \leq i \leq p+q$, such that
 $\sigma = du - \sum_i (dx^i) u_i$.  Here, $u$ has even parity, while both $x^i$ and $u_i$ have parity denoted by $|i| \in \bbZ_2$, $p$ of which are even and $q$ of which are odd.  Letting $D_{x^i} := \partial_{x^i} + u_i \partial_u$, we have
 $\cC = \langle D_{x^i}, \partial_{u_i} \rangle$ and observe that $d\sigma =\sum_{i} dx^i \wedge du_i$.  (As usual, the exterior product is meant in the super-sense.)  Locally, $(M,\cC)$ is the {\sl first jet-superspace} $J^1(\bbC^{p|q},\bbC^{1|0})$.  For later reference, we will refer to any local frame of $\cC$ having the same components as $\{ D_{x^i}, \partial_{u_i} \}$ w.r.t. $f d\sigma$ for some even invertible superfunction $f$, as a {\sl $\operatorname{CSpO}$-framing}.
 
 The Lagrange--Grassmann bundle $\pi:\widetilde{M}\to M$ is the bundle over $M$ whose total space consists of all Lagrangian subspaces of $(\cC, [\eta])$. This can be made rigorous using the functor of points approach, in such a way that $L$ is a ``super-point'' of $\widetilde M$. We refer the reader to \cite{KST2021} for details. The Lagrange--Grassmann bundle inherits the tautological superdistribution $\widetilde \cC$ given by $\widetilde{\cC}|_L = (\pi_*)^{-1}(L)$, where $L$ is a Lagrangian subspace of $\cC$.  Locally, we take bundle-adapted coordinates $(x^i,u,u_i,u_{ij})$, where $u_{ij}$ is supersymmetric in $i,j$, i.e., $u_{ij} = (-1)^{|i||j|} u_{ji}$, so that $L = \langle \widetilde{D}_{x^i} \rangle$, with
 $\widetilde{D}_{x^i} := \partial_{x^i} + u_i \partial_u + \sum_{j}u_{ij} \partial_{u_j}$. Letting
 $\sigma_k := du_k - \sum_i(dx^i) u_{ik}$, we have $\widetilde\cC = \ker(\sigma,\sigma_k) = \langle \widetilde{D}_{x^i}, \partial_{u_{ij}} \rangle$.  Locally, $(\widetilde{M},\widetilde\cC)$ is the {\sl second jet-superspace} $J^2(\bbC^{p|q},\bbC^{1|0})$.
 
Finally, we define the bundle $p:\widecheck{M} \to \widetilde{M}$ with total space consisting of all subspaces in $\widetilde\cC=\langle \widetilde{D}_{x^i}, \partial_{u_{ij}} \rangle$ that are isotropic 
(w.r.t. the Levi form of $\widetilde\cC$) and complementary in $\widetilde\cC$ to the vertical superdistribution $\langle\partial_{u_{ij}} \rangle$ of $\pi:\widetilde{M} \to M$. The total space inherits a tautological superdistribution $\widecheck \cC$. Locally, we take bundle-adapted coordinates $(x^i,u,u_i,u_{ij},u_{ijk})$, where $u_{ijk}$ are supersymmetric in $i,j,k$, so that $L = \langle \widecheck{D}_{x^i} \rangle$, with
 $\widecheck{D}_{x^i} := \partial_{x^i} + u_i \partial_u + \sum_j u_{ij} \partial_{u_j} + \sum_{j\leq k} u_{ijk} \partial_{u_{jk}}$. Letting
 $\sigma_{jk}:= du_{jk} - \sum_i(dx^i) u_{ijk}$, we have $\widecheck\cC = \ker(\sigma,\sigma_k, \sigma_{jk}) = \langle \widecheck{D}_{x^i}, \partial_{u_{ijk}} \rangle$.  Locally, $(\widecheck{M},\widecheck\cC)$ is the {\sl third jet-superspace} $J^2(\bbC^{p|q},\bbC^{1|0})$.  The higher jet-superspaces are similarly constructed, but we will not need them in this paper.

 We summarize here the relevant jet-superspaces:
  \begin{align*}
 \begin{array}{ccll}
 \mbox{Space} & \mbox{Local model} & \mbox{Coordinates} & \mbox{Cartan superdistribution} \\ \hline
 \widecheck{M}^{\phantom{C^C}} & J^3(\bbC^{p|q},\bbC^{1|0}) & (x^i,u,u_i,u_{ij},u_{ijk}) & \widecheck\cC = \langle \widecheck{D}_{x^i}, \partial_{u_{ijk}} \rangle\\
 \widetilde{M}^{\phantom{C^C}} & J^2(\bbC^{p|q},\bbC^{1|0}) & (x^i,u,u_i,u_{ij}) & \widetilde\cC = \langle \widetilde{D}_{x^i}, \partial_{u_{ij}} \rangle \\
 M^{\phantom{C^C}} & J^1(\bbC^{p|q},\bbC^{1|0}) & (x^i,u,u_i) & \cC = \langle D_{x^i}, \partial_{u_i} \rangle
 \end{array}
 \end{align*}

We recall the following basic definition:
 \begin{definition}
 Let $\cD$ be a superdistribution on a supermanifold.  Its {\sl Cauchy characteristic space} is $\Ch(\cD) := \{ \bX \in \Gamma(\cD) : \cL_\bX \cD \subset \cD \}$.
 \end{definition}
 It is not difficult to see that the Cauchy characteristic spaces of the Cartan superdistributions just defined are trivial. 
 
 A {\sl contact supervector field} on the first jet-superspace $M$ is a vector field $\bS \in \mathfrak{X}(M)$ such that $\cL_\bS \cC \subset \cC$, and completely similar definitions can be given for supervector fields on higher jet-superspaces.  By naturality of the constructions, any contact supervector field $\bS$ on $M$ naturally {\sl prolongs} to a contact supervector field on a higher jet-superspace. Conversely, any contact vector field on a higher jet-superspace is projectable over a contact supervector field on $M$: this is the Lie--B\"acklund theorem, which arises from expressing vertical superdistributions in a covariant manner (e.g., the derived superdistribution $\widetilde\cC^2 := \widetilde\cC+[\widetilde\cC,\widetilde\cC] = \langle \widetilde{D}_{x^i}, \partial_{u_i}, \partial_{u_{ij}} \rangle$ has $\Ch(\widetilde\cC^2) = \langle \partial_{u_{ij}} \rangle$, which is preserved by contact supervector fields on $\widetilde M$).  
 
 Letting $\sigma= du - \sum_i (dx^i) u_i$ be as above, any contact vector field $\bS$ on $M$ is uniquely determined by a generating superfunction $f = \iota_\bS \sigma$. Conversely any superfunction on
$M$ determines a contact supervector field. The explicit expression of $\bS=\bS_f$ in terms of its generating superfunction $f$ is given by
\begin{align} \label{E:Sf}
 \bS_f = f\partial_u - \sum_{i=1}^{p+q} (-1)^{|i|(|f|+1)} (\partial_{u_i} f) D_{x^i} + \sum_{i=1}^{p+q} (-1)^{|i||f|} (D_{x^i} f) \partial_{u_i},
 \end{align}
 and the Lie bracket on supervector fields induces a {\sl Lagrange bracket} on superfunctions via $[\bS_f,\bS_g] = \bS_{[f,g]}$:
\begin{align} \label{E:LB}
 [f,g] = f\partial_u g - (-1)^{|f||g|} g\partial_u f + \sum_{i=1}^{p+q} (-1)^{|i||f|} (D_{x^i} f) \partial_{u_i} g - \sum_{i=1}^{p+q} (-1)^{|g|(|f|+|i|)} (D_{x^i} g) \partial_{u_i} f.
 \end{align}
 The {\sl prolongation} of $\bS_f$ to $\widetilde{M}$ is then given by
 \begin{align*} 
 \widetilde{\bS}_f &= \bS_f + \sum_{j\leq k} h_{jk} \partial_{u_{jk}}, \quad h_{jk} = (-1)^{(|j|+|k|)|f|} \widetilde{D}_{x^j} \widetilde{D}_{x^k} f,
 \end{align*}
 and similarly for the higher jet-superspaces, for instance:
 \begin{align*}
 \widecheck{\bS}_f &= \widetilde{\bS}_f + \sum_{j\leq k \leq \ell} h_{jk\ell} \partial_{u_{jkl}}, \quad h_{jk\ell} = (-1)^{(|j|+|k|+|\ell|)|f|} \widecheck{D}_{x^j} \widecheck{D}_{x^k} \widecheck{D}_{x^\ell} f.
 \end{align*}

 In this article, we will focus on sub-supermanifolds $\Sigma \subset \widetilde{M}$ (for the mixed contact grading) and $\Sigma \subset \widecheck{M}$ (for the odd contact grading). Locally, $\Sigma \subset J^k(\bbC^{p|q},\bbC^{1|0})$ for $k=2$ or $k=3$ and often it is also required that $\Sigma$ is transverse to the projections to the lower jet-superspaces.  These are {\sl super-PDE} on one dependent even variable $u$, and many independent variables $x^i$ ($p$ of which are even and $q$ of which are odd). Geometric structures inherited by $\Sigma$ from the ambient space include:
 \begin{itemize}
 \item a canonical superdistribution $\cC_\Sigma$ induced from the Cartan superdistribution, more precisely, the (annihilator of the) pull-back to $\Sigma$ of the annihilator of the Cartan superdistribution, or, equivalently, the intersection of the tangent bundle $\mathcal{T}\Sigma$ with the Cartan superdistribution;
 \item a vertical subsuperdistribution $V_\Sigma=\mathcal{T}\Sigma\cap V\subset \cC_\Sigma$.  (When $\Sigma \subset \widetilde{M}$, $V = \ker(\mathcal{T}\widetilde{M} \to \mathcal{T}M)\subset \widetilde\cC$.  When $\Sigma \subset \widecheck{M}$, $V = \ker(\mathcal{T}\widecheck{M} \to \mathcal{T}\widetilde{M})\subset \widecheck\cC$.)
 \end{itemize}
 We emphasize that $V_\Sigma$ being distinguished is an important feature of the associated {\sl contact supergeometry} of $\Sigma$: it is an artifact of the Lie--B\"acklund theorem.

In \S\ref{S:mixedct} and \S\ref{S:oddct}, we clarify the geometric origins of the differential equations \eqref{E:F4-2PDE} and \eqref{E:F4-3PDE}, and establish our main result (Theorem \ref{T:main}).

\section{The mixed contact grading and related geometric structures}
\label{S:mixedct}

\subsection{A supervariety and its osculating sequence}

Consider the grading of the LSA $\fg = F(4)$ associated to the parabolic $\fp^{\rm VI}_4$ and grading element $\sfZ = \sfZ_4$.  We have $\fg_0 = \mathfrak{cosp}(4|2;\alpha)$ with $\alpha=2$, which has dimension $(10|8)$ and center $\fz(\fg_0) = \langle \sfZ_4 \rangle$, and $\bbV = \fg_{-1}$ of dimension $(6|4)$ with highest root $-\alpha_4$, cf.\ Table \ref{F:mixed-roots}.  Let $\bbP(\bbV) = \operatorname{Gr}(1|0;6|4) \cong \bbP^{5|4}$ be the projective superspace associated to the linear supermanifold $\bbV = \bbV_{\bar 0} \op \bbV_{\bar 1} \cong \bbC^{6|4}$, which has underlying topological space $\bbP(\bbV_{\bar 0}) \cong \bbP^5$ (see \cite[\S 4.3]{Man}).  We equip $\bbP(\bbV)$ with the natural action of the connected Lie supergroup $G_0 = \operatorname{COSp}(4|2;\alpha) \subset \operatorname{CSpO}(6|4)$ generated by $\fg_0 \subset \fcspo(6|4)$.

Let $\ff = \fg_0^{\rm ss} \cong \mathfrak{osp}(4|2;\alpha)$ be the semisimple part of $\fg_0$ and
$\sfZ_1,\sfZ_2, \sfZ_3$ the dual basis of the simple root system $\alpha_1$, $\alpha_2$, $\alpha_3$ of $\ff$, spanning the Cartan subalgebra of $\ff$. (Here, we are abusing notation: these elements are not necessarily dual to the entire simple root system $VI$ of $F(4)$.) Let $e_{-\alpha_4}$ be a root vector for the root $-\alpha_4$, and consider the topological point $o := [e_{-\alpha_4}] \in \bbP(\bbV_{\bar 0})$ with stabilizer subalgebra $\fq \subset \ff$, which is parabolic.  This can be described via a grading on $\ff$, namely
$
\ff = \ff_{-1} \op \ff_0 \op \ff_1
$,
with $\fq = \ff_0 \op \ff_1$ and $\ff_{\pm 1}$ abelian.  This grading arises from the grading element $\sfZ_3 \in \fh$ and we refer to it as the ``secondary'' grading (compared to the grading of $F(4)$ we started with).  In more detail, the roots of $\ff$ are organized as follows:
\begin{align} \label{E:fgr}
\begin{array}{|c|c|c|c|c} \hline
k & \Delta_{\bar 0}(\ff_k) & \Delta_{\bar 1}(\ff_k)\\ \hline\hline
1 & \alpha_3, \,\, \alpha_1 + 2\alpha_2 + \alpha_3 & \alpha_2 + \alpha_3,\,\, \alpha_1 + \alpha_2 + \alpha_3\\ \hline
0 & \pm \alpha_1 & \pm \alpha_2, \,\, \pm (\alpha_1 + \alpha_2)\\ \hline
-1 & -\alpha_3, \,\, -\alpha_1 - 2\alpha_2 - \alpha_3 & -\alpha_2 - \alpha_3,\,\, -\alpha_1 - \alpha_2 - \alpha_3\\ \hline
\end{array}
\end{align}
We remark that $\ff_0=\CC\oplus \ff_0^{\rm ss}$, where $\fz(\ff_0) = \langle \sfZ_3 \rangle$, $\ff_0^{\rm ss} \cong \fsl(2|1) \cong \mathfrak{osp}(2|2)$, and $\ff_{-1}\cong\CC^{2|2}$ is irreducible as an $\ff_0^{\rm ss}$-representation. 
 Let us concretely express this action in terms of a basis of $\ff_{-1}$ consisting of root vectors adapted to the roots listed in \eqref{E:fgr}.

\begin{lemma}
\label{lem:non-defining-osp}
The following statements are true for the representation of $\ff_0^{\rm ss}\cong\mathfrak{osp}(2|2)$ on $\ff_{-1}\cong\CC^{2|2}$: 
\begin{itemize}
	\item[(i)] It can be explicitly given by the following $4 \times 4$ matrices:
 \begin{align} \label{E:osp22}
 \begin{split}
 h_{10} &=[e_{10},f_{10}]=\begin{psm}
 0 & 0 & 0 & 0\\
 0 & 0 & 0 & 0\\
 0 & 0 & 1 & 0\\
 0 & 0 & 0 & -1
 \end{psm}, \quad e_{10} = \begin{psm}
 0 & 0 & 0 & 0\\
 0 & 0 & 0 & 0\\
 0 & 0 & 0 & 1\\
 0 & 0 & 0 & 0
 \end{psm}, \quad
 f_{10} = \begin{psm}
 0 & 0 & 0 & 0\\
 0 & 0 & 0 & 0\\
 0 & 0 & 0 & 0\\
 0 & 0 & 1 & 0
 \end{psm}, \\
 h_{01} &=[e_{01},f_{01}]= -\tfrac13\begin{psm}
 2 & 0 & 0 & 0\\
 0 & -1 & 0 & 0\\
 0 & 0 & 2 & 0\\
 0 & 0 & 0 & -1
 \end{psm}, \quad
 e_{01} = \begin{psm}
 0 & 0 & c_2 & 0\\
 0 & 0 & 0 & 0\\
 0 & 0 & 0 & 0\\
 0 & c_1 & 0 & 0
 \end{psm}, \quad f_{01} = \begin{psm}
 0 & 0 & 0 & 0\\
 0 & 0 & 0 & c_4\\
 c_3 & 0 & 0 & 0\\
 0 & 0 & 0 & 0
 \end{psm}, \\
 e_{11} &= [e_{10},e_{01}] = \begin{psm}
 0 & 0 & 0 & -c_2\\
 0 & 0 & 0 & 0\\
 0 & c_1 & 0 & 0\\
 0 & 0 & 0 & 0
 \end{psm}, \quad f_{11} =[f_{10},f_{01}]=  \begin{psm}
 0 & 0 & 0 & 0\\
 0 & 0 & -c_4 & 0\\
 0 & 0 & 0 & 0\\
 c_3 & 0 & 0 & 0
 \end{psm}.
 \end{split}
 \end{align}
where $c_2c_3=-\tfrac23$ and $c_1c_4=\tfrac13$. Here, $e_{10}, e_{01}, e_{11},f_{10},f_{01},f_{11}$ are root vectors for the roots $\alpha_1, \alpha_2, \alpha_1+\alpha_2, -\alpha_1, -\alpha_2, -\alpha_1-\alpha_2$, respectively, while $h_{10}$ and $h_{01}$ are coroots;
	\item[(ii)]  It is isomorphic to the irreducible representation with labels $(b,j)=(\tfrac16,\tfrac12)$, in the conventions of \cite{FSS}. In particular, it is not the defining representation of $\mathfrak{osp}(2|2)$ (corresponding to $(b,j)=(0,\tfrac12)$);
	\end{itemize}
\end{lemma}
\begin{proof}
The Cartan matrix 
$\begin{psm} 2 & +3\\ -1 & 0 \end{psm}$
of $\ff_0^{\rm ss}$ is obtained by removing the next-to-last and last rows and columns 
from the Cartan matrix \eqref{eq:CartanVI} of the VI Dynkin diagram of $F(4)$. By rescaling the column corresponding to the odd isotropic simple root one gets $\begin{psm} 2 & -1\\ -1 & 0 \end{psm}$, which is the standard Cartan matrix of $\ff_0^{\ss} \cong \fosp(2|2)$. Let $h_{\alpha_i}$ be coroots of the simple root system VI and $e_{\alpha_i},f_{\alpha_i}$ the corresponding positive and negative root vectors, $i=1,\ldots,4$. Then the presence of the $\fsl_2$-triple $e_{\alpha_1}, h_{\alpha_1}, f_{\alpha_1}$ acting non-trivially only on the odd subspace of $\ff_{-1}$ is clear. This gives the first row of \eqref{E:osp22}, with $h_{10}=h_{\alpha_1}$, $e_{10}=e_{\alpha_1}$, $f_{10}=f_{\alpha_1}$. Now, the Cartan matrix \eqref{eq:CartanVI} has entries equal to $\alpha_i(h_{\alpha_j})$, where $i$ is the row index and $j$ is the column index, so the action of $h_{\alpha_2}$ on $\ff_{-1}$ (in terms of a basis of root vectors) is given by $h_{\alpha_2}= \diag(2,-1,2,-1)$ and we set $h_{01} = -\tfrac13 h_{\alpha_2}$. The Serre--Chevalley relations \cite[\S 2.44]{FSS} for $\ff_0^{\ss} \cong \fosp(2|2)$ imply that $e_{\alpha_2}$ and $f_{\alpha_2}$ can be rescaled to $e_{01}$ and $f_{01}$ so that $h_{01}=[e_{01}, f_{01}]$. The location of the non-trivial entries of $e_{01}$ is clear from the roots of $\ff_{-1}$, e.g., the sum of $\alpha_2$ and $-\alpha_3$ is not a root, so the first column is trivial, while the sum of $\alpha_2$ and $-\alpha_1 - 2\alpha_2 - \alpha_3$ is a root, so only the fourth entry in the second column is non-trivial, etc.  Continuing in this manner yields the second row of matrices in \eqref{E:osp22}, with the constraints $c_2c_3=-\tfrac23$ and $c_1c_4=\tfrac13$. The last row of matrices in \eqref{E:osp22} is then clear. This proves (i). 

The irreducible representation with labels $(b,j)$
decomposes w.r.t. $\mathfrak{osp}(2|2)_{\bar 0}\cong\mathfrak{so}(2)\oplus\fsp(2)$ into $\CC^2_{b}\oplus\CC_{b+\tfrac12}\oplus\CC_{b-\tfrac12}$, where the subscript is the eigenvalue of a normalized generator $h$ of $\fso(2)$.  In the case of $\ff_{-1}$, we have $h=-\tfrac12 h_{10}-h_{01}=\diag(\tfrac{2}{3}, -\tfrac{1}{3}, \tfrac{1}{6}, \tfrac{1}{6})$ commuting with $\fsp(2)$, 
and (ii) follows.
\end{proof}

Consider the supervariety $\cV \subset \bbP(\bbV)$ defined as the $G_0$-orbit through $o=[e_{-\alpha_4}]$, and its {\sl osculating sequence} as in \cite[\S 2.4.4]{KST2021}. This is obtained as follows. Let $\cU(\fg_0)$ be the universal enveloping algebra of $\fg_0$ and $\cU_k(\fg_0)$ the $k$-filtrand, $k\geq 0$, w.r.t. the usual increasing filtration of $\cU(\fg_0)$. The natural map $\cU_k(\fg_0) \to \bbV$ given by $t \mapsto t \cdot e_{-\alpha_4}$ 
is $\ff_0^{\rm ss}$-equivariant, since $\ff_0^{\rm ss}$ annihilates $e_{-\alpha_4}$, but is not difficult to see that its image is, in fact, a $\fq$-module. 
\begin{definition}
The {\it $k$-th osculating space} $\widehat{T}^{(k)}_o \cV \subset \bbV$ of $\cV$ at $o = [e_{-\alpha_4}]$ is the image of the natural map $\cU_k(\fg_0) \to \bbV$.
\end{definition}
Since $\cU_{k+1}(\fg_0)\cdot o=\cU_k(\fg_0)\cdot o+\ff_{-1}\cdot\cU_k(\fg_0)\cdot o$, we then obtain a $\fq$-invariant filtration by successively applying $\ff_{-1}$ to $o$. We have
\begin{align*}
o = \widehat{T}_o^{(0)} \cV \,\,\subset\,\, \widehat{T}_o^{(1)} \cV \,\,\subset\,\, \widehat{T}_o^{(2)} \cV \,\,\subset\,\, \widehat{T}_o^{(3)} \cV = \bbV,
\end{align*}
with associated graded vector space $
\operatorname{gr}(\bbV) = N_0 \op N_1 \op N_2 \op N_3$, where $N_k := \widehat{T}_o^{(k)} \cV \,/\, \widehat{T}_o^{(k-1)} \cV$ is called the {\sl normal space} of degree $-k$ (w.r.t. $\sfZ_3$).
The corresponding roots are:
 \begin{align} \label{E:normal}
 \begin{array}{|c|c|c|c|}\hline
 & \multicolumn{1}{c}{\Delta_{\bar 0}(\fg_{-1})} &   \multicolumn{1}{|c|}{\Delta_{\bar 1}(\fg_{-1})} & \dim \\ \hline\hline
 N_0 & -\alpha_4 & \cdot & (1|0)\\ \hline
 N_1 & \begin{array}{c} 
 -\alpha_3 - \alpha_4\\ 
 -\alpha_1 - 2\alpha_2 - \alpha_3 - \alpha_4
 \end{array} &
 \begin{array}{c}
 -\alpha_2 - \alpha_3 - \alpha_4\\
 -\alpha_1 - \alpha_2 - \alpha_3 - \alpha_4
 \end{array} & (2|2) \\ \hline
 N_2 & \begin{array}{c} 
 -\alpha_1 - 2\alpha_2 - 2\alpha_3 - \alpha_4\\
 -2\alpha_1 - 4\alpha_2 - 2\alpha_3 - \alpha_4 
 \end{array} & 
 \begin{array}{c}
 -\alpha_1 - 3\alpha_2 - 2\alpha_3 - \alpha_4\\
 -2\alpha_1 - 3\alpha_2 - 2\alpha_3 - \alpha_4
 \end{array}& (2|2) \\ \hline
 N_3 & -2\alpha_1 - 4\alpha_2 - 3\alpha_3 - \alpha_4 & \cdot & (1|0) \\ \hline
 \end{array}
 \end{align}

The natural map $\cU_k(\fg_0) \to \bbV$ induces a surjection $\cU_k(\fg_0) \to N_k$, which in turn descends to a well-defined surjection $S^k(\fg_0) \cong \cU_k(\fg_0) / \cU_{k-1}(\fg_0) \to N_k$.  This map still has a kernel. Its restriction to $S^k(\ff_{-1})$, which we denote by $\varphi_k : S^k(\ff_{-1}) \to N_k$, is still surjective and also $\ff_0^{\rm ss}$-equivariant.  Hence, as $\ff_0^{\rm ss}$-modules, 
\begin{align} \label{E:Nk}
 N_k \cong S^k(\ff_{-1}) / \ker(\varphi_k).
\end{align}
Via these $\ff_0^{\rm ss}$-equivariant isomorphisms, $\operatorname{gr}(\bbV)$ inherits from the natural product structure on $S^\bullet(\ff_{-1})$ a supercommutative, associative $\bbZ$-graded superalgebra structure
$
 N_i \otimes N_j \to N_{i+j}
$,
 which we now make explicit.  For $\ff_{-1}$, take the following root vectors:
 \begin{align*}
 \begin{array}{|c|c|c|c|c|} \hline
 & \multicolumn{2}{c|}{\mbox{even part}} & \multicolumn{2}{c|}{\mbox{odd part}} \\ \hline
 \mbox{Root} & -\alpha_3 & -\alpha_1 - 2\alpha_2 - \alpha_3 & -\alpha_2 - \alpha_3 & -\alpha_1-\alpha_2 - \alpha_3\\  \hline
 \mbox{Root vector} & Y_1 & Y_2 & \Theta_1 & \Theta_2 \\ \hline
 \end{array}
 \end{align*}
 Using \eqref{E:normal} and the identification \eqref{E:Nk}, we have the following representatives of equivalence classes:
 \begin{align} \label{E:grVb}
 \begin{array}{|c|c|c|}\hline
 k & (N_k)_{\bar{0}} &  (N_k)_{\bar{1}}\\ \hline\hline
 0 & 1 & \cdot\\ \hline
 1 & \begin{array}{c} 
 Y_1\\ 
 Y_2
 \end{array} &
 \begin{array}{c}
 \Theta_1\\
 \Theta_2
 \end{array}\\ \hline
 2 & 
 \begin{array}{c} 
 Y_1 Y_2 \equiv \Theta_1 \Theta_2\\
 (Y_2)^2
 \end{array} & 
 \begin{array}{c}
 Y_2 \Theta_1\\
 Y_2 \Theta_2
 \end{array}\\ \hline
 3 & Y_1(Y_2)^2 \equiv Y_2 \Theta_1 \Theta_2 & \cdot\\ \hline 
 \end{array}
 \end{align}
 A priori we have $Y_1 Y_2 \equiv c\Theta_1 \Theta_2$ in $N_2$, but by rescaling the chosen basis elements of $\ff_{-1}$, we have normalized to $c = 1$. (We will show before Proposition \ref{prop:cubic-forms} that this normalization 
gives an additional constraint on the constants $c_1,\ldots,c_4$ appearing in \eqref{E:osp22}.)
Beyond this relation, we have the additional relations
 \begin{align*}
 & (Y_1)^2 = 0, \quad Y_1\Theta_1=Y_1\Theta_2=0, \quad (Y_2)^3 = 0, \quad (Y_2)^2 \Theta_1 = (Y_2)^2 \Theta_2 = 0.
 \end{align*}
The relations $(\Theta_1)^2 = (\Theta_2)^2 = 0$ and $\Theta_1 \Theta_2 = -\Theta_2 \Theta_1$ are implicit since $\Theta_1,\Theta_2$ are odd.

 \subsection{Explicit local form of the supervariety} \label{S:localV}

 We proceed as in \cite[\S 2.4.3]{KST2021}, expressing relevant supermanifolds and group actions in terms of the functor of points formalism. Letting $\bbA = \bbA_{\bar 0} \oplus \bbA_{\bar 1}$ denote an arbitrary finite-dimensional supercommutative superalgebra, the linear supermanifold $\bbV$ has associated functor of points given by 
 \begin{align*}
 \bbA \mapsto \bbV(\bbA) := (\bbV \otimes \bbA)_{\bar 0} = (\bbV_{\bar 0} \otimes \bbA_{\bar 0}) \oplus (\bbV_{\bar 1} \otimes \bbA_{\bar 1}).
 \end{align*}
 For $\bbP(\bbV)$, the functor of points is $\bbA \mapsto \bbP(\bbV)(\bbA) := \bbP^{1|0}(\bbV \otimes \bbA)$, which  consists of all free $\bbA$-modules in $\bbV \otimes \bbA$ of rank $(1|0)$.  The group $G_0$ is thought of in terms of $\bbA \mapsto G_0(\bbA)$, where the (set-theoretic) group $G_0(\bbA)$ acts on $\bbV(\bbA)$ by means of even transformations with coefficients in $\bbA$, and this induces a corresponding $G_0$-action on $\bbP(\bbV)$.  We let $\cV \subset \bbP(\bbV)$ be the $G_0$-orbit through $o = [e_{-\alpha_4}]$.  The formalism is necessary to precisely articulate the notion of exponentiation of infinitesimal transformations below in order to locally express $\cV$.
 
 Let us order the generators in \eqref{E:grVb} as follows:
 \begin{align*} 
 \begin{split}
 & 1, \quad Y_1, \quad Y_2, \quad \Theta_1,\quad \Theta_2, \\ &  Y_1 Y_2 \equiv \Theta_1\Theta_2,\quad (Y_2)^2, \quad Y_2 \Theta_1, \quad Y_2 \Theta_2, \quad Y_1(Y_2)^2 \equiv Y_2\Theta_1 \Theta_2.
 \end{split}
 \end{align*}
 Apply each generator to the chosen basepoint $e_{-\alpha_4}$ to get a corresponding basis $\mathsf{b} = ( v_0, ..., v_9 )$ of $\bbV = \fg_{-1}$.  (Note that $v_3, v_4, v_7, v_8$ are odd, and the rest are even.)  Since $\ff_{-1}$ is abelian, we have that $\bbV \cong \gr(\bbV)$ as $\ff_{-1}$-modules, so the $\ff_{-1}$-action can be expressed in the aforementioned basis in terms of the following supercommuting matrices:
 \begin{footnotesize}
\begin{align*}
\!\!\!\!\!\!\!\!\!\!
 Y_1 = \begin{psm}
 0 & 0 & 0 & 0 & 0 & 0 & 0 & 0 & 0 & 0\\
 1 & 0 & 0 & 0 & 0 & 0 & 0 & 0 & 0 & 0\\
 0 & 0 & 0 & 0 & 0 & 0 & 0 & 0 & 0 & 0\\
 0 & 0 & 0 & 0 & 0 & 0 & 0 & 0 & 0 & 0\\
 0 & 0 & 0 & 0 & 0 & 0 & 0 & 0 & 0 & 0\\
 0 & 0 & 1 & 0 & 0 & 0 & 0 & 0 & 0 & 0\\
 0 & 0 & 0 & 0 & 0 & 0 & 0 & 0 & 0 & 0\\
 0 & 0 & 0 & 0 & 0 & 0 & 0 & 0 & 0 & 0\\
 0 & 0 & 0 & 0 & 0 & 0 & 0 & 0 & 0 & 0\\
 0 & 0 & 0 & 0 & 0 & 0 & 1 & 0 & 0 & 0  
 \end{psm}, \quad 
 Y_2 = \begin{psm}
 0 & 0 & 0 & 0 & 0 & 0 & 0 & 0 & 0 & 0\\
 0 & 0 & 0 & 0 & 0 & 0 & 0 & 0 & 0 & 0\\
 1 & 0 & 0 & 0 & 0 & 0 & 0 & 0 & 0 & 0\\
 0 & 0 & 0 & 0 & 0 & 0 & 0 & 0 & 0 & 0\\
 0 & 0 & 0 & 0 & 0 & 0 & 0 & 0 & 0 & 0\\
 0 & 1 & 0 & 0 & 0 & 0 & 0 & 0 & 0 & 0\\
 0 & 0 & 1 & 0 & 0 & 0 & 0 & 0 & 0 & 0\\
 0 & 0 & 0 & 1 & 0 & 0 & 0 & 0 & 0 & 0\\
 0 & 0 & 0 & 0 & 1 & 0 & 0 & 0 & 0 & 0\\
 0 & 0 & 0 & 0 & 0 & 1 & 0 & 0 & 0 & 0  
 \end{psm}, \quad
\Theta_1 = \begin{psm}
 0 & 0 & 0 & 0 & 0 & 0 & 0 & 0 & 0 & 0\\
 0 & 0 & 0 & 0 & 0 & 0 & 0 & 0 & 0 & 0\\
 0 & 0 & 0 & 0 & 0 & 0 & 0 & 0 & 0 & 0\\
 1 & 0 & 0 & 0 & 0 & 0 & 0 & 0 & 0 & 0\\
 0 & 0 & 0 & 0 & 0 & 0 & 0 & 0 & 0 & 0\\
 0 & 0 & 0 & 0 & 1 & 0 & 0 & 0 & 0 & 0\\
 0 & 0 & 0 & 0 & 0 & 0 & 0 & 0 & 0 & 0\\
 0 & 0 & 1 & 0 & 0 & 0 & 0 & 0 & 0 & 0\\
 0 & 0 & 0 & 0 & 0 & 0 & 0 & 0 & 0 & 0\\
 0 & 0 & 0 & 0 & 0 & 0 & 0 & 0 & 1 & 0  
 \end{psm}, \quad 
 \Theta_2 = \begin{psm}
 0 & 0 & 0 & 0 & 0 & 0 & 0 & 0 & 0 & 0\\
 0 & 0 & 0 & 0 & 0 & 0 & 0 & 0 & 0 & 0\\
 0 & 0 & 0 & 0 & 0 & 0 & 0 & 0 & 0 & 0\\
 0 & 0 & 0 & 0 & 0 & 0 & 0 & 0 & 0 & 0\\
 1 & 0 & 0 & 0 & 0 & 0 & 0 & 0 & 0 & 0\\
 0 & 0 & 0 & -1 & 0 & 0 & 0 & 0 & 0 & 0\\
 0 & 0 & 0 & 0 & 0 & 0 & 0 & 0 & 0 & 0\\
 0 & 0 & 0 & 0 & 0 & 0 & 0 & 0 & 0 & 0\\
 0 & 0 & 1 & 0 & 0 & 0 & 0 & 0 & 0 & 0\\
 0 & 0 & 0 & 0 & 0 & 0 & 0 & -1 & 0 & 0  
 \end{psm}.
 \end{align*}
 \end{footnotesize}
 Now take parameters $\lambda_1, \lambda_2 \in \bbA_{\bar 0}$ and $\theta_1, \theta_2 \in \bbA_{\bar 1}$.  Exponentiation yields
 \begin{align*}
 & \exp(\lambda_1 Y_1) = \id + \lambda_1 Y_1, \quad \exp(\lambda_2 Y_2) = \id + \lambda_2 Y_2 + \frac{(\lambda_2)^2}{2} (Y_2)^2, \\
 & \exp(\theta_1 \Theta_1) = \id + \theta_1 \Theta_1, \quad \exp(\theta_2 \Theta_2) = \id + \theta_2 \Theta_2
 \end{align*}
 We now compute $\exp(\theta_2 \Theta_2) \exp(\theta_1 \Theta_1) \exp(\lambda_2 Y_2) \exp(\lambda_1 Y_1) \cdot e_{-\alpha_4}$.  Computing products, we obtain the following expression in right-coordinates for the supervariety $\cV$:
 \begin{align*}
\!\!\!\!\!\!
 \begin{psm}
 1 & 0 & 0 & 0 & 0 & 0 & 0 & 0 & 0 & 0\\
 0 & 1 & 0 & 0 & 0 & 0 & 0 & 0 & 0 & 0\\
 0 & 0 & 1 & 0 & 0 & 0 & 0 & 0 & 0 & 0\\
 \theta_1 & 0 & 0 & 1 & 0 & 0 & 0 & 0 & 0 & 0\\
 \theta_2 & 0 & 0 & 0 & 1 & 0 & 0 & 0 & 0 & 0\\
 \theta_1\theta_2 & 0 & 0 & -\theta_2 & \theta_1 & 1 & 0 & 0 & 0 & 0\\
 0 & 0 & 0 & 0 & 0 & 0 & 1 & 0 & 0 & 0\\
 0 & 0 & \theta_1 & 0 & 0 & 0 & 0 & 1 & 0 & 0\\
 0 & 0 & \theta_2 & 0 & 0 & 0 & 0 & 0 & 1 & 0\\
 0 & 0 & \theta_1\theta_2 & 0 & 0 & 0 & 0 & -\theta_2 & \theta_1 & 1
 \end{psm} \begin{psm}
 1 & 0 & 0 & 0 & 0 & 0 & 0 & 0 & 0 & 0\\
 \lambda_1 & 1 & 0 & 0 & 0 & 0 & 0 & 0 & 0 & 0\\
 \lambda_2 & 0 & 1 & 0 & 0 & 0 & 0 & 0 & 0 & 0\\
 0 & 0 & 0 & 1 & 0 & 0 & 0 & 0 & 0 & 0\\
 0 & 0 & 0 & 0 & 1 & 0 & 0 & 0 & 0 & 0\\
 \lambda_1\lambda_2 & \lambda_2 & \lambda_1 & 0 & 0 & 1 & 0 & 0 & 0 & 0\\
 \tfrac{\lambda_2^2}{2} & 0 & \lambda_2 & 0 & 0 & 0 & 1 & 0 & 0 & 0\\
 0 & 0 & 0 & \lambda_2 & 0 & 0 & 0 & 1 & 0 & 0\\
 0 & 0 & 0 & 0 & \lambda_2 & 0 & 0 & 0 & 1 & 0\\
 \tfrac{\lambda_1\lambda_2^2}{2} & \tfrac{\lambda_2^2}{2} & \lambda_1\lambda_2 & 0 & 0 & \lambda_2 & \lambda_1 & 0 & 0 & 1
 \end{psm} \begin{psm}
 1\\ 0\\ 0\\ 0\\ 0\\ 0\\ 0\\ 0\\ 0\\ 0
 \end{psm} =
 \begin{psm}
 1\\ \lambda_1\\ \lambda_2\\ \theta_1\\ \theta_2\\ \lambda_1\lambda_2 + \theta_1\theta_2 \\ \tfrac{\lambda_2^2}{2} \\ \lambda_2\theta_1 \\ \lambda_2\theta_2 \\ \lambda_2\theta_1\theta_2 + \tfrac{\lambda_1\lambda_2^2}{2}
 \end{psm}
 \end{align*} 
 
Although legitimate, the local expression obtained here for the supervariety is not convenient for facilitating our later transition to super-PDE.  Namely, $\mathsf{b} = \{ v_0,..., v_9 \}$ is in fact not a $\operatorname{CSpO}$-basis of $\bbV$, i.e., a basis w.r.t. which the $\operatorname{CSpO}$-structure $[\eta]$ is in canonical form.  (For example, $\{ D_{x^i}, \partial_{u_i} \}$ from \S\ref{S:J1J2} is a $\operatorname{CSpO}$-frame of $\cC$ w.r.t. $\eta = d\sigma = \sum_i dx^i \wedge du_i$.)

In our algebraic setting, the $\operatorname{CSpO}$-structure $[\eta]$ corresponds to the $\fg_0$-equivariant bracket $\Lambda^2 \fg_{-1} \to \fg_{-2}$.  Since $\fg_{-2}$ is the root space for $-2\alpha_1-4\alpha_2-3\alpha_3-2\alpha_4$, then the roots \eqref{E:normal} indicate that $\eta$ must be a linear combination of
$
 v^0 \wedge v^9, v^1 \wedge v^6, v^2 \wedge v^5, v^3 \wedge v^8, v^4 \wedge v^7
$.
 (We remind that $\eta$ is skew-symmetric in the super-sense, so the last two entries are in fact symmetric, and that the value of $v^i\wedge v^j$ on vectors $x, y\in\fg_{-1}$ is obtained as insertions from the left
$x \otimes y \mapsto \iota_x \iota_y (v^i\wedge v^j)$ followed by
the usual sign rule.)  Since $\ff_{-1}$ acts trivially on $\fg_{-2}$, we impose $\ff_{-1}$-{\em invariance} of $\eta$, i.e.\ $\eta(Ax,y) + (-1)^{|A||x|} \eta(x,Ay) = 0$, for all $x,y \in \fg_{-1}$, $A\in\ff_{-1}$.  This forces $\eta$ to be a multiple of
$
 v^0 \wedge v^9 - v^1 \wedge v^6 - v^2 \wedge v^5 + v^3 \wedge v^8 - v^4 \wedge v^7
$.

 Now, using the $\operatorname{CSpO}$-basis
$
 \mathsf{b}' = (v_0, -v_1, -v_2, v_3, v_4, -v_9, -v_6, -v_5, v_8, -v_7)
$
instead of the basis $\mathsf{b} = ( v_0,..., v_9 )$,
 we have the column vector of right-coordinates
 \begin{align*}
 \begin{psm}
 1\\ -\lambda_1\\ -\lambda_2\\ \theta_1\\ \theta_2\\ 
 - \tfrac{\lambda_1\lambda_2^2}{2}  -\lambda_2\theta_1\theta_2 \\
 -\tfrac{\lambda_2^2}{2} \\
 -\lambda_1\lambda_2 - \theta_1\theta_2 \\
\lambda_2\theta_2 \\ -\lambda_2\theta_1 \\ 
 \end{psm},
 \end{align*}
which projectivizes to our $\ell\in \cV$.
 Using the canonical isomorphism $\bbV \otimes \bbA \cong \bbA \otimes \bbV$, we interchange right-coordinates with left-coordinates via the ``sign-rule'', giving the following row vector of left-coordinates:
 \begin{align} \label{E:supervar}
 \left(1, \,\, -\lambda_1,\,\, -\lambda_2,\,\, -\theta_1,\,\, -\theta_2,\,\,
  -\tfrac{\lambda_1\lambda_2^2}{2} -\lambda_2\theta_1\theta_2, \,\, -\tfrac{\lambda_2^2}{2},\,\, -\lambda_1\lambda_2 - \theta_1\theta_2,\,\,
-\lambda_2\theta_2,\,\, \lambda_2\theta_1\right).
 \end{align}
 This is our local parametrization of the supervariety $\cV \subset \bbP(\bbV)$.  Finally, we remark that $\widehat{T}_o \cV$ is spanned by the root spaces associated with $N_0 \op N_1$ (see \eqref{E:normal}), so is clearly a Lagrangian subspace of $\bbV$ w.r.t. $\eta : \bigwedge^2 \fg_{-1} \to \fg_{-2}$.  By $G_0$-invariance of $\cV$ and since $G_0 \subset \operatorname{CSpO}(6|4)$, we have that $\widehat{T}_\ell \cV$ is Lagrangian for any $\ell \in \cV$.
  
 \subsection{Supersymmetric cubic forms and a key identity}
 
 The multiplication $N_1 \otimes N_2 \to N_3$ is a non-degenerate $\ff_0^{\ss}$-equivariant pairing, and $\ff_0^{\rm ss}$ acts trivially on $N_3$.  Distinguishing (the equivalence class of) the element $Y_1 (Y_2)^2$ in $N_3$ induces an identification $N_2 \cong (N_1)^*$ as $\ff_0^{\ss}$-modules. Set $W = N_1$ and $(w_1,w_2,w_3,w_4) = (Y_1,Y_2,\Theta_1,\Theta_2)$.  The dual basis w.r.t. the product $N_1 \otimes N_2 \to N_3 \cong \bbC$ is then given by 
 \begin{align} \label{E:Wd}
 w^1 = (Y_2)^2, \quad w^2 = Y_1 Y_2, \quad w^3 = Y_2 \Theta_2, \quad w^4 = -Y_2 \Theta_1.
 \end{align}
 (Again, strictly speaking we refer to equivalence classes on the right-hand sides here.)
 
 The multiplication $N_1 \otimes N_1 \to N_2$ yields an even $\ff_0^{\rm ss}$-equivariant map $W \otimes W \to W^*$, which we dualize as an even, supersymmetric, $\ff_0^{\rm ss}$-invariant cubic form $\fC \in S^3 W^*$, since the algebra structure is induced from $S^\bullet(\ff_{-1})$.  More precisely, $W \otimes W \to W^*$ is given by
 $u \otimes v \mapsto \iota_u \iota_v \fC$,
 where $\iota_u$ and $\iota_v$ refer to insertions from the left, and 
\eqref{E:Wd} indicates that 
 \begin{align} \label{E:fC}
 \fC = w^1 (w^2)^2 - 2 w^2 w^3 w^4\;.
 \end{align}
 Writing $\fC = \fC_{abc} w^a w^b w^c$ (Einstein summation convention), the only non-trivial components of $\fC$ are
$
 \fC_{122} = \fC_{212} = \fC_{221} = \tfrac{1}{3}$ and $\fC_{234} = \fC_{324} = \fC_{342} = -\fC_{243} = -\fC_{423} = -\fC_{432} = -\tfrac{1}{3}
$.
We remark that the cubic form \eqref{E:fC} is $\ff_0^{\ss}$-invariant.  Using  \eqref{E:osp22}, we find that
\begin{align*}
 \begin{split}
 0 &= e_{01} \cdot \fC = (2c_1 - c_2) (w^2)^2w^3,\\
 0 &= f_{01} \cdot \fC = -2(c_3 + c_4) w^1w^2w^4,
 \end{split} 
\end{align*}
so that the constants $c_1,c_2,c_3,c_4$ are finally constrained to
$
 (c_2,c_3,c_4) = \left(2c_1, -\frac{1}{3c_1}, \frac{1}{3c_1}\right)
$.
 \begin{prop}
\label{prop:cubic-forms}
 Any $\ff_0^{\ss}$-invariant, even, supersymmetric cubic forms in $S^3 W^*$ and $S^3 W$ are, respectively, multiples of 
 \begin{align*}
 \fC &= w^1 (w^2)^2 - 2 w^2 w^3 w^4,\\
 \fC^* &= w_1 (w_2)^2 - w_2 w_3 w_4.
 \end{align*}
 \end{prop}
 
 \begin{proof}
 A basis for the even subspace of $S^3 W$ is 
 \[
 (w_1)^3, \quad (w_1)^2 w_2, \quad w_1 (w_2)^2, \quad (w_2)^3, \quad w_1 w_3 w_4, \quad w_2 w_3 w_4.
 \]
 A direct computation using \eqref{E:osp22} shows that $\ff_0^{\ss}$-invariance of $\fC^*$ forces it to be a multiple of $c_2 w_1 (w_2)^2 - 2 c_1 w_2 w_3 w_4$, substituting $c_2 = 2c_1$ yields the result. The proof for $\fC$ in $S^3W^*$ is analogous.
 \end{proof}
 
 We refer to $\fC^*$ as the {\sl dual cubic} form.  Although any multiple of $\fC^*$ is also $\ff_0^{\ss}$-invariant, we now pin down $\fC^*$ as we have stated it so that a key cubic form identity holds.

Let $\bbA = \bbA_{\bar 0} \op \bbA_{\bar 1}$ be any finite-dimensional supercommutative superalgebra and define $W(\bbA) := (W \otimes \bbA)_{\bar 0} \cong (\bbA \otimes W)_{\bar 0}$, where the isomorphism is induced via the usual ``sign rule''.  Similarly, $W^*(\bbA) := (\bbA \otimes W^*)_{\bar 0}$.  We now extend the definition of $\fC$ from $W$ to $W(\bbA)$ using left $\bbA$-linearity:
for any $T \in W(\bbA)$, and using Einstein summation convention below, we write 
 $T = t^a w_a = \lambda_1 w_1 + \lambda_2 w_2 + \theta_1 w_3 + \theta_2 w_4$,
 where $\lambda_1,\lambda_2 \in \bbA_{\bar 0}$ and $\theta_1,\theta_2 \in \bbA_{\bar 1}$, and define
 \begin{align} \label{E:fCT}
 \fC(T^3) := t^c t^b t^a \fC_{abc} = \lambda_1 (\lambda_2)^2 + 2 \lambda_2\theta_1\theta_2\;.
 \end{align}
 We also use the notation
 \begin{align} \label{E:cubic-der}
 \fC_c(T^2) := \tfrac{1}{3} \partial_{t^c}(\fC(T^3)), \quad
 \fC_{bc}(T) := \tfrac{1}{2} \partial_{t^b}(\fC_c(T^2)),
 \end{align}
 so that $\fC_{abc} = \partial_{t^a}(\fC_{bc}(T))$ and $\fC(T^3) = t^c \fC_c(T^2) = t^c t^b \fC_{bc}(T) = t^c t^b t^a \fC_{abc}$, as expected.
 
 Given $T^* = t^*_a w^a = \mu_1 w^1 + \mu_2 w^2 + \phi_1 w^3 + \phi_2 w^4$,  with $\mu_1,\mu_2 \in \bbA_{\bar 0}$ and $\phi_1,\phi_2 \in \bbA_{\bar 1}$, we similarly have
 \begin{align*}
 \fC^*((T^*)^3) = \mu_1 (\mu_2)^2 + \mu_2 \phi_1\phi_2
 \end{align*}
 and likewise introduce tensors $(\fC^*)^c$ and $(\fC^*)^{bc}$.  
 A straightforward (but slightly tedious) direct check yields the following key identity for $\fC$ and $\fC^*$:
 
 \begin{lemma}[Cubic form identity] \label{L:cubic}
 We have:
 \begin{align} \label{E:cubic}
 \fC_b(T^2) \fC_a(T^2) (\fC^*)^{ab}(T^*) = \frac{4}{27} \fC(T^3) t^c t^*_c.
 \end{align}
 \end{lemma}
 
 Alternatively, we can easily confirm the above Lemma in Maple as follows: 
 \begin{verbatim}
restart: with(Physics):
Setup(mathematicalnotation=false): 
Setup(anticommutativeprefix={theta,phi}):
t:=[lambda1,lambda2,theta1,theta2]:
ts:=[mu1,mu2,phi1,phi2]:
C:=lambda1*lambda2^2+2*lambda2*theta1*theta2:
Cs:=mu1*mu2^2+mu2*phi1*phi2:
C1:=[seq(1/3*diff(C,t[a]),a=1..4)]:
C1s:=[seq(1/3*diff(Cs,ts[a]),a=1..4)]:
C2s:=Matrix(4,4,(a,b)->1/2*diff(C1s[b],ts[a])):
LHS:=expand(simplify(add(add(C1[b]*C1[a]*C2s[a,b],a=1..4),b=1..4))):
RHS:=expand(4/27*C*add(t[c]*ts[c],c=1..4)):
test:=expand(LHS-RHS);
 \end{verbatim}
  In our earlier studies \cite{KST2021, The2018}, the cubic form identity \eqref{E:cubic} played an important role in a PDE symmetry calculation. With this algebraic background now in-hand, let us now turn to the geometric setting.
 
 \subsection{An $F(4)$-invariant 2nd order super-PDE system}  Consider a contact supermanifold $(M^{7|4},\cC)$.  This is locally equivalent to $J^1(\bbC^{3|2},\bbC^{1|0})$, so take standard jet coordinates $(x^i,u,u_i)$ with even coordinates $x^0, x^1,x^2, u, u_0, u_1, u_2$ and odd coordinates $x^3,x^4,u_3,u_4$.  We also use the indices $0 \leq i,j,k \leq 4$, while $1 \leq a,b,c \leq 4$.
 
 We will endow the rank $(6|4)$ contact superdistribution $\cC$ with an additional geometric structure fibrewise invariant under $\fg_0 \cong \mathfrak{cosp}(4|2;\alpha)$ for $\alpha = 2$.  By Theorem \ref{thm:235H^1}, $\fg_0 \subset \mathfrak{cspo}(6|4)$ is a maximal subalgebra, so such a structure indeed reduces the structure algebra {\it precisely} to  $\fg_0$.
 
 The first such structure is a field of supervarieties $\cV \subset \bbP(\cC)$ from \S\ref{S:localV}, locally parametrized w.r.t. a $\operatorname{CSpO}$-basis as in \eqref{E:supervar}.  Explicitly, via $\fC(T^3)$ given in \eqref{E:fCT}, we can rewrite \eqref{E:supervar} as
 \begin{align*}
 \left(1, \,\, -t^a,\,\, -\tfrac{1}{2} \fC(T^3),\,\, -\tfrac{3}{2} \fC_a(T^2) \right)\;,
 \end{align*}
 and use these as components w.r.t. an arbitrary $\operatorname{CSpO}$-frame $(\bX_0,\bX_a,\bU^0,\bU^a)$, i.e.,
 \begin{align} \label{E:supervar-field}
 \left[ \bX_0 - t^a \bX_a - \tfrac{1}{2} \fC(T^3) \bU^0 - \tfrac{3}{2} \fC_a(T^2) \bU^a \right], 
 \end{align}
 where brackets denote projectivization.
 
 \begin{definition}
\hfill
\begin{itemize}
	\item[$(i)$] A {\sl mixed-contact $F(4)$-supergeometry} $(M^{7|8},\cC,\cV)$ is a contact supermanifold $(M^{7|4},\cC)$ whose contact superdistribution $\cC$  of rank $(6|4)$ is additionally equipped with a field of supervarieties $\cV \subset \bbP(\cC)$ given by the (Zariski closure of the) parametrization \eqref{E:supervar-field} w.r.t. some $\operatorname{CSpO}$-frame $(\bX_0,\bX_a,\bU^0,\bU^a)$,
		\item[$(ii)$] The symmetry superalgebra of $(M,\cC,\cV)$ consists of the contact supervector fields preserving $\cV$: 
$
\finf(M,\cC,\cV)= \{ \bX \in \fX(M) : \cL_\bX \cC \subset \cC, \, \cL_\bX \cV \subset \cV\} 
$. 
	\item[$(iii)$] The {\sl flat} mixed-contact $F(4)$-supergeometry is the mixed-contact $F(4)$-supergeometry determined by the $\operatorname{CSpO}$-frame
$
 \bX_i = \partial_{x^i} + u_i \partial_u, 
 \bU^i = \partial_{u_i}$.
\end{itemize}
 \end{definition}

 The second structure that gives rise to the $\fg_0$-reduction is the family of affine tangent spaces along $\cV$, given by $\widehat\cV = \{ \widehat{T}_\ell \cV : \ell \in \cV \}$.  These spaces are Lagrangian, so $\widehat\cV \subset \LG(\cC)$.  For $\ell \in \cV$ corresponding \eqref{E:supervar-field}, $\widehat{T}_\ell \cV$ is spanned by itself \eqref{E:supervar-field} and its derivatives w.r.t. each parameter $t^a$.  Taking appropriate linear combinations, $\widehat{T}_\ell \cV$ is then spanned by 
 \begin{align} \label{E:Lag}
 \bX_0 + \fC(T^3) \bU^0 + \tfrac{3}{2} \fC_a(T^2) \bU^a, \qquad 
 \bX_a + \tfrac{3}{2} \fC_a(T^2) \bU^0 + 3 \fC_{ab}(T) \bU^b,
 \end{align}
 where $1\leq a \leq 4$ and we used the notation from \eqref{E:cubic-der}.  While $\cV$ canonically determines $\widehat\cV$, in fact the converse is also true:

 \begin{prop}
 A mixed-contact $F(4)$-supergeometry $(M,\cC,\cV \subset \bbP(\cC))$ and $(M,\cC, \widehat\cV \subset \LG(\cC))$ have the same contact symmetries.
 \end{prop}
 \begin{proof}
Osculating $\widehat{T}_\ell \cV$ further yields the second affine tangent space $\widehat{T}^{(2)}_\ell \cV$.  For example, at $\ell = o$, we have the filtration $o \subset \widehat{T}_o \cV \subset \widehat{T}^{(2)}_o \cV$ with associated-graded vector space $N_0 \op N_1 \op N_2$ from \eqref{E:normal}.  W.r.t. the (conformal) symplectic form $\eta$, $\widehat{T}^{(2)}_o \cV$ has orthogonal complement equal to $o$ itself, hence, by $G_0$-invariance, $\widehat{T}^{(2)}_\ell \cV$ has $\eta$-orthogonal complement equal to $\ell$.  Consequently, $\cV$ is canonically determined from $\widehat\cV$.
\end{proof}

\begin{theorem} \label{T:symbound-mixed}
The symmetry superalgebra of any mixed-contact $F(4)$-supergeometry $(M^{7|4},\cC,\cV)$ has $\dim \finf(M,\cC,\cV) \leq \dim F(4) = (24|16)$.
 \end{theorem}
 \begin{proof} Let $\fg = F(4)$, equipped with the mixed-contact grading.  Since $\cV \subset \bbP(\cC)$ reduces the structure group from $\operatorname{CSpO}(6|4)$ to $G_0 = \operatorname{COSp}(4|2;\alpha)$ for $\alpha = 2$, any mixed-contact $F(4)$-supergeometry is a filtered $G_0$-structure with symbol $\fm = \fg_{-2} \op \fg_{-1}$.  From Theorem \ref{thm:235H^1}, we have $H^{d,1}(\fm,\fg) = 0$ for all $d > 0$, which is equivalent to the Tanaka--Weisfeiler prolongation satisfying $\prn(\fm,\fg_0)\cong\fg$.  The claim then follows from \cite[Thm.1.1]{KST2022}.
 \end{proof}

 Henceforth, we consider the {\em flat} mixed-contact $F(4)$-supergeometry.  From \S\ref{S:J1J2}, Lagrangian subspaces of $\cC$ are locally of the form $\bX_i + u_{ij} \bU^j = \partial_{x^i} + u_i \partial_u + u_{ij} \partial_{u_j}$, so \eqref{E:Lag} yields for $\widehat{T}_\ell \cV$ the parametric equations
 \begin{align} \label{E:2PDE}
 \begin{pmatrix}
 u_{00} & u_{0b}\\
 u_{a0} & u_{ab}
 \end{pmatrix}
  = 
 \begin{pmatrix}
 \fC(T^3) & \frac{3}{2}\fC_b(T^2)\\
 \frac{3}{2}\fC_a(T^2) & 3\fC_{ab}(T)
 \end{pmatrix}.
 \end{align}
 Using the explicit expression \eqref{E:fCT} of the cubic form, we can write this out as:
 \begin{align*}
 \begin{pmatrix}
 u_{ij}
 \end{pmatrix} = 
 \begin{pmatrix}
 \lambda_1 (\lambda_2)^2 + 2 \lambda_2\theta_1\theta_2 & \frac{1}{2} (\lambda_2)^2 & \lambda_1 \lambda_2 + \theta_1 \theta_2 & \lambda_2 \theta_2 & -\lambda_2 \theta_1\\
 \frac{1}{2} (\lambda_2)^2 & 0 & \lambda_2 & 0 & 0\\
 \lambda_1 \lambda_2 + \theta_1 \theta_2 & \lambda_2 & \lambda_1 & \theta_2 & -\theta_1\\
 \lambda_2 \theta_2 & 0 & \theta_2 & 0 & -\lambda_2\\
 -\lambda_2 \theta_1 & 0 & -\theta_1 & \lambda_2 & 0\\
 \end{pmatrix}.
 \end{align*}
 (We remind the reader that $x^3, x^4, u_3, u_4$ are odd variables.)  Eliminating the parameters via $(\lambda_1,\lambda_2,\theta_1,\theta_2) = (u_{22},u_{12},-u_{24},u_{23})$, we obtain the super-PDE system \eqref{E:F4-2PDE}.
 
 Let us now turn to the contact symmetries of our super-PDE system.  For $X = (x^1,x^2 | x^3,x^4)$ and $P = (u_1, u_2 | u_3, u_4)$, we define 
$
 \fC(X^3) = x^1 (x^2)^2 + 2 x^2 x^3 x^4$ and 
$
 \fC^*(P^3) = u_1 (u_2)^2 + u_2 u_3 u_4$.
 \begin{prop} The even function
 \begin{align*}
 f = u(u - x^i u_i) - \frac{1}{2} \fC(X^3) u_0 + \frac{1}{2} \fC^*(P^3) x^0 + \frac{9}{4} \fC_c(X^2) (\fC^*)^c(P^2)
 \end{align*}
 generates a contact symmetry of the flat mixed-contact $F(4)$-supergeometry $(M,\cC,\cV)$.
 \end{prop}
 
 \begin{proof} The same proof appearing in \cite[Prop.4.11]{KST2021} works here.  That calculation only relied on $\fC$ satisfying the key cubic form identity that we have verified in Lemma \ref{L:cubic}.
 \end{proof}
   
 It is clear that $1, x^i,u_i$ are all generating functions for contact symmetries of \eqref{E:2PDE}.  Using the Lagrange bracket \eqref{E:LB}, we obtain other symmetries, which we organize in Table \ref{F:sym}.
 
 \begin{table}[h]
 \[
 \begin{array}{|c|c|l|} \hline
 \fg_2 & & u(u - x^i u_i) - \frac{1}{2} \fC(X^3) u_0 + \frac{1}{2} \fC^*(P^3) x^0 + \frac{9}{4} \fC_c(X^2) (\fC^*)^c(P^2)\\ \hline
 \fg_1 & & x^0(u - x^i u_i) - \frac{1}{2} \fC(X^3) \\
 & & x^a(u - x^i u_i) + (-1)^{|a|} \left( \frac{3}{2} (\fC^*)^a(P^2) x^0 + \frac{9}{2} \fC_b(X^2) (\fC^*)^{ba}(P)\right)\\
 & & uu_0 - \frac{1}{2} \fC^*(P^3)\\
 & & uu_a + \frac{3}{2} \fC_a(X^2) u_0 - \frac{9}{2} \fC_{ab}(X) (\fC^*)^b(P^2)\\ \hline
 \fz(\fg_0) & & \sfZ := 2u - x^i u_i\\ \hline
 \fg_0^{\ss} & \ff_1 & x^a u_0 -  \frac{3}{2} (-1)^{|a|} (\fC^*)^a(P^2) \\
 & \fz(\ff_0) & \sfZ_0 := \frac{3}{2} x^0 u_0 + \frac{1}{2} x^c u_c \\
 & \ff_0^{\ss} & \psi^a{}_b := x^a u_b + (-1)^{|a|} ( \frac{1}{3} \delta^a{}_b x^c u_c - 9(-1)^{|a||b|}\fC_{bc}(X) (\fC^*)^{ca}(P) )\\
 & \ff_{-1} & u_a x^0 + \frac{3}{2} \fC_a(X^2) \\ \hline
 \fg_{-1} & & x^i, u_i\\ \hline
 \fg_{-2} & & 1 \\ \hline
 \end{array}
 \]
 \caption{Contact symmetries of the flat mixed-contact $F(4)$-supergeometry and associated super-PDE system (the ranges of the indices are $0 \leq i\leq 4$ and $1 \leq a,b,c \leq 4$)}
 \label{F:sym}
 \end{table}
 
 Table \ref{F:sym} applies remarkably uniformly, having appeared in \cite[Table 8]{KST2021} as well as \cite{The2018}.  Let us count dimensions in our current setting. Note that
 \begin{align*}
 \dim(\fg_{\pm 2}) = (1|0), \quad \dim(\fg_{\pm 1}) = (6|4), \quad \dim(\ff_{\pm 1}) = (2|2).
 \end{align*}
 Both $\fz(\fg_0)$ and $\fz(\ff_0)$ are $(1|0)$-dimensional, and what remains from $\dim(F(4)) = (24|16)$ is to verify that $\dim(\ff_0^{\ss}) = (4|4)$.  A direct substitution confirms that:
 
 \begin{prop} \label{P:f0} \sloppy The LSA $\ff_0^{\ss} = \langle \psi^a{}_b \rangle$ is $(4|4)$-dimensional and is spanned by the even generators $\psi^1{}_1, \psi^3{}_3, \psi^3{}_4, \psi^4{}_3$ and the 
odd generators $\psi^1{}_3, \psi^1{}_4, \psi^2{}_3, \psi^2{}_4$.
 Explicitly, we have the following basis:
 \begin{align*}
 \mbox{even generators}: &\quad 4 u_1 x^1 - 2 u_2 x^2 - u_3 x^3 - u_4 x^4, \quad
 u_3 x^3 - u_4 x^4, \quad u_4 x^3, \quad u_3 x^4;\\
 \mbox{odd generators}: &\quad -u_2 x^4 + u_3 x^1, \quad u_2 x^3 + u_4 x^1, \quad -2u_1x^4 + u_3 x^2, \quad 2u_1 x^3 +u_4 x^2.
 \end{align*}
 \end{prop}

 \begin{theorem}
\label{thm:areyoulookingforalabel}
 The contact symmetry superalgebra of the super-PDE system \eqref{E:F4-2PDE} is isomorphic to $\fg = F(4)$ and it is spanned by the $(24|16)$ symmetries in Table \ref{F:sym}.
 \end{theorem}
 
 \begin{proof}
 By Theorem \ref{T:symbound-mixed}, this is a basis for the symmetry superalgebra.  The proof that it is isomorphic to $F(4)$ follows exactly as in \cite[Thm.5.4]{KST2021}, using the existence of a symmetry $\sfZ = 2u -x^i u_i$ that acts like the grading element, i.e., $\ad(\sfZ)|_{\fg_r} = r \id_{\fg_r}$ for all $r$.
 \end{proof}

 \section{The odd contact grading and related geometric structures}
 \label{S:oddct}

We consider now the odd contact grading of Table \ref{F:odd-roots}, which is associated to the parabolic subalgebra $\fp^{\rm I}_1$.   We have $\fg_0=\fz(\fg_0)\op \ff$, where $\ff\cong B_3$ and the center $\fz(\fg_0)$ is spanned by the grading element $\sfZ = \sfZ_1$. Moreover $\fg_{-1} \cong \bbS$ is the 8-dimensional spin representation of $\ff$, so we may identify $\ff \cong \fspin(7)$ and $\fg_0 \cong \mathfrak{cspin}(7)$. We remark that $\fg_0,\fg_{\pm 2}$ are purely even, while $\fg_{\pm 1}$ is purely odd, and that $\fm = \fg_{-2}\oplus\fg_{-1}$ is an odd Heisenberg algebra with a super-skew (i.e., symmetric) non-degenerate bracket $\eta: \Lambda^2\fg_{-1} \to \fg_{-2}$.  The algebra of zero-degree derivations of $\fm$ is $\fco(\bbS)$, with associated ``contact'' group $\CO(\bbS)=\CC^\times\cdot O(\bbS)$ given by the subgroup of $\GL(\bbS)$ preserving the conformal class $[\eta] := \{ \lambda \eta : \lambda \in \bbC^\times \}$ of $\eta$. 

We note that $\fg_0 \subset \fco(\bbS)$ is a maximal subalgebra, since $\fso(\bbS) \cong \mathfrak{spin}(7)\op\bbC^7$ as irreducible representations for $\ff\cong \mathfrak{spin}(7)$.
 
 \subsection{An explicit presentation of the spin representation $\bbS$}
 \label{subsec:explicit-pres}
 In this subsection, we exclusively work in the classical setting. On $\bbV=\bbC^7$, we fix a basis $(e_1, e_2, e_3, R, f_3, f_2, f_1)$ and define a symmetric bilinear form $g$ with ``anti-diagonal'' components $g_{ij} = \delta_{i,8-j}$ in this basis.  Then $\bbV=E \op \bbC R \op F$, with $E = \langle e_1, e_2, e_3 \rangle$ and $F = \langle f_3, f_2, f_1 \rangle$ maximally isotropic subspaces, and $\ff\cong \fso(\bbV)$ is given in its standard representation by matrices of the form
 \begin{align*}
 X = \begin{psm}
 h_1 & a_{100} & a_{110} & a_{111} & a_{112} & a_{122} & 0 \\
 b_{100} & h_2 &  a_{010} & a_{011} & a_{012} & 0 & -a_{122} \\
 b_{110} & b_{010} & h_3 & a_{001} & 0 & -a_{012} & -a_{112} \\
 b_{111} & b_{011} & b_{001} & 0 & -a_{001} & -a_{011} & -a_{111} \\
 b_{112} & b_{012} & 0 & -b_{001} & -h_3 & -a_{010} & -a_{110} \\
 b_{122} & 0 & -b_{012} & -b_{011} & -b_{010} & -h_2 & -a_{100} \\
 0 & -b_{122} & -b_{112} & -b_{111} & -b_{110} & -b_{100} & -h_1
 \end{psm}\;.
 \end{align*}
The Cartan subalgebra $\fh$ consists of the diagonal matrices $h$ above, and the simple roots are $\beta_1 = \epsilon_1 - \epsilon_2$, $\beta_2 = \epsilon_2 - \epsilon_3$, $\beta_3 = \epsilon_3$, where $\epsilon_i(h) = h_i$. Focusing on the entries of $X$ given by $a_{ijk}$ or $b_{ijk}$ yields root vectors for the roots $\pm (i\beta_1 + j\beta_2 + k\beta_3)$, respectively.  The fundamental weights of $\ff$ are $(\lambda_1,\lambda_2,\lambda_3) = (\epsilon_1, \epsilon_1 + \epsilon_2, \frac{1}{2} (\epsilon_1 + \epsilon_2 + \epsilon_3))$, and we let $\bbV_\lambda$ be the irreducible representation with highest weight $\lambda$. For instance $\bbS \cong \bbV_{\lambda_3}$.

 We equip $\bbS = \Lambda^\bullet E^*$ with the Clifford action by $\bbV$:
 \begin{align} \label{E:Clifford}
 e \cdot \phi := -\sqrt{2} \iota_{e} \phi, \quad
 f \cdot \phi := +\sqrt{2} f^\flat \wedge \phi, \quad
 R \cdot \phi := \begin{cases}
 +i \phi, & \phi \in \Lambda^{\mbox{\tiny even}} E^*\\
 -i \phi, & \phi \in \Lambda^{\mbox{\tiny odd}} E^*
 \end{cases},
 \end{align}
 where $e \in E$, $f \in F$, $\phi\in\bbS$, and $\iota_e$ is the interior product. In our conventions, the highest weight vector in $\bbS$ is given by $\phi=\1\in \Lambda^{\mbox{\tiny even}} E^*$.
We identify $\fso(\bbV) \cong \Lambda^2 \bbV$ as in  \S\ref{eq:spinorial-description} and consider the spinor representation  $\sigma:\fso(\bbV)\to\mathfrak{gl}(\mathbb S)$ on $\bbS$ via
\begin{align*}
\sigma(v_1 \wedge v_2)(\phi) := \frac{1}{4} \big(v_1 \cdot v_2 \cdot \phi - v_2 \cdot v_1 \cdot \phi\big),
 \end{align*}
for all $v_1 \wedge v_2 \in \Lambda^2 \bbV \cong \fso(\bbV)$.
 Let $\vol_{E^*}:= e^1 \wedge e^2 \wedge e^3$. We define the following basis 
 \begin{align} \label{E:spinbasis}
 (\phi_0,\, \phi_1,\, \phi_2\, \phi_3,\, \psi_0,\, \psi_1,\, \psi_2,\, \psi_3) = (\1, \, e^1, \, e^2, \, e^3, \, -\vol_{E^*}, \, e^2 \wedge e^3, \, e^3 \wedge e^1, \, e^1 \wedge e^2)
 \end{align}
 of $\bbS$ and denote the corresponding dual basis elements as usual via raised indices:  $(\phi^0,\ldots, \psi^3)$.
 The induced action of $\fso(\bbV)$ on $\bbS$ is given explicitly by the following matrices:
 \begin{align} \label{E:B3SpinMat}
 \begin{psm}
 \frac{h_1 + h_2 + h_3}{2} & -\tilde{a}_{111} & -\tilde{a}_{011} & -\tilde{a}_{001} & 0 & -\tilde{a}_{012} & -\tilde{a}_{112} & -\tilde{a}_{122} \\
 -\tilde{b}_{111} & \frac{-h_1 + h_2 + h_3}{2} & \tilde{b}_{100} & \tilde{b}_{110} & \tilde{a}_{012} & 0 & \tilde{a}_{001} & -\tilde{a}_{011} \\
 -\tilde{b}_{011} & \tilde{a}_{100} & \frac{h_1 - h_2 + h_3}{2} & \tilde{b}_{010} & \tilde{a}_{112} & -\tilde{a}_{001} & 0 & \tilde{a}_{111} \\
 -\tilde{b}_{001} & \tilde{a}_{110} & \tilde{a}_{010} & \frac{h_1 + h_2 - h_3}{2} & \tilde{a}_{122} & \tilde{a}_{011} & -\tilde{a}_{111} & 0 \\
 0 & -\tilde{b}_{012} & -\tilde{b}_{112} & -\tilde{b}_{122} & \frac{-h_1 - h_2 - h_3}{2} & \tilde{b}_{111} & \tilde{b}_{011} & \tilde{b}_{001} \\
 \tilde{b}_{012} & 0 & -\tilde{b}_{001} & \tilde{b}_{011} & \tilde{a}_{111} & \frac{h_1 - h_2 - h_3}{2} & -\tilde{a}_{100} & - \tilde{a}_{110} \\
 \tilde{b}_{112} & \tilde{b}_{001} & 0 & -\tilde{b}_{111} & \tilde{a}_{011} & -\tilde{b}_{100} & \frac{-h_1 + h_2 - h_3}{2} & -\tilde{a}_{010} \\
 \tilde{b}_{122} & -\tilde{b}_{011} & \tilde{b}_{111} & 0 & \tilde{a}_{001} & -\tilde{b}_{110} & -\tilde{b}_{010} & \frac{-h_1 - h_2 + h_3}{2}
 \end{psm},
 \end{align}
 where the tilded variables are appropriate multiples of their untilded counterparts.  (The explicit factors are not needed in what follows.)

Up to scale, there is a unique $\fso(\bbV)$-invariant symmetric bilinear form and (skew) 4-form on $\bbS$:
 \begin{align}
\eta &= \phi^0 \psi^0 + \phi^1 \psi^1 + \phi^2 \psi^2 + \phi^3 \psi^3, \label{E:g}\\
 \sfQ &= \phi^0 \wedge \phi^1 \wedge \psi^0 \wedge \psi^1 + \phi^0 \wedge \phi^2 \wedge \psi^0 \wedge \psi^2 + \phi^0 \wedge \phi^3 \wedge \psi^0 \wedge \psi^3  \nonumber\\
&\;\;\;\;- \phi^1 \wedge \phi^2 \wedge \psi^1 \wedge \psi^2 - \phi^1 \wedge \phi^3 \wedge \psi^1 \wedge \psi^3 - \phi^2 \wedge \phi^3 \wedge \psi^2 \wedge \psi^3 \label{E:Q}\\
 &\;\;\;\;- 2 \phi^0 \wedge \psi^1 \wedge \psi^2 \wedge \psi^3+ 2 \phi^1 \wedge \phi^2 \wedge \phi^3 \wedge \psi^0\;.\nonumber
\end{align}
It follows that $\sfQ$ is nothing but the Cayley $4$-form on $\bbS$, expressed w.r.t. a spinorial basis that is {\it not} orthogonal. It is well-known that the stabilizer  $\Stab_{\SO(\bbS)}(\sfQ)$ of $\sfQ$ in $\SO(\bbS)$ is connected and isomorphic to $\Spin(7)$. (In fact, the stabilizer in $\GL(\bbS)$ has two connected components, see \cite[549--550]{Bryant}, $\Spin(7)$ and $i\cdot\Spin(7)$, but the latter is evidently not contained in $\SO(\bbS)$). We consider the Hodge dual operations on $\bbS$ and $\bbS^*$ determined by the orientations given by $\vol_{\bbS}=\phi_0\wedge\ldots\wedge\phi_3\wedge\psi_0\wedge\ldots\wedge\psi_3$ and, respectively, $\vol_{\bbS^*}=\vol_{\bbS}^{\flat}=\phi^0\wedge\ldots\wedge\phi^3\wedge\psi^0\wedge\ldots\wedge\psi^3$, where $\flat : \bbS\to \bbS^*$ is the musical isomorphism determined by $\eta$. We remark that the spaces of self-dual $\Lambda^4_{+}\bbS^*$ and anti self-dual $\Lambda^4_{-}\bbS^*$ forms are $\eta$-orthogonal and that $(\Lambda^4_{\pm}\bbS)^\flat=\Lambda^4_{\pm}\bbS^*$. Evidently
$\Lambda^4_{\pm}\bbS^*|_{\Lambda^4_{\mp}\bbS}=0$.
\begin{lemma}
\label{lem:stabilizerCayleyline}
The Cayley $4$-form $\sfQ$ is anti self-dual and the stabilizer $\Stab_{\CO(\bbS)}([\sfQ])$ of its conformal class $[\sfQ]$ in $\CO(\bbS)=\CC^\times\cdot \operatorname{O}(\bbS)$ is isomorphic to $\CSpin(7)=\CC^\times\cdot \Spin(7)$.
\end{lemma}
\begin{proof}
It is easy to see from  \eqref{E:Q} that $\sfQ$ is anti self-dual and $\sfQ\wedge\star\sfQ=14\vol_{\bbS^*}$, so $\eta(\sfQ,\sfQ)=14$.
If $A\in \operatorname{O}(\mathbb S)$ satisfies $A\cdot \sfQ=\lambda\sfQ$ for $\lambda\in\mathbb C^\times$, then $\lambda^2\eta(\sfQ,\sfQ)=\eta(A\cdot\sfQ,A\cdot\sfQ)=\eta(\sfQ,\sfQ)$, so $\lambda=\pm 1$.
If $\det(A)=-1$, then $A\cdot\sfQ=\lambda\sfQ$ is an equality of non-zero $4$-forms with opposite duality, which is not possible. Hence $\det(A)=1$ and the group $\Stab_{\SO(\bbS)}([\sfQ])$ has at most two connected components, corresponding to $\lambda=\pm 1$. The one of the identity is $\Spin(7)$.  In particular, $\Stab_{\SO(\bbS)}([\sfQ])$ is contained in the normalizer of $\Spin(7)$ in $\SO(\bbS)$. Since any automorphism of $\Spin(7)$ is inner, then for any $g\in \Stab_{\SO(\bbS)}([\sfQ])$ there exists $h\in\Spin(7)$ such that $gtg^{-1}=hth^{-1}$ for all $t\in\Spin(7)$. In other words $h^{-1}g$ is in the centralizer of $\Spin(7)$ in $\SO(\bbS)$, which consists only of $\pm 1$ (by Schur's Lemma, since $\Spin(7)$ acts irreducibly on $\bbS$). Hence $g=\pm h$ and $\lambda=1$, i.e., $g\in\Stab_{\SO(\bbS)}(\sfQ)\cong\Spin(7)$.  Summing up, $\Stab_{\CO(\bbS)}([\sfQ])=\mathbb C^\times\cdot \Stab_{\SO(\bbS)}([\sfQ])=\mathbb C^\times\cdot \Stab_{\SO(\bbS)}(\sfQ)\cong \mathbb C^\times\cdot \Spin(7)$.
\end{proof}
We refer to the subgroup $G_0:=\CSpin(7)=\CC^\times\cdot \Spin(7)$ of $\CO(\bbS)$ as the {\em structure group}. Let us define several objects related to the spin representation $\bbS$:
 \begin{itemize}
 \item[$(i)$] The spinor variety $\cV \subset \bbP(\bbS)$ is the $\Spin(7)$-orbit through the isotropic line $o = [\1] \in \bbS$.  Concretely, $\cV$ is the connected, 6-dimensional, null quadric
 \begin{align*}
 \cV = \big\{ [\phi] : \eta(\phi,\phi) = 0 \big\}.
 \end{align*}
(Since both the spinor variety $\cV$ and the null quadric are Zariski-closed irreducible subsets in $\mathbb{P}(\bbS)$ and they have the same dimension, they have to coincide.)  
 \item[$(ii)$] The Lagrangian--Grassmannian $\LG(\bbS)$ consists of all $\eta$-Lagrangian subspaces of $\bbS$.  Its tangent space is modelled on $4\times 4$ skew-symmetric matrices, so $\dim \LG(\bbS) = 6$.  From \cite[p.189]{OV1994}, $\LG(\bbS)$ consists of two open $\SO(\bbS)$-orbits (hence two connected components), and is a single $\operatorname{O}(\bbS)$-orbit. Regarding any $L \in \LG(\bbS)$ via the Pl\"ucker embedding $\LG(\bbS) \hookrightarrow \bbP(\Lambda^4 \bbS)$, the two open $\SO(\bbS)$-orbits consist of
the space $\LG_+(\bbS)$ of self-dual $\eta$-Lagrangian subspaces (e.g., $L=\langle \phi_0, \phi_1, \phi_2, \phi_3\rangle$) and, respectively, the space $\LG_-(\bbS)$ of anti self-dual ones (e.g., $L=\langle \psi_0, \phi_1, \phi_2, \phi_3\rangle$).
\end{itemize} 

We emphasize that both $\cV$ and $\LG(\bbS)$ are actually $\CO(\bbS)$-invariant, so they do {\it not} enforce the desired reduction from $\CO(\bbS)$ to the structure group $G_0=\CSpin(7)$.
 
The stabilizer subalgebra  in $\ff$ at $o=[\1]$ is the parabolic subalgebra $\ff^0$ of $\ff$ corresponding to $\tilde{b}_{ijk} = 0$ in \eqref{E:B3SpinMat}.  There is an $\ff^0$-invariant filtration $\ff = \ff^{-2} \supset \ff^{-1} \supset \ff^0 \supset \ff^1 \supset \ff^2$ of $\ff$, with filtrands $\ff^i = \bigoplus_{j \geq i} \ff_j$ written in terms of a compatible grading $\ff = \ff_{-2} \op ... \op \ff_2$:
 \begin{align*}
 \begin{array}{ccccccc}
 \ff_{-2}\cong\Lambda^2\CC^3 & \ff_{-1}\cong\CC^3 & \ff_{0}\cong\mathfrak{gl}(3) & \ff_{1}\cong(\CC^3)^* & \ff_{2}\cong\Lambda^2(\CC^3)^*\\ \hline
 f_i \wedge f_j &
 R \wedge f_j &
 e_i \wedge f_j &
 R \wedge e_j &
 e_i \wedge e_j
 \end{array} \qquad (1 \leq i,j \leq 3).
 \end{align*}
 We define $\bbS^k = \ff^k \cdot \1$ for $-2 \leq k \leq 2$, which yields an $\ff^0$-invariant filtration of $\bbS$:
 \begin{align*}
 \bbS =: \bbS^{-3} \supset \bbS^{-2} = o^\perp \supset \bbS^{-1} \supset \bbS^0 = o \supset 0\;,
 \end{align*}
with $\bbS^{-1} = \langle \phi_0, \phi_1, \phi_2, \phi_3 \rangle \in \LG_+(\bbS)$.  Trivially extending the filtration to all integers, $\bbS$ becomes a filtered $\ff$-representation, i.e., $\ff^i \cdot \bbS^j \subset \bbS^{i+j}$, for all $i,j \in \bbZ$.
 Via the transitive action of $\Spin(7)$ on the null quadric, a filtration
$
 \bbS = \bbS_\ell^{-3} \cV \supset \bbS_\ell^{-2} \cV = \ell^\perp \supset \bbS^{-1}_\ell \cV \supset \bbS^0_\ell \cV = \ell \supset 0
$
 is then induced at any $\ell \in \cV$.  
\begin{definition}\hfill
\begin{itemize}
	\item[$(i)$] We call $ \widehat{T}_\ell \cV:=\bbS^{-2}_\ell \cV$ the {\sl affine tangent space} to $\cV$ at $\ell$. Geometrically, it is the tangent space to the cone over $\cV$,
	\item[$(ii)$] We call $L_\ell \cV := \bbS^{-1}_\ell \cV$ the {\sl (affine) Lagrangian tangent space} to $\cV$ at $\ell$. It is a self-dual $\eta$-Lagrangian subspace of $\widehat{T}_\ell \cV$. 
\end{itemize}
\end{definition}
We let $L(\cV) := \{ L_\ell \cV : \ell \in \cV \}$ and note that the map 
$
\Phi : \cV \to L(\cV)
$
sending $\ell \mapsto L_\ell \cV$ is $\Spin(7)$-equivariant and surjective by construction. We also introduce the flag manifold given by the incidence relation $\ell\subset L_\ell \cV$:
\begin{equation*}
F(\cV)=\{(\ell,L): \ell\in\cV, L = L_\ell\cV\}\;,
\end{equation*}
which fibres (bijectively) over $\cV$.
 \begin{prop} \label{P:VLV} Consider the spin representation $\bbS$ of $\ff\cong\mathfrak{spin}(7)$ and the Cayley $4$-form $\sfQ \in \Lambda^4 \bbS^*$ as in \eqref{E:Q}. Regarding any $L \in \LG_+(\bbS)$ via the Pl\"ucker embedding $\LG_+(\bbS) \hookrightarrow \bbP(\Lambda^4_+ \bbS)$, then:
 \begin{enumerate}
 \item $L(\cV) =\LG_+(\bbS)$, the connected component of self-dual $\eta$-Lagrangian subspaces in $\LG(\bbS)$.
\item $L(\cV) \subsetneq \{ L \in \LG(\bbS) : \sfQ|_L = 0 \}$. 
\item 
The unique $\CSpin(7)$-invariant element of $\bbP(\Lambda^4 \bbS^*)$ is $[\sfQ]$, 
which vanishes on all of $L(\cV)$.
 \item The map $\Phi : \cV \to L(\cV)$ is a bijection.
 \item Given $\ell = [\phi] \in \cV$, consider $\iota_\phi \sfQ$ and its restriction to $\ell^\perp$.  The subspace of vectors of $\ell^\perp$ inserting trivially into $\iota_\phi \sfQ$ is $L_\ell \cV$.  This distinguishes $F(\cV)$ from $[\sfQ]$.
\item  The stabilizer of $F(\cV)$ in $\CO(\bbS)=\CC^\times\cdot O(\bbS)$ is $\CSpin(7)$. This  distinguishes $[\sfQ]$ from $F(\cV)$.
 \end{enumerate}
 \end{prop}
\begin{proof} 
We regard any $L \in \LG_+(\bbS)$ via the Pl\"ucker embedding $\LG_+(\bbS) \hookrightarrow \bbP(\Lambda^4_+ \bbS)$.
We first note that $\bbS \cong \bbV_{\lambda_3}$ and $\Lambda_+^{4}\bbS \cong \bbV_{2\lambda_3}$, so that the stabilizer subalgebras of the highest weight lines $o =[\1]\in\bbP(\bbS)$ and $\bbS^{-1}=[\phi_0 \wedge \phi_1 \wedge \phi_2 \wedge \phi_3]\in 
\bbP(\Lambda^4_+ \bbS)$ are both equal to the parabolic subalgebra $\ff^0$ of $\ff$. By construction $L(\cV) = \Spin(7)\cdot \bbS^{-1}\subset \LG_+(\bbS)$ is 
the orbit in $\bbP(\Lambda^4_+ \bbS)$ of minimal dimension $\dim \ff / \ff^0 = 6$, in particular 
it is an irreducible projective variety, i.e., a Zariski-closed irreducible subset in $\bbP(\Lambda^4_+ \bbS)$. Now $\LG_+(\bbS)=
\SO(\bbS) \cdot \bbS^{-1}\subset \bbP(\Lambda^4_+ \bbS)$ is also  Zariski-closed and irreducible of dimension $6$, so that $L(\cV)$ is also open in $\LG_+(\bbS)$, and $L(\cV)$ and $\LG_+(\bbS)$ must agree.
Thus, (1) is proved.
 
By anti self-duality of the Cayley form $\sfQ$, we clearly have $\sfQ|_L = 0$ for all $L \in L(\cV)=\LG_+(\bbS)$. Now, 
$\LG(\bbS) \hookrightarrow \bbP(\Lambda^4 \bbS)$ via the Pl\"ucker embedding and
 \begin{align} \label{E:L4S}
 \Lambda^4\bbS \cong  \Lambda^4_-\bbS\oplus\Lambda^4_+\bbS
\cong\big(\bbC \op \bbV_{\lambda_1} \op \bbV_{2\lambda_1}\big) \op \bbV_{2\lambda_3}
 \end{align}
as decomposition into inequivalent irreducible representations for  $\ff$. Clearly the kernel of $\sfQ$ in $ \Lambda^4\bbS $ is given by $\bbV_{\lambda_1} \op \bbV_{2\lambda_1}\op \bbV_{2\lambda_3}$.
We note that $[\phi_0 \wedge \phi_2 \wedge \phi_3 \wedge \psi_1]$ is also a highest weight line, with weight $2\lambda_1$. Thus $\sfQ=0$ on this line and $L(\cV) \subsetneq \{ L \in \LG(\bbS) : \sfQ|_L = 0 \}$, proving (2). Claim (3) is also immediate from \eqref{E:L4S}.

The map $\Phi$ sends the highest weight line of $\cV \hookrightarrow \bbP(\bbS) \cong \bbP(\bbV_{\lambda_3})$ to the highest weight line of $L(\cV) \hookrightarrow \bbP(\bbV_{2\lambda_3})$.  By $\Spin(7)$-equivariancy of $\Phi$, it must abstractly correspond to the 2nd Veronese embedding $[\phi] \mapsto [\phi^2]$, which is bijective, and so (4) is proved. 
 
 Take $\ell=[\phi_0] \in \cV$. By \eqref{E:Q}, the $3$-form $\iota_{\phi_0} \sfQ$ restricted to $\ell^\perp=\langle \phi_0,\phi_1, \phi_2, \phi_3, \psi_1, \psi_2, \psi_3 \rangle$ becomes just $-2\psi^1 \wedge \psi^2 \wedge \psi^3$, and we immediately confirm the claim for $\ell = [\phi_0]$.  The general result follows then from $\Spin(7)$-equivariancy and its transitive action on $\cV$.  This proves (5).

By construction of $F(\cV)$ and since the $L_\ell \cV$'s are self-dual $\eta$-Lagrangian subspaces, we have $\CSpin(7)\subset\Stab_{\CO(\bbS)}(F(\cV))\subset \bbC^\times\cdot \SO(\bbS)$ for the stabilizer of $F(\cV)$ in $\CO(\bbS)$.  We then write $\Stab_{\CO(\bbS)}(F(\cV))=\bbC^\times\cdot H$, for some closed subgroup $H$ of $\SO(\bbS)$ containing $\Spin(7)$. On the other hand $F(\cV)$ is {\em not}
$\SO(\bbS)$-stable (e.g., take $g\in \operatorname{SO}(\bbS)$ that fixes  $\phi_0$, $\psi_0$, $\phi_1$, $\psi_1$ and interchanges $\phi_i$ with $\psi_i$ for $i=2,3$, so that $g\cdot \bbS^{-1}\neq \bbS^{-1}$) 
and $\mathfrak{spin}(7)$ is a maximal subalgebra of $\fso(\bbS)$, hence the
connected component of the identity of $H$ coincides with $\Spin(7)$. 
It then follows that $H$ is contained in the normalizer of $\Spin(7)$ in $\SO(\bbS)$, which is just $\Spin(7)$ itself, cf. the proof of Lemma \ref{lem:stabilizerCayleyline}.  In summary $H=\Spin(7)$, hence $\Stab_{\CO(\bbS)}(F(\cV))=\CSpin(7)$, and the last claim of (6) follows directly from (3).
\end{proof}

To locally describe $\cV$, we exponentiate the action of $\ff_- := \ff_{-2} \op \ff_{-1}$ on $\1 \in \bbS$, so that linear coordinates on $\ff_-$ induce coordinates on an open subset of $U \subset \cV$. Explicitly, we set $\widetilde{a}_{ijk} = \tilde{b}_{ij0} = h_i = 0$ in \eqref{E:B3SpinMat}, let $(\tilde{b}_{111},\tilde{b}_{011},\tilde{b}_{001},\tilde{b}_{012},\tilde{b}_{112},\tilde{b}_{122}) = (u_{23},u_{31},u_{12},u_{01},u_{02},u_{03})$, and then use the matrix exponential on the basis  \eqref{E:spinbasis} to obtain
 \begin{align} \label{E:expgm}
 \begin{psm}
 1 & 0 & 0 & 0 & 0 & 0 & 0 & 0\\
 -u_{23} & 1 & 0 & 0 & 0 & 0 & 0 & 0\\
 -u_{31} & 0 & 1 & 0 & 0 & 0 & 0 & 0\\
 -u_{12} & 0 & 0 & 1 & 0 & 0 & 0 & 0\\
 u_{01} u_{23} + u_{02} u_{31} + u_{03} u_{12} & -u_{01} & -u_{02} & -u_{03} & 1 & u_{23} & u_{31} & u_{12}\\
 u_{01} & 0 & -u_{12} & u_{31} & 0 & 1 & 0 & 0\\
 u_{02} & u_{12} & 0 & -u_{23} & 0 & 0 & 1 & 0\\
 u_{03} & -u_{31} & u_{23} & 0 & 0 & 0 & 0 & 1
 \end{psm}\;,
 \end{align}
where:
 \begin{itemize}
 \item[$(i)$] the projectivization of the first column of \eqref{E:expgm} describes some $\ell \in U$ (in its right-coordinates);
 \item[$(ii)$] the first four columns of \eqref{E:expgm} describes the Lagrangian tangent space $L_\ell \cV$.  Using column reduction, we obtain:
 \begin{align*}
 \begin{psm}
 1 & 0 & 0 & 0 \\
 0 & 1 & 0 & 0 \\
 0 & 0 & 1 & 0 \\
 0 & 0 & 0 & 1 \\
 0 & -u_{01} & -u_{02} & -u_{03} \\
 u_{01} & 0 & -u_{12} & u_{31} \\
 u_{02} & u_{12} & 0 & -u_{23} \\
 u_{03} & -u_{31} & u_{23} & 0 
 \end{psm},
 \end{align*}
 so $\{ u_{ij} \}$ are standard affine coordinates on $\LG_+(\bbS)$.  (The bottom $4\times 4$ block is skew.)
 \end{itemize}
 We further observe for later use that the 2nd-4th columns of \eqref{E:expgm} can also be obtained by applying the (negatives of the) following derivations to the first column:
 \begin{align} \label{E:Lvert}
 \partial_{u_{23}} - u_{12} \partial_{u_{02}} + u_{31} \partial_{u_{03}}, \quad
 \partial_{u_{31}} - u_{23} \partial_{u_{03}} + u_{12} \partial_{u_{01}}, \quad
 \partial_{u_{12}} - u_{31} \partial_{u_{01}} + u_{23} \partial_{u_{02}}.
 \end{align}
\begin{rem}
From Proposition \ref{P:VLV}, $L(\cV)=\LG_+(\bbS)$ is a connected component of $\LG(\bbS)$, so we cannot hope to produce a 2nd order super-PDE as in the mixed contact case.  Naturally, this motivates looking to higher jets, e.g., for a 3rd order super-PDE, and we will do so below.
\end{rem}
\begin{rem} 
We conclude with a brief detour on how triality for $\fso(8)$ is related to our approach.  Let $\bar{\bbV} = \bbC^8 = \bar{E} \op \bar{F}$ be an isotropic decomposition of the standard representation of $ \fso(8)$ and define a Clifford action on $\bar\bbS = \Lambda^\bullet \bar E^*$ similar to \eqref{E:Clifford}, with positive and negative chirality $8$-dimensional irreducible representations $\bar\bbS_+ = \Lambda^{\mbox{\tiny even}} \bar E^*$ and $\bar\bbS_- = \Lambda^{\mbox{\tiny odd}} \bar E^*$. Then:
 \begin{enumerate}
 \item[(a)] Fix an isotropic spinor $\1 \in \bar\bbS_+$.  Its Clifford annihilator in $\bar{\bbV}$ is $\bar{E}$;
 \item[(b)] Fix  a non-isotropic spinor $\phi \in \bar\bbS_-$, e.g., $\phi = e^1 + e^{234}$;
 \item[(c)] Take the Clifford product of $\bar{E}$ and $\phi$.  The image is the distinguished 4-dimensional Lagrangian subspace $\bar\bbS^{-1}_+=\langle 1, e^{23}, e^{24}, e^{34}\rangle$ in $\bar\bbS_+$.
 \end{enumerate}
Our picture can be restored with appropriate identifications $\ff=\stab_{\fso(8)}(\phi)$, $\bbV=\phi^\perp\cap\bar\bbS_-$, $\bbS=\bar\bbS_+\cong\bar \bbV$ (as a representation for $\ff$). The Cayley $4$-form $\sfQ$ on $\bbS$ is obtained by squaring $\phi$ via the isomorphism $\odot^2\bar\bbS_-\cong\Lambda^0\bar \bbV^*\oplus\Lambda^4_-\bar \bbV^*$, with the explicit formula akin to those used in supergravity theories (see also \S\ref{eq:spinorial-description}). The main advantage of the approach of this section is the fact that the reduction of structure group is encoded in the Cayley $4$-form, which lives naturally on $\mathbb S$ and does not require the introduction of auxiliary spaces of spinors for $\fso(8)$.
 \end{rem}
\subsection{$F(4)$-supergeometries of odd-contact type}
 
 Consider a contact supermanifold $(M^{1|8},\cC)$ with rank $(0|8)$ contact superdistribution $\cC$.  Suppose that $\cC = \ker(\sigma)$ for some local defining (even) contact 1-form $\sigma$.  We say that the collection $\{ \omega^i, \theta_i \}_{i=0}^3$ of odd 1-forms  on $M$ is a {\sl local conformal coframe} of $\cC$ if $d\sigma = \lambda \sum_i\omega^i \wedge \theta_i$ on $\cC$ for some even invertible superfunction $\lambda$.  (As usual, $\wedge$ is meant to be skew in the supersense, i.e., just symmetric in this case.)  Given this, we define the following even supersymmetric 4-tensor on $\cC$:
 \begin{align} \label{E:Qct}
 \sfQ &= \omega^0 \omega^1 \theta_0 \theta_1 + \omega^0 \omega^2 \theta_0 \theta_2 + \omega^0 \omega^3 \theta_0 \theta_3 - 2 \omega^0 \theta_1 \theta_2 \theta_3 \nonumber\\
 & \qquad - \omega^1 \omega^2 \theta_1 \theta_2 - \omega^1 \omega^3 \theta_1 \theta_3 - \omega^2 \omega^3 \theta_2 \theta_3 + 2 \omega^1 \omega^2 \omega^3 \theta_0 \quad\in \Gamma(\odot^4 \cC^*),
 \end{align}
 which is modelled on \eqref{E:Q}.  (The restriction of the local conformal coframing to $\cC$ has been omitted on the r.h.s. for simplicity.  We also remark that $\cC$ is generated by odd supervector fields, so $\sfQ$ is skew in the classical sense on these generators.)
 
 \begin{definition}
\hfill
\begin{itemize}
	\item[$(i)$]  An {\sl odd-contact $F(4)$-supergeometry} $(M^{1|8},\cC,[\sfQ])$ is a contact supermanifold $(M^{1|8},\cC)$ whose contact superdistribution $\cC$ of rank $(0|8)$ is  equipped with a conformal class $[\sfQ]$ of a quartic tensor $\sfQ \in \Gamma(\odot^4 \cC^*)$, which is locally of the form \eqref{E:Qct} for some local conformal coframe of $\cC$.
	\item[$(ii)$] The symmetry superalgebra of $(M,\cC, [\sfQ])$ consists of the contact supervector fields preserving $[\sfQ]$: 
$
\finf(M,\cC,[\sfQ])= \{ \bX \in \fX(M) : \cL_\bX \cC \subset \cC, \, \exists \mu \mbox{ with } \cL_\bX \sfQ = \mu \sfQ \mbox{ on } \cC \} 
$. 
\end{itemize}

 \end{definition}
 
 \begin{theorem} \label{T:symbound}
The symmetry superalgebra of any odd-contact $F(4)$-supergeometry $(M^{1|8},\cC,[\sfQ])$ has $\dim \finf(M,\cC,[\sfQ]) \leq \dim F(4) = (24|16)$.
 \end{theorem}
 \begin{proof} Let $\fg = F(4)$, equipped with the odd-contact grading.  Since $[\sfQ]$ reduces the structure group from $\CO(8)$ to $G_0 = \CSpin(7)$, any odd-contact $F(4)$-supergeometry is a filtered $G_0$-structure with symbol $\fm$.  From Theorem \ref{sec:Spencer}, we have $H^{d,1}(\fm,\fg) = 0$ for all $d > 0$, which is equivalent to the Tanaka--Weisfeiler prolongation satisfying $\prn(\fm,\fg_0)\cong\fg$.  The claim then follows from \cite[Thm.1.1]{KST2022}.
 \end{proof}
 
 Locally, $(M^{1|8},\cC)$ is identified with $J^1(\bbC^{0|4},\bbC^{1|0})$, for which we have standard coordinates $(x^i,u,u_i)$, with $x^i,u_i$ odd for $0 \leq i \leq 3$ and $u$ even, and $\cC = \ker(\sigma)$ where $\sigma = du - \sum_i(dx^i) u_i$.  Any contact supervector field is identified with a generating superfunction via \eqref{E:Sf}.
 
 \begin{definition}
 Define $\sfQ$ as in \eqref{E:Qct} with $(\omega^i,\theta_i) = (dx^i,du_i)$.  We refer to $(M,\cC,[\sfQ])$ so defined as the {\sl flat} odd-contact $F(4)$-supergeometry.
 \end{definition}
 
 \begin{theorem} \label{T:flatF4ct}
The {\sl flat} odd-contact $F(4)$-supergeometry has symmetry superalgebra $F(4)$.
 \end{theorem}
 
 \begin{proof} In Table \ref{F:FlatQctSym}, we give a $(24|16)$-dimensional space of generating superfunctions that we can directly verify being symmetries of $(M,\cC,[\sfQ])$.  (This is easy for $\fg_{-2}$ and $\fg_{-1}$, direct but tedious for the generator of $\fg_2$, and then all other symmetries are generated by taking the Lagrange bracket \eqref{E:LB}.)  By Theorem \ref{T:symbound}, this is a basis for the symmetry superalgebra.  The proof that it is isomorphic to $F(4)$ follows exactly as in \cite[Thm.5.4]{KST2021}, using the existence of a symmetry $\sfZ = 2u - x^i u_i$ that acts like the grading element. 
 \end{proof} 
 
\begin{table}[h]
 \[
 \begin{array}{|c|c|} \hline
 \fg_2 & \begin{array}{c} u^2 + u (u_0 x^0 + u_1 x^1 + u_2 x^2 + u_3 x^3) \\
 - \tfrac{2}{3} u_0 (u_1 x^0 x^1 + u_2 x^0 x^2 + u_3 x^0 x^3) + \tfrac{1}{3} u_0 x^1 x^2 x^3 \\
 -\tfrac{1}{3}(u_1 u_2 u_3 x^0 + u_1 u_2 x^1 x^2 + u_1 u_3 x^1 x^3 + u_2 u_3 x^2 x^3)
 \end{array}
\\ \hline
 \fg_1 & \begin{array}{c} 
 -(u + \tfrac{1}{3} u_1 x^1 + \tfrac{1}{3} u_2 x^2 + \tfrac{1}{3} u_3 x^3 ) x^0 + \tfrac{1}{3} x^1 x^2 x^3,\\
 -(u + \tfrac{1}{3} u_0 x^0 + \tfrac{2}{3} u_2 x^2 + \tfrac{2}{3} u_3 x^3) x^1 - \tfrac{1}{3} u_2 u_3 x^0,\\
 -(u + \tfrac{1}{3} u_0 x^0 + \tfrac{2}{3} u_1 x^1 + \tfrac{2}{3} u_3 x^3) x^2 - \tfrac{1}{3} u_3 u_1 x^0,\\
 -(u + \tfrac{1}{3} u_0 x^0 + \tfrac{2}{3} u_1 x^1 + \tfrac{2}{3} u_2 x^2 ) x^3 - \tfrac{1}{3} u_1 u_2 x^0,\\
 -(u + \tfrac{2}{3} u_1 x^1 + \tfrac{2}{3} u_2 x^2 + \tfrac{2}{3} u_3 x^3 ) u_0 + \tfrac{1}{3} u_1 u_2 u_3,\\
 -(u + \tfrac{2}{3} u_0 x^0 + \tfrac{1}{3} u_2 x^2 + \tfrac{1}{3} u_3 x^3 ) u_1 - \tfrac{1}{3} x^2 x^3 u_0,\\
 -(u + \tfrac{2}{3} u_0 x^0 + \tfrac{1}{3} u_1 x^1 + \tfrac{1}{3} u_3 x^3 ) u_2- \tfrac{1}{3} x^3 x^1 u_0,\\
 -(u + \tfrac{2}{3} u_0 x^0 + \tfrac{1}{3} u_1 x^1 + \tfrac{1}{3} u_2 x^2 ) u_3 - \tfrac{1}{3} x^1 x^2 u_0
 \end{array} \\ \hline
 \fg_0 & \begin{array}{c} 
 x^0 x^1,\,\, x^0 x^2,\,\, x^0 x^3\\
 u_1 x^0 + x^2 x^3, \,\, u_2 x^0 + x^3 x^1, \,\, u_3 x^0 + x^1 x^2, \\
 u_0 x^1 - u_2 u_3, \,\, u_0 x^2 - u_3 u_1, \,\, u_0 x^3 - u_1 u_2, \\
 u_0 x^0 - u, \,\, u_1 x^1 + u,\,\, u_2 x^2 + u,\,\, u_3 x^3 + u,\\
 u_1 x^2, \,\, u_2 x^3, \,\, u_3 x^1, \\
 u_1 x^3, \,\, u_3 x^2, \,\, u_2 x^1, \\
 u_0 u_1, \,\, u_0 u_2, \,\, u_0 u_3
 \end{array}\\ \hline
 \fg_{-1} & x^0,\,\, x^1,\,\, x^2,\,\, x^3, \,\, u_0, \,\, u_1,\,\, u_2,\,\, u_3 \\ \hline
 \fg_{-2} & 1 \\ \hline
 \end{array}
 \]
 \caption{Symmetries of $(M,\cC,[\sfQ])$ when $(\omega^i,\theta^i) = (dx^i,du_i)$ in \eqref{E:Qct}}
 \label{F:FlatQctSym}
 \end{table}

We will geometrically reinterpret this structure in the sections to follow.

 \begin{rem} \label{R:omega2}
 In terms of Table \ref{F:FlatQctSym}, we can make Proposition \ref{prop:explicit-spinorial} explicit: $\1 = 1 \in \fg_{-2}$, while the stated elements of $\fg_2$ and $\fg_1$ in Table \ref{F:FlatQctSym} are respectively $\1^\dagger$ and $(x^0)^\dagger, (x^1)^\dagger, ..., (u_3)^\dagger$.  In the basis $\{ x^0, x^1, ..., u_3 \}$ of $\fg_{-1}$, we have:
 \begin{footnotesize}
 \begin{align*}
 \left(\omega^{(2)}(x^i,x^j)\right)_{0 \leq i,j \leq 3} &= \begin{pmatrix} 
 0 & 2 x^0 x^1 & 2 x^0 x^2 & 2 x^0 x^3 \\
 -2 x^0 x^1 & 0 & u_3 x^0+x^1 x^2 & -u_2 x^0-x^3 x^1 \\
 -2 x^0 x^2 & -u_3 x^0-x^1 x^2 & 0 & u_1 x^0+x^2 x^3 \\
 -2 x^0 x^3 & u_2 x^0+x^3 x^1 & -u_1 x^0-x^2 x^3 & 0 
 \end{pmatrix}\;,\\
 \left(\omega^{(2)}(u_i,u_j)\right)_{0 \leq i,j \leq 3} &= \begin{pmatrix} 
 0 & 2 u_0 u_1 & 2 u_0 u_2 & 2 u_0 u_3 \\
 -2 u_0 u_1 & 0 & x^3 u_0+u_1 u_2 & -x^2 u_0-u_3 u_1 \\
 -2 u_0 u_2 & -x^3 u_0-u_1 u_2 & 0 & x^1 u_0+u_2 u_3 \\
 -2 u_0 u_3 & x^2 u_0+u_3 u_1 & -x^1 u_0-u_2 u_3 & 0 
 \end{pmatrix}\;,\\
 \begin{split}
 \left(\omega^{(2)}(x^i,u_j)\right)_{0 \leq i,j \leq 3} &= \begin{pmatrix} 
 0 & -u_1 x^0-x^2 x^3 & -u_2 x^0+x^1 x^3 & -u_3 x^0-x^1 x^2\\
 -u_0 x^1+u_2 u_3 & 0 & -2 u_2 x^1 & -2 u_3 x^1\\
 -u_0 x^2-u_1 u_3 & -2 u_1 x^2 & 0 & -2 u_3 x^2\\
 -u_0 x^3+u_1 u_2 & -2 u_1 x^3 & -2 u_2 x^3 & 0\\ 
 \end{pmatrix}\\
 &\qquad + \tfrac{1}{2} \diag\left(-3u_0 x^0-u_1 x^1-u_2 x^2-u_3 x^3, -u_0 x^0-3u_1 x^1+u_2 x^2+u_3 x^3, \right. \\
 &\qquad\qquad\qquad\quad \left. -u_0 x^0+u_1 x^1-3u_2 x^2+u_3 x^3, -u_0 x^0+u_1 x^1+u_2 x^2-3u_3 x^3\right)\;.
 \end{split}
 \end{align*}
 \end{footnotesize}
 \end{rem}
\begin{rem} \label{R:Cayley}
The conformal class of the Cayley $4$-form $\sfQ$ can be recovered as the supersymmetric counterpart of Freudenthal's quartic invariant \cite{Fre1954,Fre1964}, or its realization \cite{LM2002,Hel} in terms of a contact grading $\fg = \fg_{-2} \op ... \op \fg_2$ as the symmetric 4-tensor $\sfQ(x)\1=(\ad_x)^4 \1^\dagger$, for $x \in \fg_{-1}$ and some fixed non-zero $\1 \in \fg_{-2}$ and $\1^\dagger \in \fg_2$.  (If $\fg = G_2$, then $\sfQ$ is the cubic discriminant, while if $\fg = D_4$, then $\sfQ$ is the Cayley hyperdeterminant.)  In our $F(4)$ odd-contact setting, given that $\fg_{-1}$ is purely odd, we have for all $s_1,s_2,s_3,s_4\in \fg_{-1}\cong\mathbb S$,
\begin{align}
\label{eq:Helenius}
\tfrac16\sfQ(s_1,s_2,s_3,s_4)\1&=\tfrac1{4!} \sum_{\rho} (-1)^\rho[[[[\1^\dagger,s_{\rho(1)}],s_{\rho(2)}],s_{\rho(3)}],s_{\rho(4)}]\;,
\end{align}
where $(-1)^\rho$ is the sign of a permutation $\rho$ of $\{1,\dots,4\}$. In fact, the R.H.S. of \eqref{eq:Helenius} defines a skew $4$-form on $\mathbb S$, which is proportional to $\sfQ$ due to $\fspin(7)$-invariance. Using Proposition \ref{prop:explicit-spinorial},
we then see that this R.H.S. is equal to $\tfrac1{4!} \sum_{\rho} (-1)^\rho\tfrac{1}{3}\eta\big(\sigma(\omega^{(2)}(s_{\rho(1)},s_{\rho(2)}))s_{\rho(3)},s_{\rho(4)}\big)\1$, and settle the constant of proportionality using the explicit expressions of the Lie brackets in terms of generating superfunctions in Remark \ref{R:omega2}.
\end{rem}

 \subsection{A distinguished superdistribution on the incidence Lagrange--Grassmann bundle}
 \label{S:DLG}
 Given an odd-contact $F(4)$-supergeometry $(M^{1|8},\cC,[\sfQ])$, the following geometric objects are also distinguished:
 \begin{itemize}
 \item[(i)] a {\it null quadric} $\cV = \{ \eta = 0 \} \subset \bbP(\cC)$.  (Locally, $[\eta] := [d\sigma|_\cC]$, where $\cC = \ker(\sigma)$.)
 \item[(ii)] the {\sl incidence Lagrange--Grassmann bundle} $\pi:\widetilde{M}^o\to M$ determined by the Lagrangian tangent spaces along $\cV$ distinguished  by (5) of Proposition \ref{P:VLV}, namely
\begin{equation*}
\widetilde{M}^o=\{(\ell,L):\ell\in\cV,L = L_\ell\cV\}\;,
\end{equation*}
where $\cV$ is as in (i).
\end{itemize}
A number of observations are in order:
\begin{itemize}
	\item We emphasize that these notions are defined ``pointwise'' on $M$, which is a supermanifold, so the functor of points approach is required to interpret the definitions rigorously.  Concretely, we shall consider a local frame of odd sections of $\cC$ that is dual to a conformal coframing, and keep the dependence on the ``points'' of $M$ (i.e, the components w.r.t. the frame, depending on even and odd coordinates) explicit.
\item The notion of a purely odd projective superspace presents subtleties (see the survey papers \cite{ManSur, BP}): in our case, there are no free $\bbA$-modules in $\bbS\otimes \bbA$ of rank $(1|0)$ for any given finite-dimensional supercommutative superalgebra $\bbA = \bbA_{\bar 0} \op \bbA_{\bar 1}$, so 
we shall deal with free $\bbA$-modules in $\bbS\otimes \bbA$ of rank $(0|1)$, i.e., with the Grassmannian of {\em odd} lines. This is the same as the classical projective space of $\bbS$ (regarded as a purely even vector space), allowing us to transfer the results in the classical setting of \S\ref{subsec:explicit-pres} to the framework of this section.
	\item By Proposition \ref{P:VLV}, the incidence Lagrange--Grassmann bundle naturally identifies with the subbundle $LG_+(\cC)$ of the Lagrange--Grassmann bundle $\pi:\widetilde M=LG(\cC)\to M$ consisting of all self-dual planes. However, their  symmetry superalgebras are quite different, since, in the former case, also the incidence relation has to be preserved.
\end{itemize}  
Since we have $\dim \LG(\bbS) = (6|0)$, then Proposition \ref{P:VLV}(1) indicates that $\dim \widetilde{M}^o = (7|8)$. Moreover $\widetilde{M}^o$ is equipped with the tautological  Cartan superdistribution $\widetilde\cC$ of rank $(6|4)$.  By the incidence relation, any $L$ in $\widetilde{M}^o$ is a Lagrangian subspace with $L = L_\ell \cV$ for a {\em unique} $\ell \in \cV$, which in turn pulls back to a $(6|1)$-subsuperdistribution $\widetilde\cC^o$ of $\widetilde\cC$. Explicitly:
\begin{equation*}
\widetilde\cC^o|_L=
(\pi_*)^{-1}(\ell)\subset (\pi_*)^{-1}(L)=\widetilde{\cC}|_L\;.
\end{equation*}
Since $L = L_\ell \cV$, then the inclusion $\widetilde{\cC}\subset(\widetilde\cC^o)^2$ follows: indeed $\widetilde{\cC}|_L$ is the pullback of $L$ and this is generated by taking the bracket of $\widetilde\cC^o$ with (a subdistribution of) the vertical distribution $V = \langle \partial_{u_{ij}} \rangle$ for $\pi:\widetilde{M}^o \to M$, see the obervation leading to \eqref{E:Lvert}.
  
 \begin{definition}  \label{D:Ddist} Let $(M^{1|8},\cC,[\sfQ])$ be an odd-contact $F(4)$-supergeometry. Then we define:
\begin{itemize}
\item[$(i)$] the superdistribution $\cD$ on $\widetilde{M}^o$ as the subdistribution of $\widetilde\cC^o$ satisfying the tensorial condition $[\cD,\widetilde\cC]\subset(\widetilde\cC^o)^2$.
\item[$(ii)$] the {\sl symmetry superalgebra} of the incidence Lagrange--Grassmann bundle 
$(\widetilde{M}^o,\widetilde\cC,\cD)$  as
$
 \finf(\widetilde{M}^o,\widetilde\cC,\cD) = \{ \bX \in \fX(\widetilde{M}^o) : \cL_\bX \widetilde\cC \subset \widetilde\cC, \cL_\bX \cD \subset \cD \}
$.
\end{itemize}
\end{definition}
 In supergeometry, the even and odd parts of a superdistribution generally do not have intrinsic meaning {\em as superdistributions}, as the even and odd sections are not sheaves of locally-free modules (not even for the even superfunctions, see \cite[\S 2.4]{KST2021}). Put differently, passing to the ``even part'' of a local frame is not a well-defined global operation, unless the structure group of the principal bundle of frames is purely even.  This can be algebraically seen for the $\cD$ of Definition  \ref{D:Ddist} as we have outlined in Remark \ref{R:5gr} below; we will also provide an independent geometric confirmation in the flat case, which is our main focus from now on.

We make this explicit for the flat odd-contact $F(4)$-supergeometry.  In jet-like coordinates $(x^i,u,u_i)$ on $(M^{1|8},\cC)$, we have $\eta= d\sigma|_\cC = (dx^i \wedge du_i)|_\cC$ and we consider 
the local conformal frame  $(D_{x^0},\ldots,D_{x^3},\partial_{u_0},\ldots, \partial_{u_3})$ of $\cC$, 
where $D_{x^i}= \partial_{x^i} + u_i \partial_u$.  Let $(x^i,u,u_i,u_{ij}=-u_{ji})$ be bundle-adapted coordinates on $(\widetilde{M}^o,\widetilde\cC)$, so that $L = \langle \widetilde{D}_{x^i} \rangle$, with
 $\widetilde{D}_{x^i} := D_{x^i} + \sum_{j}u_{ij} \partial_{u_j}$.  

From (1) and (4) of Proposition \ref{P:VLV}, the $(u_{ij})$'s can be regarded as locally parametrizing $\cV \subset \bbP(\cC)$ in some fibre over $M$: the $\ell \in \cV$ for which $L=L_\ell \cV$ is obtained by (the projectivization of) the first column of \eqref{E:expgm}, thought as components w.r.t. the chosen local conformal frame.  This yields
the odd supervector field
 \begin{align*}
 \begin{split}
 D_{x^0} - u_{23} D_{x^1} - u_{31} D_{x^2} - u_{12} D_{x^3}+ &(u_{01} u_{23} + u_{02} u_{31} + u_{03} u_{12}) \partial_{u_0}
 + u_{01} \partial_{u_1} + u_{02} \partial_{u_2} + u_{03} \partial_{u_3}\\
 &= \widetilde{D}_{x^0} - u_{23} \widetilde{D}_{x^1} - u_{31} \widetilde{D}_{x^2} - u_{12} \widetilde{D}_{x^3},
 \end{split}
 \end{align*}
 and, using the definition of $\widetilde\cC^o$, we then get
\begin{equation*}
\begin{aligned}
\widetilde\cC^o&= \langle \partial_{u_{ij}}, \widetilde{D}_{x^0} - u_{23} \widetilde{D}_{x^1} - u_{31} \widetilde{D}_{x^2} - u_{12} \widetilde{D}_{x^3} \rangle\;,\\
(\widetilde\cC^o)^2&=\widetilde\cC^o\oplus\langle \widetilde{D}_{x^1}, \widetilde{D}_{x^2}, \widetilde{D}_{x^3}, \partial_{u_1} + u_{23} \partial_{u_0},
 \partial_{u_2} + u_{31} \partial_{u_0},
 \partial_{u_3} + u_{12} \partial_{u_0}\rangle\;.
\end{aligned}
\end{equation*}
  In particular $\widetilde\cC=\langle \widetilde{D}_{x^i}, \partial_{u_{ij}}\rangle\subset(\widetilde\cC^o)^2$. Now the superdistribution $\cD \subset\widetilde\cC^o$ from Definition \ref{D:Ddist} is defined by the condition $[\cD,\widetilde\cC]\subset(\widetilde\cC^o)^2$ and, using \eqref{E:Lvert}, we arrive at $\cD$ given in \eqref{E:5gr}. We compute its weak derived flag, which has growth $(3|1), (3|3), (0|3), (0|1), (1|0)$, and see that  $\cD^2 = \widetilde\cC$, $\cD^3=(\widetilde\cC^o)^2$, $\cD^4=\widetilde\cC^2=(\widetilde\cC^o)^3$, and the full tangent bundle $\mathcal T\widetilde{M}^o = \cD^5 = \widetilde\cC^3 = (\widetilde\cC^o)^4$.

 \begin{align} \label{E:5gr}
 \begin{array}{|c|c|c|} \hline
 & \mbox{even} & \mbox{odd} \\ \hline\hline
 \cD &  
 \begin{array}{l} \partial_{u_{23}} - u_{12} \partial_{u_{02}} + u_{31} \partial_{u_{03}}, \\
 \partial_{u_{31}} - u_{23} \partial_{u_{03}} + u_{12} \partial_{u_{01}}, \\
 \partial_{u_{12}} - u_{31} \partial_{u_{01}} + u_{23} \partial_{u_{02}}
 \end{array} & \widetilde{D}_{x^0} - u_{23} \widetilde{D}_{x^1} - u_{31} \widetilde{D}_{x^2} - u_{12} \widetilde{D}_{x^3}\\ \hline
 \cD^2 / \cD & \partial_{u_{01}}, \,\, \partial_{u_{02}}, \,\, \partial_{u_{03}} & \widetilde{D}_{x^1}, \,\, \widetilde{D}_{x^2}, \,\, \widetilde{D}_{x^3} \\ \hline
 \cD^3 / \cD^2 & & 
 \begin{array}{l} 
 \partial_{u_1} + u_{23} \partial_{u_0},\,\, 
 \partial_{u_2} + u_{31} \partial_{u_0},\,\,
 \partial_{u_3} + u_{12} \partial_{u_0}
 \end{array}  \\ \hline
 \cD^4 / \cD^3 & & \partial_{u_0} \\ \hline
 \cD^5 / \cD^4 & \partial_u & \\ \hline
 \end{array}
 \end{align}
 
Now:
 \begin{itemize}
 \item[$(i)$] The vertical superdistribution $\ker(\pi_*)= \langle \partial_{u_{ij}} \rangle$ of $\pi:\widetilde{M}^o\to M$ is distinguished under contact transformations, so $\cD_{\bar{0}} := \cD \cap \ker(\pi_*)$ is a distinguished superdistribution.
 \item[$(ii)$] The Cauchy characteristic space $\Ch(\cD^3) = \{ \bX \in \Gamma(\cD^3) : \cL_\bX \cD^3 \subset \cD^3 \}$ of $\cD^3$ is covariantly defined from $\cD$, so $\cD_{\bar{1}}:=\Ch(\cD^3)$ is also a distinguished superdistribution.
 \end{itemize}
 From \eqref{E:5gr}, we explicitly have 
\begin{equation*}
\begin{aligned}
\cD_{\bar 0}&=\langle\partial_{u_{23}} - u_{12} \partial_{u_{02}} + u_{31} \partial_{u_{03}},
 \partial_{u_{31}} - u_{23} \partial_{u_{03}} + u_{12} \partial_{u_{01}}, 
 \partial_{u_{12}} - u_{31} \partial_{u_{01}} + u_{23} \partial_{u_{02}}
\rangle,\\
\cD_{\bar{1}} &= \langle \widetilde{D}_{x^0} - u_{23} \widetilde{D}_{x^1} - u_{31} \widetilde{D}_{x^2} - u_{12} \widetilde{D}_{x^3} \rangle,
\end{aligned}
\end{equation*}
of rank $(3|0)$ and $(0|1)$, respectively. Evidently $\cD=\cD_{\bar 0}\oplus\cD_{\bar 1}$ and $\widetilde\cC^o=(\cD_{\bar0})^2\oplus\cD_{\bar1}$. 
\begin{theorem} \label{T:F4-5gr}
Let $(\widetilde{M}^o,\widetilde\cC,\cD)$ be the incidence Lagrange--Grassmann bundle
associated to the flat odd-contact $F(4)$-supergeometry $(M,\cC,[\sfQ])$.
Then $\finf(\widetilde{M}^o,\widetilde\cC,\cD) \cong F(4)$.
 \end{theorem}
 
 \begin{proof} We geometrically constructed $(\widetilde{M}^o,\widetilde\cC, \cD)$ from the flat $(M,\cC,[\sfQ])$, so all symmetries  of the latter (Table \ref{F:FlatQctSym}) are inherited by the former.  By Theorem \ref{T:flatF4ct}, $F(4)$ injects into $\finf(\widetilde{M}^o,\widetilde\cC,\cD)$. Conversely, let us consider a symmetry $\bX \in \finf(\widetilde{M}^o,\widetilde\cC,\cD)$. 
It is contact by definition, so it projects to a contact vector field $\bS$ on $(M,\cC)$, and
it satisfies
$\cL_\bX \cD_{\bar{0}} \subset \cD_{\bar{0}}$ and $\cL_\bX \cD_{\bar{1}} \subset \cD_{\bar{1}}$ thanks to our general considerations. We claim that $\bX$ preserves each fiber of $\pi:\widetilde{M}^o \to M$, which is $F(\cV_x) = \{ (\ell, L) : \ell \in \cV_x, L = L_\ell (\cV_x) \}$ over some superpoint $x$ of $M$.  In fact, the incidence structure is $\ell = \pi_*\cD_{\bar1}|_L\subset \pi_*\widetilde\cC|_L = L$, and $\cD_{\bar{1}}$ and $\widetilde\cC$ are preserved by $\bX$, and so $F(\cV_x)$ is preserved by $\bS$.  From (6) of Proposition \ref{P:VLV}, the conformal class of the Cayley $4$-form $[\sfQ]$ is determined from this data, so $\bS$ must preserve it and  $\finf(M,\cC,[\sfQ]) \cong \finf(\widetilde{M}^o,\widetilde\cC,\cD)$.
 \end{proof}
 
 \begin{rem} \label{R:5gr}
The structure considered in Theorem \ref{T:F4-5gr} is the flat model for the parabolic supergeometry $M^{\rm I}_{12} := F(4) / P^{\rm I}_{12}$, which fibres over $M^{\rm I}_1 := F(4) / P^{\rm I}_1$.  An odd-contact $F(4)$-supergeometry is the geometric structure associated with the latter, which gives rise to an instance of the former via the correspondence space type construction discussed in this section.  Abstractly, a $|5|$-grading with purely even structure algebra $\fg_0$ arises from the former
-- see Table \ref{F:P12-roots}.
\begin{table}[h]
 \[
 \begin{array}{|c|l|l|} \hline
 k & \multicolumn{1}{c|}{\Delta_{\bar 0}^+(\fg_k)} &  \multicolumn{1}{c|}{\Delta_{\bar 1}^+(\fg_k)} \\ \hline
 0 & 
 \begin{array}{l}
 \alpha_3, \,\,
 \alpha_4, \,\,
 \alpha_3+\alpha_4
 \end{array} & 
  \\ \hline
 1 & 
 \begin{array}{l}
 \alpha_2, \,\,
 \alpha_2+\alpha_3, \,\,
 \alpha_2+\alpha_3+\alpha_4
 \end{array} &
 \begin{array}{l}
 \alpha_1
 \end{array} \\ \hline
 2 &
 \begin{array}{l}
 2\alpha_2+\alpha_3,\\
 2\alpha_2+\alpha_3+\alpha_4,\\
 2\alpha_2+2\alpha_3+\alpha_4
 \end{array} &
 \begin{array}{l}
 \alpha_1+\alpha_2,\\ 
 \alpha_1+\alpha_2+\alpha_3,\\
 \alpha_1+\alpha_2+\alpha_3+\alpha_4,\\
 \end{array}
 \\ \hline
 3 & &
 \begin{array}{l}
 \alpha_1+2\alpha_2+\alpha_3,\ \\
 \alpha_1+2\alpha_2+\alpha_3+\alpha_4,\ \\
 \alpha_1+2\alpha_2+2\alpha_3+\alpha_4
 \end{array} \\ \hline
 4 & &
 \begin{array}{l}
 \alpha_1+3\alpha_2+2\alpha_3+\alpha_4
 \end{array} \\ \hline
 5 & 
 \begin{array}{l}
 2\alpha_1+3\alpha_2+2\alpha_3+\alpha_4
 \end{array}
 & \\ \hline 
 \end{array}
 \]
 \caption{Grading associated to parabolic $\fp^{\rm I}_{12}$}
 \label{F:P12-roots}
 \end{table}
  \end{rem}

  \subsection{A remarkable 3rd order super-PDE} 
  
  According to \S\ref{S:J1J2}, we can construct a further bundle $p:\widecheck{M} \to \widetilde{M}$ equipped with a canonical superdistribution $\widecheck\cC$, and we can now consider its pullback $\widecheck{M}{}^o \to \widetilde{M}^o$ over the incidence Lagrange--Grassmann bundle $\widetilde{M}^o$.  
Specifically, $(\widetilde{M}^o, \widetilde\cC)$ admits the vertical superdistribution $\Ch(\widetilde\cC^2)=\langle\partial_{u_{ij}} \rangle$, and the total space of $\widecheck{M}{}^o$ consists of all subspaces $E$ of $\widetilde\cC=\langle \widetilde{D}_{x^i}, \partial_{u_{ij}} \rangle$ that are isotropic w.r.t. the Levi form of $\widetilde\cC$ and complementary in $\widetilde\cC$ to $\Ch(\widetilde\cC^2)$. Then $\widecheck\cC$ is obtained tautologically, by pulling back these subspaces to $\widecheck{M}{}^o$. Being an open sub-supermanifold of $(\widecheck{M}{},\widecheck\cC)$, the pair $(\widecheck{M}{}^o,\widecheck\cC)$ is locally isomorphic to the 3rd jet-superspace $J^3(\bbC^{0|4},\bbC^{1|0})$ with standard coordinates $(x^i,u,u_i,u_{ij},u_{ijk})$, where $0 \leq i,j,k \leq 3$. The aforementioned subspaces of $\widetilde\cC$ are of the form $E = \langle \widecheck{D}_{x^i} \rangle$, where
 \begin{align}
\label{eq:Dchecks}
 \widecheck{D}_{x^i} := \widetilde{D}_{x^i} + \sum_{j < k} u_{ijk} \partial_{u_{jk}} = \partial_{x^i} + u_i \partial_u + u_{ij} \partial_{u_j} + \sum_{j < k} u_{ijk} \partial_{u_{jk}}.
 \end{align}
  
The key point here is that the incidence Lagrange--Grassmann bundle $(\widetilde{M}^o, \widetilde\cC)$ is enhanced with the subsuperdistribution $\cD=\cD_{\bar 0}\oplus\cD_{\bar 1} \subset \widetilde\cC$ with covariant even and odd components, as observed in \S\ref{S:DLG} and Theorem \ref{T:F4-5gr}, 
and this allows us to distinguish a sub-supermanifold $\Sigma \subset \widecheck{M}{}^o$, i.e., a 3rd order super-PDE!  Namely, we restrict to those isotropic subspaces $E$ of $\widetilde\cC$ complementary to $\Ch(\widetilde\cC^2)$ that, in addition, {\bf\em contain $\cD_{\bar{1}}$}.  Explicitly, from $E = \langle \widecheck{D}_{x^i} \rangle$, we observe that
 \begin{align} \label{E:Llift}
 \begin{split}
 &\widecheck{D}_{x^0} - u_{23} \widecheck{D}_{x^1} - u_{31} \widecheck{D}_{x^2} - u_{12} \widecheck{D}_{x^3} \\
 &=  \widetilde{D}_{x^0} - u_{23} \widetilde{D}_{x^1} - u_{31} \widetilde{D}_{x^2} - u_{12} \widetilde{D}_{x^3} + \sum_{j < k} (u_{0jk} - u_{23} u_{1jk} - u_{31} u_{2jk} - u_{12} u_{3jk}) \partial_{u_{jk}}\;,
 \end{split}
 \end{align}
so the requirement of containing $\cD_{\bar{1}} = \langle \widetilde{D}_{x^0} - u_{23} \widetilde{D}_{x^1} - u_{31} \widetilde{D}_{x^2} - u_{12} \widetilde{D}_{x^3} \rangle$ forces the above summation to vanish.  Using the skew-symmetry of the indices of $u_{ijk}$ and focusing on the coefficients of $\partial_{u_{jk}}$ where $0 < j < k$, we efficiently arrive at the remarkable super-PDE 
 \begin{align*}
 u_{0ab} = u_{ab} u_{123}, \quad 1 \leq a < b \leq 3\;,
 \end{align*}
and re-substitution shows that this is sufficient for vanishing of the summation above.  
 
 \begin{theorem} \label{T:3PDE} Consider the sub-supermanifold $\Sigma \subset \widecheck{M}{}^o$ of the 
3rd jet-superspace $\widecheck{M}{}^o$
constructed as above from 
the incidence Lagrange--Grassmann bundle $(\widetilde{M}^o,\widetilde\cC,\cD)$ associated to 
the flat odd-contact $F(4)$-supergeometry $(M,\cC,[\sfQ])$. In standard coordinates $(x^i,u,u_i,u_{ij},u_{ijk})$ of $\widecheck{M}{}^o\cong J^3(\bbC^{0|4},\bbC^{1|0})$ (here $0 \leq i < j < k \leq 3$, the independent variables $x^i$ are odd and the dependent variable $u$ even), the sub-supermanifold  $\Sigma$ is given by the 3rd order super-PDE 
 \begin{align} \label{E:3PDE}
 u_{0ab} = u_{ab} u_{123}, \quad 1 \leq a < b \leq 3,
 \end{align}
and its contact symmetry superalgebra is isomorphic to $F(4)$.
 \end{theorem}
 
 \begin{proof} The contact symmetry superalgebra of $\Sigma$ consists of all contact vector fields $\bX \in \fX(\widecheck{M}{}^o)$ that are tangent to $\Sigma \subset \widecheck{M}{}^o$. Since the super-PDE $\Sigma$ was geometrically constructed from the incidence Lagrange--Grassmann bundle $(\widetilde{M}^o,\widetilde\cC,\cD)$, all contact symmetries of the latter lift to contact symmetries of $\Sigma \subset \widecheck{M}{}^o$.  

Conversely, consider the contact supergeometry of the sub-supermanifold $\Sigma \subset \widecheck{M}{}^o$ given by \eqref{E:3PDE}, with induced coordinates $(x^i,u,u_i,u_{ij},t = u_{123})$.  The vertical bundle $V = \langle \partial_t \rangle$ is distinguished under contact transformations, and so is the superdistribution $\widecheck\cD=\widecheck\cC\cap \mathcal{T}\Sigma$ induced on $\Sigma$ from
$\widecheck\cC= \langle \widecheck{D}_{x^i}, \partial_{u_{ijk}} \rangle$. Explicitly
$$\widecheck\cD := \langle \widecheck{D}_{x^0}, \,  \widecheck{D}_{x^1}, \,  \widecheck{D}_{x^2}, \,  \widecheck{D}_{x^3}, \, \partial_t \rangle\;,$$ 
where the $ \widecheck{D}_{x^i}$'s are as in \eqref{eq:Dchecks}
with the relations \eqref{E:3PDE} imposed.  The Cauchy characteristic space of $\widecheck\cD$ is given by
 \begin{align*}
 \Ch(\widecheck\cD) = \{ \bX \in \Gamma(\widecheck\cD) : \cL_\bX \widecheck\cD \subset \widecheck\cD \} &=  \langle \widecheck{D}_{x^0} - u_{23} \widecheck{D}_{x^1} - u_{31} \widecheck{D}_{x^2} - u_{12} \widecheck{D}_{x^3} \rangle\\
 &= \langle \widetilde{D}_{x^0} - u_{23} \widetilde{D}_{x^1} - u_{31} \widetilde{D}_{x^2} - u_{12} \widetilde{D}_{x^3} \rangle,
 \end{align*}
 where we have used \eqref{E:Llift} in the last step.  Consequently, $\Ch(\widecheck\cD) \op V$ is Frobenius integrable and invariant under contact transformations.
 
We define the superdistribution $\cD_{\bar{1}}$ of rank $(0|1)$ on $\widetilde M^o$ as push-forward of the quotient $(\Ch(\widecheck\cD) \op V) / V$.  Next, the fibre of $\Sigma \to M$ is distinguished by contact transformations, so its quotient by $V$ is also distinguished and can be pushed-forward to the superdistribution $\langle \partial_{u_{ij}} \rangle$ on $\widetilde M^o$ of rank $(6|0)$. We define a superdistribution $\widetilde\cC^o=\langle \partial_{u_{ij}} \rangle\oplus\cD_{\bar 1}$ of rank $(6|1)$, and
$\cD_{\bar{0}}$ as the subsuperdistribution of $\langle \partial_{u_{ij}} \rangle$ such that 
$\cD := \cD_{\bar{0}} \op \cD_{\bar{1}}$ satisfies
$[\cD,\widetilde\cC]\subset(\widetilde\cC^o)^2$.  Explicitly, we see that $\cD$ has the same local form as in \eqref{E:5gr}. Any contact symmetry of $\Sigma \subset \widecheck{M}{}^o$ induces a symmetry of $\widecheck\cD$, $\Ch(\widecheck\cD)$ and $V$, hence a symmetry of the incidence Lagrange--Grassmann bundle $(\widetilde{M}^o, \widetilde\cC,\cD)$.

In summary, the contact symmetries of $\Sigma \subset \widecheck{M}{}^o$ are in bijective correspondence with the symmetries of $(\widetilde{M}^o, \widetilde\cC,\cD)$ and
Theorem \ref{T:F4-5gr} implies the desired result. 
 \end{proof}

\begin{rem} One can independently confirm that all the generating superfunctions 
of the symmetries of $(M,\cC,[\sfQ])$ (see Table \ref{F:FlatQctSym}) indeed prolong to contact symmetries of the super-PDE \eqref{E:3PDE}, and Theorem \ref{T:3PDE} implies that this list is complete.
\end{rem}

Finally, let us briefly remark on the solution superspace of our PDE system \eqref{E:3PDE}.
This is introduced via the functor of points, the main 
r$\hat{\text{o}}$le being played by even superfunctions which are parametrized by
elements in the auxiliary finite-dimensional supercommutative superalgebra $\bbA = \bbA_{\bar 0} \oplus \bbA_{\bar 1}$ (these are the super-points of such a superspace).
Concretely, we let $u = u(x^i)$ be an even superfunction of the odd coordinates $x^i$, so $u$ is a supersymmetric polynomial in the $x^i$'s with constant coefficients in $\mathbb A$.  Differentiating $u_{012} = u_{12} u_{123}$ w.r.t. $x^3$, we obtain $u_{3012} = u_{312} u_{123} + u_{12} u_{3123} = 0$ since $u_{123}$ is odd, so we obtain the compatibility condition $u_{0123} = 0$.  Thus, 
\begin{align*}
u = \lambda + \sum_{i < j} \lambda_{ij} x^i x^j + \sum_i \theta_i x^i + \sum_{i < j < k} \theta_{ijk} x^i x^j x^k,
\end{align*}
where the constants $\lambda,\lambda_{ij}$ are even, while $\theta_i,\theta_{ijk}$ are odd.  Substitution in \eqref{E:3PDE} then forces
$\theta_{012} = -(\lambda_{12}+\theta_{012}x^0) \epsilon$,
$\theta_{013} = -(\lambda_{13}+\theta_{013}x^0) \epsilon$,  
$\theta_{023} = -(\lambda_{23}+\theta_{023}x^0) \epsilon$, 
with $\epsilon := u_{123}=\theta_{123}$. Using \eqref{E:3PDE} again, we then see that $\theta_{012}\theta_{123}=u_{012}u_{123}=u_{12}u_{123}u_{123}=0$, and similarly $\theta_{013}\theta_{123}=0$, $\theta_{023}\theta_{123}=0$. In summary, we arrive at
\begin{align*}
\theta_{012} = -\lambda_{12} \epsilon, \quad
\theta_{013} = -\lambda_{13} \epsilon, \quad 
\theta_{023} = -\lambda_{23} \epsilon,
\end{align*}
and the solution superspace of \eqref{E:3PDE} is $(7|5)$-dimensional.

\section*{Acknowledgments}

We thank Boris Kruglikov for helpful discussions. The first author is supported by a tenure-track RTDB position, and and acknowledges the MIUR Excellence Department
Project MatMod@TOV awarded to the Department of Mathematics, University
of Rome Tor Vergata, CUP E83C23000330006. The research leading to these results has received funding from the Norwegian Financial Mechanism 2014-2021 (project registration number 2019/34/H/ST1/00636), the Troms\o{} Research Foundation (project ``Pure Mathematics in Norway''), the UiT Aurora project MASCOT, and this article/publication is based upon work from COST Action CaLISTA CA21109 supported by COST (European Cooperation in Science and Technology), \href{https://www.cost.eu}{https://www.cost.eu}.

\appendix
\section{Spencer cohomology for the odd contact grading: proof of Theorem \ref{sec:Spencer}}
\label{S:Spencer-odd-contact-grading}
We will extensively and tacitly use the spinorial presentation given by Proposition \ref{prop:explicit-spinorial} and split the proof into several cases, depending on the degree $d\geq 1$.    The relevant spaces of cocycles and coboundaries are $Z^{d,1}(\fm,\fg)$ and $B^{d,1}(\fm,\fg)$ respectively.

\subsection{Case $d=1$.}\par\noindent

\begin{proposition} We have $B^{1,1}(\fm,\fg) = Z^{1,1}(\fm,\fg)\cong \bbS$, and hence $H^{1,1}(\fm,\fg)=0$.
\end{proposition}
\begin{proof}
Note that $Z^{1,1}(\fm,\fg)$ includes 
at least one module isomorphic to $\mathbb S$, since $B^{1,1}(\fm,\fg)=\partial\fg_1$.  Let $\omega + \alpha \in Z^{1,1}(\fm,\fg)$, with components $\omega = \omega^{\ss} + \omega^Z : \fg_{-1} \to \fso(\bbV) \oplus \bbC Z$ and $\alpha : \fg_{-2} \to \fg_{-1}$, and modify it by a coboundary $\partial t^\dagger$ for $t \in \fg_{-1}$.  Then $(\omega + \alpha + \partial t^\dagger)(\1) = \alpha(\1) + [\1,t^\dagger] = \alpha(\1) - t$.  Setting $t = \alpha(\1)$ and relabelling, we may assume that $\alpha = 0$.  Thus, we have for all $s \in \fg_{-1}$,
 \begin{align*} 
 0 &= \partial \omega(\1,s) = [\1,\omega^Z(s)Z] = 2\omega^Z(s) \1 \quad\Rightarrow\quad \omega^Z(s)=0, 
\\
 0 &=\partial\omega(s,s)=2[\omega^{\ss}(s),s]-2\omega^{Z}(s)s = 2[\omega^{\ss}(s),s]\;, \label{eq:s-d-II}
\end{align*}
 Thus, $\omega^{\ss}=0$ by Proposition \ref{prop:omega-vanishes}, hence $\omega = 0$ and the results follow.

\end{proof}
\subsection{Case $d=2$.}\label{subsubsec:d=2}\par\noindent

\begin{proposition} We have $B^{2,1}(\fm,\fg) = Z^{2,1}(\fm,\fg)\cong \bbC$, and hence $H^{2,1}(\fm,\fg) = 0$.
\end{proposition}
\begin{proof}
Let $\omega+\alpha \in Z^{2,1}(\fm,\fg)$, where $\omega : \fg_{-1} \to \fg_1$ and $\alpha : \fg_{-2} \to \fg_0$.  For any $s \in \fg_{-1}$,
\begin{equation*}
0 =\partial(\omega+\alpha)(\1,s)=[\1,\omega(s)]+[\alpha(\1),s] 
\end{equation*}
so $\omega=\alpha(\1)$ as endomorphisms of $\mathbb S$, and
\begin{equation}
\label{eq:s-d-II-2}
\begin{aligned}
0&=\partial(\omega+\alpha)(s,s)=2[\omega(s),s]-\langle s,s\rangle
\alpha(\1)\\
&=\tfrac{2}{3}\omega^{(2)}(\omega(s),s)-\langle \omega(s),s\rangle Z
-\langle s,s\rangle
\omega
\;,
\end{aligned}
\end{equation}
using Proposition \ref{prop:explicit-spinorial}.  Picking $s$ non-isotropic, we see that $\omega\in\fg_0\cong\fso(\bbV)\oplus\mathbb C$.

It is easy to see that we may modify the cocycle by an appropriate coboundary so that $\omega\in\fso(\bbV)$. 
Writing then
$\omega=\tfrac12\omega^{\mu\nu}\Gamma_{\mu\nu}$, equation
\eqref{eq:s-d-II-2} becomes
\begin{equation*}
\begin{aligned}
0&=-\tfrac{2}{3}\omega^{(2)}(s,\omega(s))-\langle s,s\rangle\omega=-\tfrac{2}{3}\tfrac14\omega^{\mu\nu}
\omega^{(2)}(s,\Gamma_{\mu\nu}s)-\tfrac12(\overline s s)\omega^{\mu\nu}\Gamma_{\mu\nu}\\
&=-\tfrac{1}{3}\tfrac14\omega^{\mu\nu}
\Big(\overline s\Gamma_{\alpha\beta}\Gamma_{\mu\nu}s\Big)\Gamma^{\alpha\beta}
-\tfrac12(\overline s s)\omega^{\mu\nu}\Gamma_{\mu\nu}\\
&=-\tfrac{1}{3}\tfrac14\omega^{\mu\nu}
\Big(\overline s \Gamma_{\alpha\beta\mu\nu}s\Big)\Gamma^{\alpha\beta}
-\tfrac13(\overline s s)\omega^{\mu\nu}\Gamma_{\mu\nu}\;,
\end{aligned}
\end{equation*}
where in the last step we used that $\Gamma_{\alpha\beta}\Gamma_{\mu\nu}\equiv\Gamma_{\alpha\beta\mu\nu}+\delta_{\beta\mu}\delta_{\alpha\nu}-\delta_{\alpha\mu}\delta_{\beta\nu}$ modulo $\Lambda^2 \bbV\subset\Lambda^2\mathbb S$. Since
$\odot^2\mathbb S\cong \Lambda^0 \bbV\oplus\Lambda^3 \bbV\cong \Lambda^0 \bbV\oplus\Lambda^4 \bbV$, one readily gets $\omega=0$.
\end{proof}

\subsection{Case $d=3$.}\par\noindent

\begin{lemma}
The differential $\partial:C^{3,1}(\fm,\fg)\to C^{3,2}(\fm,\fg)$ is injective, and hence $H^{3,1}(\fm,\fg)=0$.
\end{lemma}
\begin{proof} Let $\omega + \alpha \in Z^{3,1}(\fm,\fg)$, where $\omega : \fg_{-1} \to \fg_2$ and $\alpha : \fg_{-2} \to \fg_1$, and  write $\omega = \tilde\omega\1^\dagger$ for some $\tilde\omega : \fg_{-1} \to \bbC$, and $\alpha(\1) = t^\dagger$ for some $t \in \mathbb S$.  For all $s \in \fg_{-1}$, we have
 \begin{align*}
 0 &= \partial(\omega + \alpha)(\1,s) = [\1,\omega(s)] + [s,\alpha(\1)] = \tilde\omega(s) Z + \tfrac{1}{3} \omega^{(2)}(t,s) - \tfrac{1}{2} \langle t, s \rangle Z\\
 0 &= \partial(\omega + \alpha)(s,s) = -2[s, \omega(s)] - \alpha([s,s]) = 2\tilde\omega(s) s^\dagger - \langle s, s \rangle t^\dagger
 \end{align*}
 The first equation implies $\tilde\omega = \tfrac{1}{2} \overline{t}$.  The second equation then implies $\langle t, s \rangle s^\dagger = \langle s, s \rangle t^\dagger$ for all $s \in \fg_{-1}$, so $t = 0$.  Thus, $\alpha = 0$ and $\omega = 0$ follow, proving our claim. 
\end{proof}
\subsection{Case $d=4$.}\par\noindent

\begin{lemma}
The differential $\partial:C^{4,1}(\fm,\fg)\to C^{4,2}(\fm,\fg)$ is injective, and hence $H^{4,1}(\fm,\fg)=0$. 
\end{lemma}
\begin{proof} Let $\alpha \in Z^{4,1}(\fm,\fg)$, where $\alpha : \fg_{-2} \to \fg_2$.  Thus, $0 = \partial\alpha(s,s) = -\langle s, s \rangle \alpha(\1)$ for all $s \in \fg_{-1}$, so $\alpha(\1) = 0$ and $\alpha = 0$.
\end{proof}
\subsection{Case $d\geq 5$.}\par\noindent
We have $C^{d,1}(\fm,\fg)=0$ for all $d\geq 5$
simply by degree reasons. 
\begin{corollary}
The Spencer cohomology groups $H^{d,1}(\fm,\fg)=0$ for all $d\geq 5$.
\end{corollary}

\end{document}